\documentclass[a4paper]{scrartcl}

\usepackage[utf8]{inputenc}
\usepackage[T1]{fontenc}

\usepackage{amsmath, amsthm, amssymb}
\usepackage{mathtools}
\usepackage{mathrsfs}
\usepackage{enumitem}
\newlist{hypothenum}{enumerate}{3}
\setlist[hypothenum,1]{label=(\roman*)}
\usepackage{url}
\usepackage[all,cmtip]{xy}
\usepackage{tikz-cd}
\usepackage{tikz}
\usetikzlibrary{calc}
\usetikzlibrary{patterns}
\usetikzlibrary{decorations.pathmorphing}
\usepackage{relsize}
\usepackage{longtable}
\usepackage{needspace}
\usepackage{caption}
\usepackage{subcaption}
\usepackage{multirow}
\usepackage{arydshln}

\numberwithin{equation}{section}

\theoremstyle{plain}
\newtheorem*{mainthm}{Main Theorem}
\newtheorem{theorem}{Theorem}[section]
\newtheorem{proposition}[theorem]{Proposition}

\newtheorem{lemma}[theorem]{Lemma}
\newtheorem{corollary}[theorem]{Corollary}

\newtheorem{assumption}[theorem]{Assumption}

\theoremstyle{definition}
\newtheorem{definition}[theorem]{Definition}
\newtheorem{example}[theorem]{Example}

\theoremstyle{remark}
\newtheorem{remark}[theorem]{Remark}

\newcommand{\eps}{\varepsilon}
\newcommand{\setsuch}[2]{\left\{ #1 \; \middle| \; #2 \right\}}
\newcommand{\restr}[2]{{\left. #1 \right|}_{#2}}
\newcommand{\subl}{{\mathsmaller{<}}}
\newcommand{\subg}{{\mathsmaller{>}}}
\newcommand{\sube}{{\mathsmaller{=}}}
\newcommand{\suble}{{\mathsmaller{\leq}}}
\newcommand{\subge}{{\mathsmaller{\geq}}}
\newcommand{\transl}{t}
\newcommand{\rotat}{r}
\newcommand{\ext}{\mathsf{\Lambda}}

\newcommand{\bigo}{\mathcal{O}}

\DeclareMathOperator{\Conv}{Conv}
\DeclareMathOperator{\Aff}{Aff}
\DeclareMathOperator{\Id}{Id}
\DeclareMathOperator{\Ad}{Ad}
\DeclareMathOperator{\Stab}{Stab}
\DeclareMathOperator{\GL}{GL}

\DeclareMathOperator{\SL}{SL}
\DeclareMathOperator{\PSL}{PSL}
\DeclareMathOperator{\SO}{SO}
\DeclareMathOperator{\PSO}{PSO}

\DeclareMathOperator{\Orth}{O}

\newcommand{\ie}{i.e.\ }
\newcommand{\eg}{e.g.\ }

\newcommand{\longlongrightarrow}{\xrightarrow{\hspace*{1cm}}}
\newcommand{\jordan}{\operatorname{Jd}}
\newcommand{\cartan}{\operatorname{Ct}}

\begin{document}

\title{Proper affine actions in non-swinging representations}
\author{Ilia Smilga}
       
\maketitle

\begin{abstract}
For a semisimple real Lie group $G$ with an irreducible representation $\rho$ on a finite-dimensional real vector space $V$, we give a sufficient criterion on $\rho$ for existence of a group of affine transformations of $V$ whose linear part is Zariski-dense in $\rho(G)$ and that is free, nonabelian and acts properly discontinuously on~$V$.
\end{abstract}



\section{Introduction}

\subsection{Background and motivation}
\label{sec:background}

The present paper is part of a larger effort to understand discrete groups $\Gamma$ of affine transformations (subgroups of the affine group $\GL_n(\mathbb{R}) \ltimes \mathbb{R}^n$) acting properly discontinuously on the affine space $\mathbb{R}^n$. The case where $\Gamma$ consists of isometries (in other words, $\Gamma \subset \Orth_n(\mathbb{R}) \ltimes \mathbb{R}^n$) is well-understood: a classical theorem by Bieberbach says that such a group always has an abelian subgroup of finite index.

We say that a group $G$ acts \emph{properly discontinuously} on a topological space~$X$ if for every compact $K \subset X$, the set $\setsuch{g \in G}{g K \cap K \neq \emptyset}$ is finite. We define a \emph{crystallographic} group to be a discrete group $\Gamma \subset \GL_n(\mathbb{R}) \ltimes \mathbb{R}^n$ acting properly discontinuously and such that the quotient space $\mathbb{R}^n / \Gamma$ is compact. In \cite{Aus64}, Auslander conjectured that any crystallographic group is virtually solvable, that is, contains a solvable subgroup of finite index. Later, Milnor \cite{Mil77} asked whether this statement is actually true for any affine group acting properly discontinuously. The answer turned out to be negative: Margulis \cite{Mar83, Mar87} gave a nonabelian free group of affine transformations with linear part Zariski-dense in $\SO(2, 1)$, acting properly discontinuously on $\mathbb{R}^3$. On the other hand, Fried and Goldman \cite{FG83} proved the Auslander conjecture in dimension 3 (the cases $n=1$ and $2$ are easy). Recently, Abels, Margulis and Soifer \cite{AMS12} proved it in dimension $n \leq 6$. See \cite{AbSur} for a survey of already known results.

Margulis's breakthrough was soon followed by the construction of other counterexamples to Milnor's conjecture. The first advance was made by Abels, Margulis and Soifer~\cite{AMS02}: they generalized Margulis's construction to subgroups of the affine group
\[\SO(2n+2,2n+1) \ltimes \mathbb{R}^{4n+3},\]
for all values of~$n$. The author further generalized this in his previous paper~\cite{Smi14}, by finding such subgroups in the affine group~$G \ltimes \mathfrak{g}$, where $G$~is any noncompact semisimple real Lie group, acting on its Lie algebra~$\mathfrak{g}$ by the adjoint representation. Recently Danciger, Gu{\'e}ritaud and Kassel~\cite{DGKCox} found examples of affine groups acting properly discontinuously that were neither virtually solvable nor virtually free.

Proliferation of these counterexamples leads to the following question. Consider a semisimple real Lie group~$G$; for every representation~$\rho$ of~$G$ on a finite-dimensional real vector space~$V$, we may consider the affine group~$G \ltimes V$. Which of those affine groups contain a nonabelian free subgroup with linear part Zariski-dense in~$G$ and acting properly discontinuously on~$V$?

In this paper, we give a fairly general sufficient condition on the representation~$\rho$ for existence of such subgroups. Before stating this condition, we need to introduce a few classical notations.

\subsection{Basic notations}
\label{sec:lie}

For the remainder of the paper, we fix a semisimple real Lie group~$G$; let $\mathfrak{g}$~be its Lie algebra. Let us introduce a few classical objects related to~$\mathfrak{g}$ and~$G$ (defined for instance in Knapp's book \cite{Kna96}, though our terminology and notation differ slightly from his).

We choose in $\mathfrak{g}$:
\begin{itemize}
\item a Cartan involution $\theta$. Then we have the corresponding Cartan decomposition $\mathfrak{g} = \mathfrak{k} \oplus \mathfrak{q}$, where we call $\mathfrak{k}$ the space of fixed points of $\theta$ and $\mathfrak{q}$ the space of fixed points of $-\theta$. We call $K$ the maximal compact subgroup with Lie algebra $\mathfrak{k}$.
\item a \emph{Cartan subspace} $\mathfrak{a}$ compatible with $\theta$ (that is, a maximal abelian subalgebra of $\mathfrak{g}$ among those contained in $\mathfrak{q}$). We set $A := \exp \mathfrak{a}$.
\item a system $\Sigma^+$ of positive restricted roots in $\mathfrak{a}^*$. Recall that a \emph{restricted root} is a nonzero element $\alpha \in \mathfrak{a}^*$ such that the restricted root space
\[\mathfrak{g}^\alpha := \setsuch{Y \in \mathfrak{g}}{\forall X \in \mathfrak{a},\; [X, Y] = \alpha(X)Y}\]
is nontrivial. They form a root system $\Sigma$; a system of positive roots $\Sigma^+$ is a subset of $\Sigma$ contained in a half-space and such that $\Sigma = \Sigma^+ \sqcup -\Sigma^+$.

We call~$\Pi$ be the set of simple restricted roots in~$\Sigma^+$. We call
\[\mathfrak{a}^{++} := \setsuch{X \in \mathfrak{a}}{\forall \alpha \in \Sigma^+,\; \alpha(X) > 0}\]
the (open) dominant Weyl chamber of $\mathfrak{a}$ corresponding to~$\Sigma^+$, and
\[\mathfrak{a}^{+} := \setsuch{X \in \mathfrak{a}}{\forall \alpha \in \Sigma^+,\; \alpha(X) \geq 0} = \overline{\mathfrak{a}^{++}}\]
the closed dominant Weyl chamber.
\end{itemize}
Then we denote
\begin{itemize}
\item $M$ the centralizer of $\mathfrak{a}$ in $K$, $\mathfrak{m}$ its Lie algebra.
\item $L$ the centralizer of $\mathfrak{a}$ in $G$, $\mathfrak{l}$ its Lie algebra. It is clear that $\mathfrak{l} = \mathfrak{m} \oplus \mathfrak{a}$, and well known (see \eg \cite{Kna96}, Proposition 7.82a) that $L = MA$.
\item $\mathfrak{n}^+$ (resp. $\mathfrak{n}^-$) the sum of the restricted root spaces $\mathfrak{g}^\alpha$ for $\alpha$ in~$\Sigma^+$ (resp. in $-\Sigma^+$), and $N^+ := \exp(\mathfrak{n}^+)$ and $N^- := \exp(\mathfrak{n}^-)$ the corresponding Lie groups.
\item $\mathfrak{p}^+ := \mathfrak{l} \oplus \mathfrak{n}^+$ and $\mathfrak{p}^- := \mathfrak{l} \oplus \mathfrak{n}^-$ the corresponding minimal parabolic subalgebras, $P^+ := LN^+$ and $P^- := LN^-$ the corresponding minimal parabolic subgroups.
\item $W := N_G(A)/Z_G(A)$ the restricted Weyl group.
\item $w_0$ the \emph{longest element} of the Weyl group, that is, the unique element such that $w_0(\Sigma^+) = \Sigma^-$.
\end{itemize}

See Examples 2.3 and 2.4 in the author's previous paper~\cite{Smi14} for working through these definitions in the cases $G = \PSL_n(\mathbb{R})$ and~$G = \PSO^+(n, 1)$.

Finally, if $\rho$~is a representation of~$G$ on a finite-dimensional real vector space~$V$, we call:
\begin{itemize}
\item the \emph{restricted weight space} in~$V$ corresponding to a form~$\lambda \in \mathfrak{a}^*$ the space
\[V^\lambda := \setsuch{v \in V}{\forall X \in \mathfrak{a},\; \rho(X) \cdot v = \lambda(X)v};\]
\item a \emph{restricted weight} of the representation~$\rho$ any form~$\lambda \in \mathfrak{a}^*$ such that the corresponding weight space is nonzero.
\end{itemize}

\begin{remark}
The reader who is unfamiliar with the theory of noncompact semisimple real Lie groups may focus on the case where $G$~is \emph{split}, \ie its Cartan subspace~$\mathfrak{a}$ is actually a Cartan subalgebra (just a maximal abelian subalgebra, without any additional hypotheses). In that case the restricted roots are just roots, the restricted weights are just weights, and the restricted Weyl group is just the usual Weyl group. Also the algebra~$\mathfrak{m}$ vanishes and $M$~is a discrete group.

However, the case where $G$~is split does not actually require the full strength of this paper, in particular because quasi-translations (see Section~\ref{sec:quasi-translations}) then reduce to ordinary translations.
\end{remark}

\subsection{Statement of main result}

Let $\rho$~be an irreducible representation of~$G$ on a finite-dimensional real vector space~$V$. Without loss of generality, we may assume that $G$~is connected and acts faithfully. We may then identify the abstract group~$G$ with the linear group~$\rho(G) \subset \GL(V)$. Let~$V_{\Aff}$~be the affine space corresponding to~$V$. The group of affine transformations of~$V_{\Aff}$ whose linear part lies in~$G$ may then be written~$G \ltimes_\rho V$ or simply~$G \ltimes V$ (where $V$~stands for the group of translations). Here is the main result of this paper.
\begin{mainthm}
Let $G$ be a semisimple real Lie group, and let $\rho$ be an irreducible representation of~$G$ on a finite-dimensional real vector space~$V$ that satisfies the following conditions:
\begin{hypothenum}
\item \label{itm:main_condition} there exists a vector $v \in V$ such that:
\begin{enumerate}[label=(\alph*)]
\item \label{itm:fixed_by_l} $\forall l \in L,\; l(v) = v$, and
\item \label{itm:not_fixed_by_w0} $\tilde{w}_0(v) \neq v$, where $\tilde{w}_0$ is any representative in~$G$ of~$w_0 \in N_G(A)/Z_G(A)$;
\end{enumerate}
\item \label{itm:technical_condition} there exists an element~$X_0 \in \mathfrak{a}$ such that~$-w_0(X_0) = X_0$ and for every nonzero restricted weight~$\lambda$ of~$\rho$, we have $\lambda(X_0) \neq 0$.
\end{hypothenum}
Then there exists a subgroup~$\Gamma$ in the affine group~$G \ltimes_\rho V$ whose linear part is Zariski-dense in~$G$ and that is free, nonabelian and acts properly discontinuously on~$V_{\Aff}$.
\end{mainthm}
\begin{remark}
\label{w0_action_makes_sense}
Note that the choice of the representative~$\tilde{w}_0$ in~\ref{itm:main_condition}\ref{itm:not_fixed_by_w0} does not matter, precisely because by~\ref{itm:main_condition}\ref{itm:fixed_by_l} the vector~$v$ is fixed by~$L = Z_G(A)$.
\end{remark}

We call representations satisfying condition~\ref{itm:technical_condition} ``\emph{non-swinging}'' representations (see Section~\ref{sec:symmetry_pos_neg} to understand why). This is only a technical assumption: if we remove it, the theorem remains true. This more general result is proved in the author's forthcoming paper~\cite{Smi16b}.

Note that the previously-known examples do fall under the scope of this theorem:
\begin{example}~
\begin{enumerate}
\item For $G = \SO^+(2n+2, 2n+1)$, the standard representation (acting on~$V = \mathbb{R}^{4n+3}$) satisfies these conditions (see Remark~\ref{remark_no_swinging} and Examples~\ref{transl_space_example}.1.b and~\ref{inverse_ok_example}.1 for details). So Theorem~A from~\cite{AMS02} is a particular case of this theorem.
\item If the real semisimple Lie group~$G$ is noncompact, the adjoint representation satisfies these conditions (see Remark~\ref{remark_no_swinging} and Examples~\ref{transl_space_example}.3 and~\ref{inverse_ok_example}.2 for details). So the main theorem of~\cite{Smi14} is a particular case of this theorem.
\end{enumerate}
\end{example}
\begin{remark}
When $G$~is compact, no representation can satisfy these conditions: indeed in that case $L$ is the whole group~$G$ and condition~\ref{itm:main_condition}\ref{itm:fixed_by_l} fails. So for us, only noncompact groups are interesting. This is not surprising: indeed, any compact group acting on a vector space preserves a positive-definite quadratic form, and so falls under the scope of Bieberbach's theorem.
\end{remark}

\subsection{Strategy of the proof}
The proof has a lot in common with the author's previous paper~\cite{Smi14}. The main idea (which comes back to Margulis's seminal paper~\cite{Mar83}) is to introduce, for some affine maps~$g$, an invariant that measures the translation part of~$g$ along a particular affine subspace of~$V$. The key part of the argument (just as in~\cite{Mar83} and in~\cite{Smi14}) is then to show that, under some conditions, the invariant of the product of two maps is roughly equal to the sum of their invariants (Proposition~\ref{invariant_additivity}). Here are the two main difficulties that were not present in~\cite{Smi14}.
\begin{itemize}
\item The first one is that~\cite{Smi14} crucially relies on the following fact: if two maps are $\mathbb{R}$-regular (\ie the dimension of their centralizer is the lowest possible), in general position with respect to each other and strongly contracting (when acting on~$\mathfrak{g}$), their product is still $\mathbb{R}$-regular. The natural generalization of the notion of an $\mathbb{R}$-regular map that is adapted to an arbitrary representation is that of a ``generic'' map, \ie a map that has as few eigenvalues of modulus~$1$ (counted with multiplicity) as possible. Unfortunately, the corresponding statement is then no longer true in an arbitrary representation. If the representation is ``too large'', \ie if it contains restricted weights that are not multiples of restricted roots (see Example~\ref{equivalence_class_examples}.2), there are several different ``types'' of generic maps, depending on the region where their Jordan projection (see Definition~\ref{Jd_Ct_definition}) falls. 

In order to ensure that the product of two generic maps $g$ and $h$ (that are in general position and strongly contracting) is still generic, we need to control the Jordan projection $\jordan(gh)$ of the product based on the Jordan projections $\jordan(g), \jordan(h)$ of the factors. To do this, we use ideas developed by Benoist in~\cite{Ben96, Ben97}: when $g$ and $h$ are in general position and sufficiently contracting, he showed that $\jordan(gh)$~is approximately equal to~$\jordan(g) + \jordan(h)$. So if we restrict all maps to have the same ``type'', our argument works.

\item Here is where the second difficulty comes: the argument of~\cite{Ben96} only works for maps that are actually $\mathbb{R}$-regular in addition to being generic. In most representations this is automatically true: if every restricted root occurs as a restricted weight, then every generic map is in particular $\mathbb{R}$-regular. But when the representation is ``too small'', this is not the case. (A surprising fact is that a handful of representations are actually ``too large'' and ``too small'' at the same time: see Example~\ref{equivalence_class_examples}.4!)

As an example, consider the subgroup~$G$ of~$\GL_5(\mathbb{R})$ consisting of transformations preserving the quadratic form
\[x_1 x_3 + x_2 x_4 + x_5^2.\]
This is a form of signature~$(3, 2)$, so $G \simeq \SO(3, 2)$. Now take any real number $\lambda > 1$ and any $x \in \mathbb{R}$; then the element
\[g = \begin{pmatrix}
\lambda & \lambda x & 0 & 0 & 0 \\
0 & \lambda & 0 & 0 & 0 \\
0 & 0 & \lambda^{-1} & 0 & 0 \\
0 & 0 & - \lambda^{-1} x & \lambda^{-1} & 0 \\
0 & 0 & 0 & 0 & 1
\end{pmatrix} \in G\]
is generic in the standard representation (``pseudohyperbolic'' in the terminology of~\cite{AMS02} and~\cite{Smi13}), but not $\mathbb{R}$-regular (when $x \neq 0$ it is not even semisimple!).

Just as there are two different notions of being ``generic'' (the notion of $\mathbb{R}$-regularity, which is adapted to the adjoint representation, and the notion of being generic in~$\rho$), there are also two different notions of being ``in general position'', two different notions of being ``strongly contracting'' and so on. The results of~\cite{Ben96} rely on the stronger version of every property.

If we had used them as such, our Proposition~\ref{regular_product_qualitative} about products of maps ``of given type'' (and the subsequent propositions that rely on it) would no longer include, as a particular case, the corresponding result for~$G = \SO^+(n+1, n)$ (namely Lemma~5.6, point~(1) in~\cite{AMS02}). Instead, we would need to duplicate all definitions, and to always require that the maps we deal with satisfy both versions of the constraints. (In particular, we would probably lose the benefit of the unified treatment of the linear part and translation part, as outlined in Remark~\ref{unified_treatment}).

This weaker version is in theory sufficient for us, because it is known that ``almost all'' elements are $\mathbb{R}$-regular. So it is actually possible to construct the group~$\Gamma$ in such a way that its elements all have this additional property, and thus provide a working proof of the Main Theorem. But we felt that the simpler, stronger version of Proposition~\ref{regular_product_qualitative} was interesting in its own right.

To prove it, we needed a generalization of the results of~\cite{Ben96}. Benoist's subsequent paper~\cite{Ben97} does seem to provide such a generalization, by proving similar theorems with the hypothesis of $\mathbb{R}$-regularity replaced by what he calls ``$\theta$-proximality'', for~$\theta$ some subset of~$\Pi$. This is quite close to what we are looking for; but unfortunately, the results of~\cite{Ben97} rely on the assumption that the Jordan projections of the maps lie in a vector subspace of~$\mathfrak{a}$ (see Remark~\ref{comparison_with_Benoist} for details), which is unacceptably restrictive for us. So in Section~\ref{sec:jordan_additivity} of this paper, we redeveloped this theory in a suitably general way. We did reuse some basic results from~\cite{Ben97}; for example one of the key steps of our proof, Proposition~\ref{proximal_product} (about products of proximal maps), is very similar to Lemma~2.2.2 from~\cite{Ben97}.
\end{itemize}

\subsection{Plan of the paper}
In Section~\ref{sec:algebraic_prelim}, we give some background from representation theory.

In Section~\ref{sec:choice}, we study the dynamics of elements of~$A$. We choose one particular element~$X_0 \in \mathfrak{a}$ with some nice properties, with the goal of eventually ``modeling'', in some sense, generators of the group~$\Gamma$ on~$\exp(X_0)$.

In Section~\ref{sec:dynamics_X_0}, we study the dynamics of elements~$g$ of the affine group~$G \ltimes V$ that are of type~$X_0$ (see Definition~\ref{type_X0_definition}). This section culminates in the definition of the Margulis invariant of~$g$, which measures the translation part of~$g$ along its ``axis''.

In Section~\ref{sec:quantitative}, we study some quantitative properties of such elements~$g$. In particular we define a quantitative measure of being ``in general position'', and a quantitative measure of being ``strongly contracting''; both of these notions are tailored to the choice of~$\rho$ and of~$X_0$. We also define analogous notions for proximal maps, and prove a theorem (Proposition~\ref{proximal_product}) about products of proximal maps.

Section~\ref{sec:jordan_additivity} is where most of the new ideas of this paper are exploited. Here we apply the theory of products of proximal maps to a selection of ``fundamental representations''~$\rho_i$ (defined in Proposition~\ref{fundamental_real_representation}). The goal is to show that the product of two strongly contracting maps of type~$X_0$ in general position is still of type~$X_0$.

In Section~\ref{sec:quantitative_properties_of_products}, we now apply the theory of products of proximal maps to suitable exterior powers of the maps~$g$, in order to study the quantitative properties of products of elements of type~$X_0$. This section follows Section~3.2 of~\cite{Smi14} very closely.

Section~\ref{sec:additivity} contains the key part of the proof. We prove that if we take two strongly contracting maps of type~$X_0$ in general position, the Margulis invariant of their product is close to the sum of their Margulis invariants. This section follows Section~4 of~\cite{Smi14} very closely.

The very short Section~\ref{sec:induction} uses induction to extend the results of the two previous sections to products of an arbitrary number of elements. We omit the proof, as it is a straightforward generalization of Section~5 in~\cite{Smi14}.

Section~\ref{sec:construction} contains the proof of the Main Theorem. It follows Section~6 of~\cite{Smi14} quite closely, but there are a couple of additions.

\subsection{Acknowledgements}
I am very grateful to my Ph.D. advisor, Yves Benoist, whose help while I was working on this paper has been invaluable to me.

I would also like to thank Bruno Le Floch for some interesting discussions, which in particular helped me gain more insight about weights of representations.

\section{Algebraic preliminaries}
\label{sec:algebraic_prelim}
In this section, we give some background about real finite-dimensional representations of semisimple real Lie groups.

In Subsection~\ref{sec:eigenvalues_in_different_representations}, for any element $g \in G$, we relate the eigenvalues and singular values of~$\rho(g)$ (where $\rho$~is some representation) to some ``absolute'' properties of~$g$.

In Subsection~\ref{sec:restricted_weights}, we enumerate some properties of restricted weights of a real finite-dimensional representation of a real semisimple Lie group.

\subsection{Eigenvalues in different representations}
\label{sec:eigenvalues_in_different_representations}

The goal of this subsection is to prove Proposition~\ref{eigenvalues_and_singular_values_characterization}, which expresses the eigenvalues and singular values of a given element $g \in G$ acting in a given representation~$\rho$, exclusively in terms of the structure of~$g$ in the abstract group~$G$ (respectively its Jordan decomposition and its Cartan decomposition).

\begin{proposition}[Jordan decomposition]
Let $g \in G$. There exists a unique decomposition of~$g$ as a product $g = g_h g_e g_u$, where:
\begin{itemize}
\item $g_h$~is conjugate in~$G$ to an element of~$A$ (\emph{hyperbolic});
\item $g_e$~is conjugate in~$G$ to an element of~$K$ (\emph{elliptic});
\item $g_u$~is conjugate in~$G$ to an element of~$N^+$ (\emph{unipotent});
\item these three maps commute with each other.
\end{itemize}
\end{proposition}
\begin{proof}
This was proved by Kostant: see~\cite{Kos73}, Proposition~2.1. Alternatively, see~\cite{Ebe96}, Theorem~2.19.24. Note however that technically, Kostant, Eberlein and our paper use three different sets of definitions of a hyperbolic, elliptic or unipotent element. That our definitions are equivalent to Kostant's (which are the ones used most commonly) is shown in \cite{Ebe96}, Theorem~2.19.16. That Eberlein's definitions are equivalent to Kostant's is shown in \cite{Ebe96}, Proposition~2.19.18.
\end{proof}

\begin{proposition}[Cartan decomposition]
Let $g \in G$. Then there exists a decomposition of~$g$ as a product~$g = k_1 a k_2$, with~$k_1, k_2 \in K$ and~$a = \exp(X)$ with~$X \in \mathfrak{a}^{+}$. Moreover, the element~$X$ is uniquely determined by~$g$.
\end{proposition}
\begin{proof}
This is a classical result; see \eg Theorem~7.39 in~\cite{Kna96}.
\end{proof}

\begin{definition}
\label{Jd_Ct_definition}
For every element~$g \in G$, we define:
\begin{itemize}
\item the \emph{Jordan projection} of~$g$, written~$\jordan(g)$, to be the unique element of the closed dominant Weyl chamber~$\mathfrak{a}^{+}$ such that the hyperbolic part $g_h$ (from the Jordan decomposition $g = g_h g_e g_u$ given above) is conjugate to~$\exp(\jordan(g))$;
\item the \emph{Cartan projection} of~$g$, written~$\cartan(g)$, to be the element~$X$ from the Cartan decomposition given above.
\end{itemize}
\end{definition}

To talk about singular values, we need to introduce a Euclidean structure. We are going to use a special one. 
\begin{lemma}
\label{K-invariant}
Let $\rho_*$ be some real representation of~$G$ on some space~$V_*$. There exists a $K$-invariant positive-definite quadratic form~$B_*$ on~$V_*$ such that all the restricted weight spaces are pairwise $B_*$-orthogonal.
\end{lemma}
We want to reserve the plain notation~$\rho$ for the ``default'' representation, to be fixed once and for all at the beginning of Section~\ref{sec:choice}. We use the notation~$\rho_*$ so as to encompass both this representation~$\rho$ and the representations~$\rho_i$ defined in Proposition~\ref{fundamental_real_representation}.

Such quadratic forms have already been considered previously: see for example Lemma 5.33.a) in~\cite{BenQui}.
\begin{example}
If $\rho_* = \Ad$~is the adjoint representation, then $B_*$~is the form~$B_\theta$ given by
\[\forall X, Y \in \mathfrak{g},\quad B_{\theta}(X, Y) = -B(X, \theta Y)\]
(see~(6.13) in~\cite{Kna96}), where $B$~is the Killing form and $\theta$~is the Cartan involution.
\end{example}
\begin{proof}
This follows from the well-known fact that for any morphism $G \to H$ of reductive Lie groups (here we take $H = \GL(V_*)$), one can always find a Cartan involution of~$H$ that is compatible with a given Cartan involution of~$G$. Alternatively, the form~$B_*$ can be constructed as a restriction of a positive-definite Hermitian form on $V_*^\mathbb{C}$ that is invariant by a suitable maximal compact subgroup of~$G^\mathbb{C}$ (and such a Hermitian form can be found by the usual trick of averaging over the action of that compact group).
\end{proof}

Recall that the \emph{singular values} of a map~$g$ in a Euclidean space are defined as the square roots of the eigenvalues of $g^* g$, where $g^*$ is the adjoint map. The largest and smallest singular values of~$g$ then give respectively the operator norm of~$g$ and the reciprocal of the operator norm of~$g^{-1}$.

\begin{proposition}
\label{eigenvalues_and_singular_values_characterization}
Let $\rho_*: G \to \GL(V_*)$ be any representation of~$G$ on some vector space~$V_*$; let $\lambda_*^1, \ldots, \lambda_*^{d_*}$ be the list of all the restricted weights of~$\rho_*$, repeated according to their multiplicities. Let $g \in G$; then:
\begin{hypothenum}
\item The list of the moduli of the eigenvalues of~$\rho_*(g)$ is given by
\[\left( e^{\lambda_*^i(\jordan(g))} \right)_{1 \leq i \leq d_*}.\]
\item The list of the singular values of~$\rho_*(g)$, with respect to a $K$-invariant Euclidean norm~$B_*$ on~$V_*$ that makes the restricted weight spaces of~$V_*$ pairwise orthogonal (such a norm exists by Lemma~\ref{K-invariant}), is given by
\[\left( e^{\lambda_*^i(\cartan(g))} \right)_{1 \leq i \leq d_*}.\]
\end{hypothenum}
\end{proposition}

\begin{proof}~
\begin{hypothenum}
\item Let $g = g_h g_e g_u$ be the Jordan decomposition of~$g$.

It is then well-known that $\rho_*(g_e)$ and $\rho_*(g_u)$ are still respectively elliptic and unipotent in $\GL(V_*)$, and in particular have eigenvalues of modulus~$1$. Since $g_h$, $g_e$ and~$g_u$ all commute with each other, we deduce that the eigenvalues of~$\rho_*(g)$ are equal, in modulus, to those of~$\rho_*(g_h)$.


On the other hand, $g_h$ is by definition conjugate to~$\exp(\jordan(g))$, so $\rho_*(g_h)$~has the same eigenvalues as~$\rho_*(\exp(\jordan(g))$.

Finally, since $\exp(\jordan(g))$ is in~$A$ (the group corresponding to the Cartan subspace), the list of the eigenvalues of~$\rho_*(\exp(\jordan(g)))$ is by definition given by
\[\left( e^{\lambda_*^i(\jordan(g))} \right)_{1 \leq i \leq d_*}.\]

\item Let $g = k_1 \exp(\cartan(g)) k_2$ be the Cartan decomposition of~$g$. Since~$\rho_*(k_1)$ and~$\rho_*(k_2)$ are $B_*$-orthogonal maps, the $B_*$-singular values of~$\rho_*(g)$ coincide with those of the map~$\exp(\cartan(g))$; since~$\exp(\cartan(g))$, being an element of~$A$, is self-adjoint, its singular values coincide with its eigenvalues. We conclude as in the previous point. \qedhere
\end{hypothenum}
\end{proof}

\subsection{Properties of restricted weights}
\label{sec:restricted_weights}

In this subsection, we introduce a few properties of restricted weights of real finite-dimensional representations. (Proposition~\ref{convex_hull_and_inequalities} is actually a general result about Coxeter groups.) The corresponding theory for ordinary weights is well-known: see for example Chapter~V in~\cite{Kna96}.

Let $\alpha_1, \ldots, \alpha_r$ be an enumeration of the set~$\Pi$ of simple restricted roots generating~$\Sigma^+$. For every~$i$, we set
\begin{equation}
\alpha'_i := \begin{cases}
2\alpha_i &\text{ if $2\alpha_i$ is a restricted root} \\
\alpha_i &\text{ otherwise.}
\end{cases}
\end{equation}
For every index~$i$ such that $1 \leq i \leq r$, we define the $i$-th \emph{fundamental restricted weight}~$\varpi_i$ by the relationship
\begin{equation}
2\frac{\langle \varpi_i, \alpha'_j \rangle}{\| \alpha'_j \|^2} = \delta_{ij}
\end{equation}
for every $j$ such that $1 \leq j \leq r$.

By abuse of notation, we will often allow ourselves to write things such as ``for all~$i$ in some subset $\Pi' \subset \Pi$, $\varpi_i$ satisfies...'' (tacitly identifying the set~$\Pi'$ with the set of indices of the simple restricted roots that are inside).

In the following proposition, for any subset $\Pi' \subset \Pi$, we denote:
\begin{itemize}
\item by~$W_{\Pi'}$ the Weyl subgroup of type~$\Pi'$:
\begin{equation}
\label{eq:WPi'_definition}
W_{\Pi'} := \langle s_\alpha \rangle_{\alpha \in \Pi'};
\end{equation}
\item by~$\mathfrak{a}^{+}_{\Pi'}$ the fundamental domain for the action of~$W_{\Pi'}$ on~$\mathfrak{a}$:
\begin{equation}
\label{eq:levi_weyl_chamber_definition}
\mathfrak{a}^{+}_{\Pi'} := \setsuch{X \in \mathfrak{a}}
{\forall \alpha \in \Pi',\; \alpha(X) \geq 0},
\end{equation}
which is a kind of prism whose base is the dominant Weyl chamber of~$W_{\Pi'}$.
\end{itemize}
\begin{proposition}
\label{convex_hull_and_inequalities}
Take any $\Pi' \subset \Pi$, and let us fix $X \in \mathfrak{a}^{+}_{\Pi'}$. Let $Y \in \mathfrak{a}$. Then the following two conditions are equivalent:
\begin{hypothenum}
\item the vector~$Y$ is in $\mathfrak{a}^{+}_{\Pi'}$ and satisfies the system of linear inequalities
\[\begin{cases}
\forall i \in \mathrlap{\Pi',}\qquad\qquad \varpi_i(Y) \leq \varpi_i(X), \\
\forall i \in \mathrlap{\Pi \setminus \Pi',}\qquad\qquad \varpi_i(Y) = \varpi_i(X);
\end{cases}\]
\item the vector~$Y$ is in $\mathfrak{a}^{+}_{\Pi'}$ and also in the convex hull of the orbit of~$X$ by~$W_{\Pi'}$.
\end{hypothenum}
\end{proposition}
\begin{proof}
For $\Pi' = \Pi$, this is well known: see \eg \cite{Hall15}, Proposition~8.44.

Now let~$\Pi'$ be an arbitrary subset of~$\Pi$. We may translate everything by the vector
\[\sum_{i \in \Pi \setminus \Pi'} \varpi_i(X) H_i\]
(where $(H_i)_{i \in \Pi}$ is the basis of~$\mathfrak{a}$ dual to the basis~$(\varpi_i)_{i \in \Pi}$ of~$\mathfrak{a}^*$), which is obviously fixed by~$W_{\Pi'}$. Thus we reduce the problem to the case where
\begin{equation}
\label{eq:vector_subspace_piprime}
\forall i \in \Pi \setminus \Pi', \quad \varpi_i(X) = 0.
\end{equation}
Now let~$\Sigma'$ be the intersection of~$\Sigma$ with the vector space~$\mathfrak{a}_{\Pi'}$ determined by this system of equations, which is also the linear span of $(\alpha_i)_{i \in \Pi'}$. Then $\Sigma'$ is a root system that has:
\begin{itemize}
\item $\Pi'$ as a simple root system;
\item $W_{\Pi'}$ as the Weyl group;
\item $\mathfrak{a}^{+}_{\Pi'} \cap \mathfrak{a}_{\Pi'}$ as the dominant Weyl chamber.
\end{itemize}
This reduces the problem to the case $\Pi' = \Pi$.
\end{proof}

\begin{proposition}
\label{restr_weight_lattice}
Every restricted weight of every representation of $\mathfrak{g}$ is a linear combination of fundamental restricted weights with integer coefficients.
\end{proposition}
\begin{proof}
This is a particular case of Proposition~5.8 in~\cite{BT65}. For a correction concerning the proof, see also Remark~5.2 in~\cite{BT72compl}.
\end{proof}

\begin{proposition}
\label{highest_restr_weight}
If $\rho_*$ is an irreducible representation of $\mathfrak{g}$, there is a unique restricted weight~$\lambda_*$ of~$\rho_*$, called its \emph{highest restricted weight}, such that no element of the form $\lambda_* + \alpha_i$ with $1 \leq i \leq r$ is a restricted weight of~$\rho_*$.
\end{proposition}
\begin{remark}
In contrast to the situation with non-restricted weights, the highest restricted weight is not always of multiplicity~$1$; nor is a representation uniquely determined by its highest restricted weight.
\end{remark}
\begin{proof}
This easily follows from the existence and uniqueness of the ordinary (non-restricted) highest weight, given for example in \cite{Kna96}, Theorem~5.5~(d).
%
\end{proof}

\begin{proposition}
\label{convex_hull}
Let $\rho_*$ be an irreducible representation of $\mathfrak{g}$; let $\lambda_*$ be its highest restricted weight. Let $\Lambda_{\lambda_*}$ be the restricted root lattice shifted by~$\lambda_*$:
\[\Lambda_{\lambda_*} := \setsuch{\lambda_* + c_1 \alpha_1 + \cdots + c_r \alpha_r}{c_1, \ldots, c_r \in \mathbb{Z}}.\]
Then the set of restricted weights of~$\rho_*$ is exactly the intersection of the lattice~$\Lambda_{\lambda_*}$ with the convex hull of the orbit $\setsuch{w(\lambda_*)}{w \in W}$ of~$\lambda_*$ by the restricted Weyl group.
\end{proposition}
\begin{proof}
Once again, this follows from the corresponding result for non restricted weights (see \eg \cite{Hall15}, Theorem~10.1) by passing to the restriction. In the case of restricted weights, one of the inclusions is stated in~\cite{Hel08}, Proposition~4.22.
\end{proof}


Theorem~7.2 in~\cite{Tits71} yields as a special case the following result:

\begin{proposition}
\label{fundamental_real_representation}
For every index~$i$ such that $1 \leq i \leq r$, there exists an irreducible representation~$\rho_i$ of~$G$ on a space~$V_i$ whose highest restricted weight is equal to~$n_i \varpi_i$ (for some positive integer~$n_i$) and has multiplicity~$1$.
\end{proposition}
Here is a result describing the restricted weights of these representations.

\begin{lemma}
\label{fund_repr_other_weights}
Fix an index~$i$ such that $1 \leq i \leq r$. Then:
\begin{hypothenum}
\item $\rho_i$ has $n_i \varpi_i - \alpha_i$ as a restricted weight;
\item all restricted weights of~$\rho_i$ other than $n_i \varpi_i$ have the form
\[n_i \varpi_i - \alpha_i - \sum_{j=1}^r c_j \alpha_j,\]
with $c_j \geq 0$ for every $j$.
\end{hypothenum}
\end{lemma}
\begin{proof}~
\begin{hypothenum}
\item We have
\begin{align}
s_{\alpha_i}(n_i \varpi_i)
  &= s_{\alpha'_i}(n_i \varpi_i) \nonumber \\
  &= n_i \varpi_i - 2 n_i \frac{\langle \varpi_i, \alpha'_i \rangle}{\langle \alpha'_i, \alpha'_i \rangle}\alpha'_i \nonumber \\
  &= n_i \varpi_i - n_i \alpha'_i
\end{align}
(recall that~$\alpha'_i$ is equal to~$2\alpha_i$ if~$2\alpha_i$ is a restricted root and to~$\alpha_i$ otherwise). By Proposition~\ref{convex_hull}, $s_{\alpha_i}(n_i \varpi_i)$~is a restricted weight of~$\rho_i$ (because it is the image of a restricted weight of~$\rho_i$ by an element of the Weyl group) and then $n_i \varpi_i - \alpha_i$ is also a restricted weight of~$\rho_i$ (as a convex combination of two restricted weights of~$\rho_i$, that belongs to the restricted root lattice shifted by~$n_i \varpi_i$).
\item Let $\lambda$ be some restricted weight of~$\rho_i$. By Proposition~\ref{convex_hull} taken together with Proposition~\ref{convex_hull_and_inequalities}, we already know that it can be written as
\[\lambda = n_i \varpi_i - \sum_{j=1}^r c'_j \alpha_j,\]
where all coefficients $c'_j$ are nonnegative integers. It remains to show that if $\lambda \neq n_i \varpi_i$, then necessarily $c'_i > 0$.

Assume that $c'_i = 0$. By Proposition~8.42 in~\cite{Hall15}, we lose no generality in assuming that $\lambda$~is dominant. Let $\Pi_i := \Pi \setminus \{i\}$; by Proposition~\ref{convex_hull_and_inequalities}, it follows that $\lambda$~is then in the convex hull of the orbit of~$n_i \varpi_i$ by~$W_{\Pi_i}$. But clearly $W_{\Pi_i}$~fixes~$\varpi_i$, hence also~$n_i \varpi_i$. The conclusion follows. \qedhere
\end{hypothenum}
\end{proof}

\section{Choice of a reference Jordan projection}
\label{sec:choice}

For the remainder of the paper, we fix $\rho$ an irreducible representation of~$G$ on a finite-dimensional real vector space~$V$. For the moment, $\rho$~may be any representation; but in the course of the paper, we shall gradually introduce several assumptions on~$\rho$ (namely Assumptions \ref{zero_is_a_weight}, \ref{no_swinging}, \ref{transl_space} and~\ref{inverse_ok}) that will ensure that $\rho$~satisfies the hypotheses of the Main Theorem.

We denote by $\Omega$ the set of restricted weights of~$\rho$. For any $X \in \mathfrak{a}$, we define $\Omega^\subg_X$ (resp. $\Omega^\subl_X$, $\Omega^\sube_X$, $\Omega^\subge_X$, $\Omega^\suble_X$) to be the set of all restricted weights of~$\rho$ that take a positive (resp. negative, zero, nonnegative, nonpositive) value on~$X$:
\begin{align*}
\Omega^\subge_X &:= \setsuch{\lambda \in \Omega}{\lambda(X) \geq 0}
   & \Omega^\subg_X &:= \setsuch{\lambda \in \Omega}{\lambda(X) > 0} \\
\Omega^\suble_X &:= \setsuch{\lambda \in \Omega}{\lambda(X) \leq 0}
   & \Omega^\subl_X &:= \setsuch{\lambda \in \Omega}{\lambda(X) < 0} \\
\Omega^\sube_X &:= \setsuch{\lambda \in \Omega}{\lambda(X) = 0}. & &
\end{align*}
The goal of this section is to study these sets, and to choose a vector~$X_0 \in \mathfrak{a}^+$ for which the corresponding sets have some nice properties. The motivation for their study is that they parametrize the dynamical spaces (defined in Subsection~\ref{sec:dynamical_spaces_definition}) of~$\exp(X_0)$ (obviously), and actually of any element $g \in G$ whose Jordan projection ``has the same type'' as~$X_0$ (see Proposition~\ref{Asubge_is_ampa}).

In Subsection~\ref{sec:generic}, we introduce the notion of a generic vector $X \in \mathfrak{a}$, and impose a first constraint on~$\rho$: that $0$~be a restricted weight.

In Subsection~\ref{sec:types}, we introduce an equivalence relation on the set of generic vectors that identifies elements with the same dynamics, and give several examples.

In Subsection~\ref{sec:symmetry_pos_neg}, we introduce the notion of a symmetric vector $X \in \mathfrak{a}$, and ensure that~$\rho$ does not exclude generic vectors from being symmetric.

In Subsection~\ref{sec:parabolics}, we define parabolic subgroups and subalgebras of type~$X$; we also associate to every~$X \in \mathfrak{a}^+$ a set~$\Pi_X$ of simple restricted roots and a subgroup~$W_X$ of the restricted Weyl group.

In Subsection~\ref{sec:extreme}, we prove Proposition~\ref{change_of_X_0}, which shows that every equivalence class of generic vectors has a representative that has ``as much symmetry'' as the whole equivalence class, called an ``extreme'' representative.

At the end of this section, we shall fix once and for all an extreme, symmetric, generic vector $X_0 \in \mathfrak{a}^+$, which will serve as a reference Jordan projection (see the definition at the beginning of Subsection~\ref{sec:algebraic}).

\subsection{Generic elements}
\label{sec:generic}
We say that an element $X \in \mathfrak{a}$ is \emph{generic} if
\[\Omega^\sube_X \subset \{0\}.\]
\begin{remark}
This is indeed the generic case: it happens as soon as~$X$ avoids a finite collection of hyperplanes, namely the kernels of all nonzero restricted weights of~$\rho$.
\end{remark}

\begin{assumption}
\label{zero_is_a_weight}
From now on, we assume that $0$~is a restricted weight of~$\rho$:
\[0 \in \Omega,
\quad\text{ or equivalently }\quad
\dim V^0 > 0.\]
\end{assumption}
\begin{remark}
By Proposition~\ref{convex_hull}, this is the case if and only if the highest restricted weight of~$\rho$ is a $\mathbb{Z}$-linear combination of restricted roots.
\end{remark}

\begin{remark}
\label{zero_weight_follows_from_hypotheses}
We lose no generality in assuming this property, because it comes as a consequence of condition~\ref{itm:main_condition}\ref{itm:fixed_by_l} of the Main Theorem (which is also Assumption~\ref{transl_space}, see below). Indeed, any nonzero vector fixed by~$L$ is in particular fixed by~$A \subset L$, which means that it belongs to the zero restricted weight space.
\end{remark}

\begin{remark}
Conversely, this assumption provides the bare minimum without which the conclusion of the Main Theorem is certain to fail. In fact without this assumption, the group $G \ltimes V$ cannot even have \emph{any} infinite Zariski-dense subgroup acting properly. Indeed, let $\Gamma$~be such a subgroup; using a lemma due to Selberg, we lose no generality in assuming $\Gamma$ to be torsion-free. On the other hand, the linear part of a generic element~$g$ of such a group does not have~$1$ as an eigenvalue. This means that $g$ has a fixed point, which is a contradiction.
\end{remark}

In that case, for generic~$X$ we actually have
\[\Omega^\sube_X = \{0\}.\]

\subsection{Types of elements of~$\mathfrak{a}$}
\label{sec:types}

For two vectors~$X, Y \in \mathfrak{a}$, we say that $Y$ \emph{has the same type} as~$X$ if
\begin{equation}
\begin{cases}
\Omega^\subg_Y = \Omega^\subg_X; \\
\Omega^\subl_Y = \Omega^\subl_X,
\end{cases}
\end{equation}
\ie if every restricted weight takes the same sign on both of them. This implies that all five sets $\Omega^\subge$, $\Omega^\suble$, $\Omega^\sube$, $\Omega^\subg$ and~$\Omega^\subl$ coincide for~$X$ and~$Y$, hence that $\exp(X)$ and~$\exp(Y)$ have the same dynamical spaces (see Subsection~\ref{sec:dynamical_spaces_definition}).

This is an equivalence relation on~$\mathfrak{a}$, which partitions~$\mathfrak{a}$ into finitely many equivalence classes. We are only interested in generic equivalence classes. Some generic $X \in \mathfrak{a}$ being fixed, we call
\begin{equation}
\mathfrak{a}_{\rho, X} := \setsuch{Y \in \mathfrak{a}}{
\begin{cases}
\forall \lambda \in \Omega^\subg_X,\; \lambda(Y) > 0; \\
\forall \lambda \in \Omega^\subl_X,\; \lambda(Y) < 0
\end{cases}}
\end{equation}
its equivalence class in~$\mathfrak{a}$. If $X$~is dominant, we additionally call
\begin{equation}
\mathfrak{a}^{+}_{\rho, X} := \mathfrak{a}_{\rho, X} \cap \mathfrak{a}^{+}
\end{equation}
its equivalence class in the closed dominant Weyl chamber~$\mathfrak{a}^{+}$.
\begin{remark}
Every equivalence class is a convex cone. Also, these equivalence classes actually coincide  with connected components of the set of generic vectors.
\end{remark}
\begin{example}~
\label{equivalence_class_examples}
\begin{enumerate}
\item If $G$ is any noncompact semisimple real Lie group and $\rho = \Ad$ is its adjoint representation (so that $V = \mathfrak{g}$):
\begin{itemize}
\item A vector~$X \in \mathfrak{a}$ is generic if and only if it lies in one of the open Weyl chambers. In particular a vector~$X$ in~$\mathfrak{a}^+$ is generic if and only if it lies in~$\mathfrak{a}^{++}$.
\item All elements of~$\mathfrak{a}^{++}$ have the same type; so there is only one generic equivalence class in~$\mathfrak{a}^+$. Specifically, for any vector $X \in \mathfrak{a}^{++}$, we have ${\mathfrak{a}_{\Ad, X} = \mathfrak{a}^+_{\Ad, X} = \mathfrak{a}^{++}}$.
\end{itemize}
\item Take $G = \SO^+(3, 2)$. The root system is then~$B_2$:
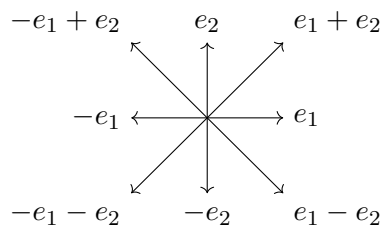
\[\begin{tikzpicture}
\def\roots{
(-1,1),  (0,1),  (1,1),
(-1,0),          (1,0),
(-1,-1), (0,-1), (1,-1)
}
\foreach \lambda in \roots{
  \draw[->] (0, 0) -- \lambda;
}
\node[above] at (0,1) {$e_2$};
\node[above right] at (1,1) {$e_1 + e_2$};
\node[right] at (1,0) {$e_1$};
\node[below right] at (1,-1) {$e_1 - e_2$};
\node[below] at (0,-1) {$-e_2$};
\node[below left] at (-1,-1) {$-e_1 - e_2$};
\node[left] at (-1,0) {$-e_1$};
\node[above left] at (-1,1) {$-e_1 + e_2$};
\end{tikzpicture}\]

As this group is split, the roots are also the restricted roots. Let $\rho$~be the representation with highest weight~$2e_1 + e_2$ (in the notations of \cite{Kna96}, Appendix~C). This is a representation of dimension~$35$, whose weights (also restricted weights) are as follows:
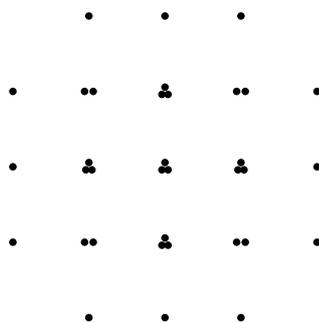
\[\begin{tikzpicture}
\def\simpleweights{
         (-1,2),  (0,2),  (1,2),
(-2,1),                           (2,1),
(-2,0),                           (2,0),
(-2,-1),                          (2,-1),
         (-1,-2), (0,-2), (1,-2)
}
\def\doubleweights{
(-1,1),          (1,1),
(-1,-1),         (1,-1)
}
\def\tripleweights{
         (0,1),
(-1,0),  (0,0),  (1,0),
         (0,-1)
}
\foreach \lambda in \simpleweights{
  \node[circle, fill, inner sep=1pt] at \lambda {};
}
\foreach \lambda in \doubleweights{
  \node[right, circle, fill, inner sep=1pt] at \lambda {};
  \node[left, circle, fill, inner sep=1pt] at \lambda {};
}
\foreach \lambda in \tripleweights{
  \node[below right, circle, fill, inner sep=1pt] at \lambda {};
  \node[above, circle, fill, inner sep=1pt] at \lambda {};
  \node[below left, circle, fill, inner sep=1pt] at \lambda {};
}
\end{tikzpicture}\]
(the number of dots at each node represents multiplicity.) Then every equivalence class in~$\mathfrak{a}$ is contained in some Weyl chamber: see Figure~\ref{fig:large_repr}. The dominant Weyl chamber~$\mathfrak{a}^+$ is split into two equivalence classes by the line of slope~$\frac{1}{2}$, kernel of the weight~$-e_1 + 2e_2$.

\begin{figure}[b]
\centering
\begin{subfigure}{0.45\textwidth}
\[\begin{tikzpicture}[scale=0.9]
\def\roots{
(-1,1),  (0,1),  (1,1),
(-1,0),          (1,0),
(-1,-1), (0,-1), (1,-1)
}
\def\someroots{
(-1,1),  (0,1),  (1,1),
(-1,0)
}
\def\weights{
         (-1,2),  (0,2),  (1,2),
(-2,1),  (-1,1),  (0,1),  (1,1),  (2,1),
(-2,0),  (-1,0),  (0,0),  (1,0),  (2,0),
(-2,-1), (-1,-1), (0,-1), (1,-1), (2,-1),
         (-1,-2), (0,-2), (1,-2)
}
\begin{scope}
  \clip (-3, -3) rectangle (3, 3);
  \fill[black!10!white] (0,0) -- (4,0) -- (4,2) -- (0,0);
  \fill[black!20!white] (0,0) -- (4,2) -- (4,4) -- (0,0);
  \foreach \lambda in \weights{
    \draw[black!30!white, line width = 1.2pt] ($ (0,0)!5!90:\lambda $) -- ($ (0,0)!5!-90:\lambda $);
  }
  \foreach \lambda in \someroots{
    \draw[loosely dashed] ($ (0,0)!5!90:\lambda $) -- ($ (0,0)!5!-90:\lambda $);
  }
  \begin{scope}
    \clip (0,0) -- (4,0) -- (4,4) -- (0,0);
    \foreach \x in {0,...,100}
      \draw (0.2*\x,0) -- (0,0.2*\x);
  \end{scope}
\end{scope}
\end{tikzpicture}\]
\caption{$\rho$ of highest weight~$2e_1 + e_2$}
\label{fig:large_repr}
\end{subfigure}
\quad
\begin{subfigure}{0.45\textwidth}
\[\begin{tikzpicture}[scale=0.9]
\def\roots{
(-1,1),  (0,1),  (1,1),
(-1,0),          (1,0),
(-1,-1), (0,-1), (1,-1)
}
\def\someroots{
(-1,1),  (0,1),  (1,1),
(-1,0)
}
\def\weights{
         (0,1),
(-1,0),  (0,0),  (1,0),
         (0,-1)
}
\begin{scope}
  \clip (-3, -3) rectangle (3, 3);
  \fill[black!7!white] (0,0) rectangle (4,4);
  \foreach \lambda in \weights{
    \draw[black!30!white, line width = 1.2pt] ($ (0,0)!5!90:\lambda $) -- ($ (0,0)!5!-90:\lambda $);
  }
  \foreach \lambda in \someroots{
    \draw[loosely dashed] ($ (0,0)!5!90:\lambda $) -- ($ (0,0)!5!-90:\lambda $);
  }
  \begin{scope}
    \clip (0,0) -- (4,0) -- (4,4) -- (0,0);
    \foreach \x in {0,...,100}
      \draw (0.2*\x,0) -- (0,0.2*\x);
  \end{scope}
\end{scope}
\end{tikzpicture}\]
\caption{$\rho$ of highest weight~$e_1$}
\label{fig:small_repr}
\end{subfigure}
\caption{Equivalence classes and Weyl chambers for two different representations of~$G = \SO^+(3,2)$. Dashed lines represent walls of Weyl chambers. Thick gray lines represent kernels of nonzero weights, which separate the different equivalence classes. The dominant Weyl chamber $\mathfrak{a}^+$ is hatched. All equivalence classes in~$\mathfrak{a}$ that intersect~$\mathfrak{a}^+$ are shaded, with different shades if there are more than one.}
\label{fig:equivalence_classes}
\end{figure}
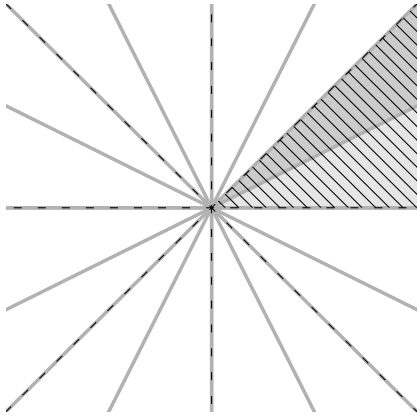
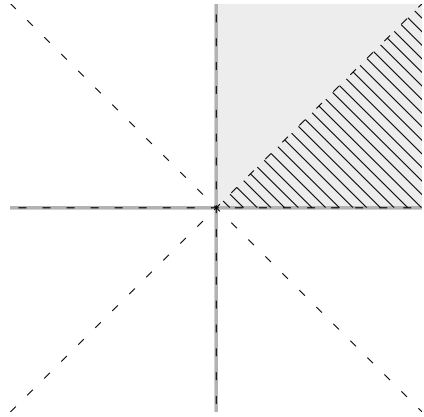
 
\item Take $G = \SO^+(3, 2)$ and $\rho$ the standard representation on $V = \mathbb{R}^{5}$. Using once again the notations of \cite{Kna96}, Appendix~C, its highest weight is~$e_1$ and its weights are $\pm e_1$, $\pm e_2$ and~$0$ (of course all with multiplicity~$1$). Then (see Figure~\ref{fig:small_repr}):
\begin{itemize}
\item A vector~$X \in \mathfrak{a}^+$ is generic if and only if it avoids the ``horizontal'' wall of the dominant Weyl chamber (the one normal to~$e_2$).
\item All such vectors have the same type. So for a generic~$X \in \mathfrak{a}^+$, the equivalence class~$\mathfrak{a}^+_{\rho, X}$ is the half-open dominant Weyl chamber, with the diagonal wall included and the horizontal wall excluded.
\item The whole equivalence class~$\mathfrak{a}_{\rho, X}$ is then an open quadrant of the plane~$\mathfrak{a}$, consisting of two half-open Weyl chambers glued back-to-back along their shared diagonal wall.
\end{itemize}

\item Suppose that $G$ and $\rho$~are such that the set of the restricted weights of~$\rho$ neither contains all restricted roots of~$G$, nor is contained in the set of restricted roots of~$G$ and their multiples. Then both phenomena occur at the same time: equivalence classes in~$\mathfrak{a}$ neither contain nor are contained in the Weyl chambers.

Examples are not immediate to come up with: the author even mistakenly believed for some time that no such representations existed. However, here is one such example:
\begin{itemize}
\item Take~$G = \operatorname{PSp}_4(\mathbb{R})$ (which is a split form), following the notation convention of~\cite{Kna96}: this is a group of rank~$4$ with a standard representation of dimension~$8$ (most people would call it $\operatorname{PSp}_8(\mathbb{R})$ instead). In the notations of~\cite{Kna96}, Appendix~C, its roots are all the possible expressions of the form~$\pm e_j \pm e_i$ or~$\pm 2 e_i$. 
\item Take~$\rho$ to be the representation with highest weight~$e_1 + e_2 + e_3 + e_4$. It has:
\begin{itemize}
\item the $16$~weights of the form $\pm e_1 \pm e_2 \pm e_3 \pm e_4$, with multiplicity~$1$;
\item the $24$~weights of the form $\pm e_i \pm e_j$, with multiplicity~$1$;
\item the zero weight with multiplicity~$2$,
\end{itemize}
for a total dimension of~$42$.
\end{itemize}
The reader may check that there are then three different ``types'' of generic vectors in the dominant Weyl chamber~$\mathfrak{a}^+$, with \eg the following representatives:
\begin{enumerate}[label=(\alph*)]
\item $X = (4, 2, 1, 0)$;
\item $X = (5, 3, 2, 1)$;
\item $X = (4, 3, 2, 0)$.
\end{enumerate}
In cases (a) and (c), we notice that $X$ lies on the wall normal to~$2e_4$; its equivalence class then contains a whole slice of that wall.
\end{enumerate}
\end{example}

\subsection{Swinging}
\label{sec:symmetry_pos_neg}
We start this subsection with the following observation: if the Jordan projection of~$g$ is~$X$, then the Jordan projection of~$g^{-1}$ is $-w_0(X)$, where $w_0$ is the ``longest element'' of the Weyl group that interchanges positive and negative restricted roots (see Section~\ref{sec:lie}).

We would like to ensure that for every element~$g$ of the group~$\Gamma$ we are trying to construct, the element~$g$ itself and its inverse~$g^{-1}$ have similar dynamics. To do that, we would like~$X$ and~$-w_0(X)$ to be of the same type. Replacing if necessary $X$ and~$-w_0(X)$ by their midpoint, we lose no generality in assuming they are actually equal.

\begin{definition}
We say that an element $X \in \mathfrak{a}$ is \emph{symmetric} if it is invariant by~$-w_0$:
\[-w_0(X) = X.\]
\end{definition}

Unfortunately, it is not always possible to find a vector~$X$ that is both symmetric and generic, as shown by the following example:
\begin{example}
Take $G = \SL_3(\mathbb{R})$. It is a split form, so its restricted root system is the same as its root system, namely $A_2$:
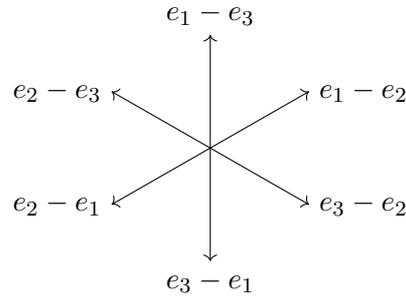
\[\begin{tikzpicture}[scale=1.5]
\foreach \theta in {30,90,...,330}
{
  \draw[->] (0, 0) -- (\theta:1);
}
\node[above] at (90:1) {$e_1 - e_3$};
\node[left] at (150:1) {$e_2 - e_3$};
\node[left] at (210:1) {$e_2 - e_1$};
\node[below] at (270:1) {$e_3 - e_1$};
\node[right] at (330:1) {$e_3 - e_2$};
\node[right] at (30:1) {$e_1 - e_2$};
\end{tikzpicture}\]
For this group, $-w_0$ is the map that exchanges the two simple positive roots $e_1 - e_2$ and $e_2 - e_3$ (we use the notations of~\cite{Kna96}, Appendix~C); in the picture above, it corresponds to the reflection about the vertical axis. So a vector $X \in \mathfrak{a}^+$ is symmetric if and only if it lies on that vertical axis (which bisects the dominant Weyl chamber~$\mathfrak{a}^+$).

Now consider the representation~$\rho$ of~$G$ with highest weight~$2e_1 - e_2 - e_3$. Note that this is three times the first fundamental weight, so $\rho$~is actually the third symmetric product~$S^3 \mathbb{R}^3$ of the standard representation. Here are its weights:
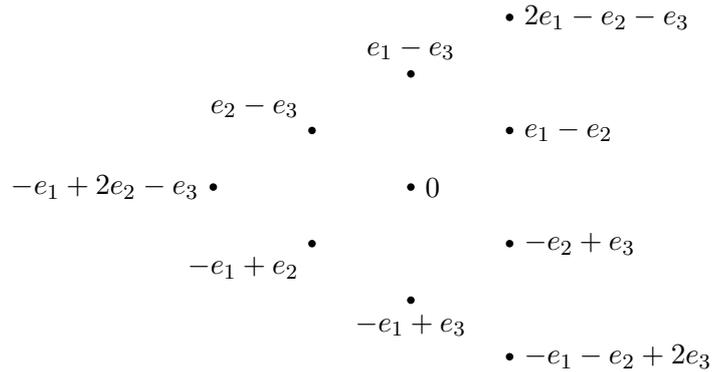
\[\begin{tikzpicture}[scale=1.5]
\coordinate (e13) at (90:1);
\coordinate (e12) at (30:1);
\node[circle, fill, inner sep=1pt, label=left:$-e_1 + 2e_2 - e_3$] at ($ (e13) - 2*(e12) $) {};
\node[circle, fill, inner sep=1pt, label=above left:$e_2 - e_3$] at ($ (e13) - (e12) $) {};
\node[circle, fill, inner sep=1pt, label=above:$e_1 - e_3$] at (e13) {};
\node[circle, fill, inner sep=1pt, label=right:$2e_1 - e_2 - e_3$] at ($ (e13) + (e12) $) {};
\node[circle, fill, inner sep=1pt, label=below left:$-e_1 + e_2$] at ($ (0,0) - (e12) $) {};
\node[circle, fill, inner sep=1pt, label=right:$0$] at (0,0) {};
\node[circle, fill, inner sep=1pt, label=right:$e_1 - e_2$] at (e12) {};
\node[circle, fill, inner sep=1pt, label=below:$-e_1 + e_3$] at ($ (0,0) - (e13) $) {};
\node[circle, fill, inner sep=1pt, label=right:$-e_2 + e_3$] at ($ (0,0) - (e13) + (e12) $) {};
\node[circle, fill, inner sep=1pt, label=right:$-e_1 - e_2 + 2e_3$] at ($ (0,0) - 2*(e13) + (e12) $) {};
\end{tikzpicture}\]
We see that any symmetric vector necessarily annihilates the weight~$-e_1 + 2e_2 - e_3$, hence it cannot be generic.
\end{example}

We call this phenomenon ``swinging''. Here is the picture to have in mind: when we apply the involution~$-w_0$ to some generic~$X$, the annihilator of~$X$ (\ie the hyperplane of~$\mathfrak{a}^*$ consisting of linear forms that vanish on~$X$) ``swings'' past the weight~$-e_1 + 2e_2 - e_3$, thus switching it from the set~$\Omega^\subg$ to the set~$\Omega^\subl$.

From now on, we assume that this issue does not arise:

\begin{assumption}[``No swinging'']
\label{no_swinging}
From now on, we assume that $\rho$~is such that there exists a symmetric generic element of~$\mathfrak{a}$.
\end{assumption}
This is precisely condition~\ref{itm:technical_condition} from the Main Theorem.
\begin{remark}~
\label{remark_no_swinging}
\begin{itemize}
\item It is well-known that when the restricted root system of~$G$ has any type other than $A_n$~(with~$n \geq 2$), $D_{2n+1}$ or~$E_6$, we actually have $w_0 = -\Id$. For those groups, \emph{every} vector $X \in \mathfrak{a}$ is symmetric, and so every representation satisfies this condition.
\item For the remaining groups, a straightforward linear algebra manipulation shows that this condition is equivalent to the following: no nonzero restricted weight of~$\rho$ must fall into the linear subspace
\begin{equation}
\setsuch{\lambda \in \mathfrak{a}^*}{w_0 \lambda = \lambda}
\end{equation}
(the ``axis of symmetry'' of~$w_0$ in~$\mathfrak{a}^*$). For example, this is always true for the adjoint representation (any restricted root fixed by~$w_0$ would need to be positive and negative at the same time). Heuristically, this seems to hold when the highest restricted weight is ``small'', but to quickly fail when it gets ``large enough''.
\end{itemize}
\end{remark}

\subsection{Parabolic subgroups and subalgebras}
\label{sec:parabolics}

A parabolic subgroup (or subalgebra) is usually defined in terms of a subset~$\Pi'$ of the set~$\Pi$ of simple restricted roots. We find it more convenient however to use a slightly different language. To every such subset corresponds a facet of the Weyl chamber, given by intersecting the walls corresponding to elements of~$\Pi'$. We may exemplify this facet by picking some element~$X$ in it that does not belong to any subfacet. Conversely, for every~$X \in \mathfrak{a}^+$, we define the corresponding subset
\begin{equation}
\Pi_{X} := \setsuch{\alpha \in \Pi}{\alpha(X) = 0}.
\end{equation}
The parabolic subalgebras and subgroups of type~$\Pi_{X}$ can then be very conveniently rewritten in terms of~$X$, as follows.
\begin{remark}
The set~$\Pi_X$ actually encodes the ``type'' of~$X$ with respect to the adjoint representation.
\end{remark}
\begin{definition}
For every $X \in \mathfrak{a}^+$, we define:
\begin{itemize}
\item $\mathfrak{p}_{X}^+$ and $\mathfrak{p}_{X}^-$ the parabolic subalgebras of type~$X$, and~$\mathfrak{l}_{X}$ their intersection:
\[\mathfrak{p}_{X}^+ := \mathfrak{l} \oplus \bigoplus_{\alpha(X) \geq 0} \mathfrak{g}^\alpha;\]
\[\mathfrak{p}_{X}^- := \mathfrak{l} \oplus \bigoplus_{\alpha(X) \leq 0} \mathfrak{g}^\alpha;\]
\[\mathfrak{l}_{X} := \mathfrak{l} \oplus \bigoplus_{\alpha(X) = 0} \mathfrak{g}^\alpha.\]

\item $P_{X}^+$ and $P_{X}^-$ the corresponding parabolic subgroups, and~$L_{X}$ their intersection:
\[P_{X}^+ := N_G(\mathfrak{p}_{X}^+);\]
\[P_{X}^- := N_G(\mathfrak{p}_{X}^-);\]
\[L_{X} := P_{X}^+ \cap P_{X}^-.\]
\end{itemize}
\end{definition}

An object closely related to these parabolic subgroups (see formula~\eqref{eq:Bruhat_type_X_0}, the Bruhat decomposition for parabolic subgroups) is the stabilizer of~$X$ in the Weyl group:
\begin{definition}
For any $X \in \mathfrak{a}^{+}$, we set
\[W_{X} := \setsuch{w \in W}{wX = X}.\]
\end{definition}
\begin{remark}
The group~$W_X$ is also closely related to the set~$\Pi_X$. Indeed, it follows immediately that a simple restricted root~$\alpha$ belongs to~$\Pi_X$ if and only if the corresponding reflection~$s_\alpha$ belongs to~$W_X$. Conversely, it is well-known (Chevalley's lemma, see \eg \cite{Kna96}, Proposition~2.72) that these reflections actually generate the group~$W_X$.

Thus $W_{X}$~is actually the same thing as~$W_{\Pi_X}$ (\ie the group~$W_{\Pi'}$ as defined in~\eqref{eq:WPi'_definition}, with $\Pi' = \Pi_X$).
\end{remark}
\begin{example}
To help understand the conventions we are taking, here are the extreme cases:
\begin{enumerate}
\item If $X$ lies in the open Weyl chamber~$\mathfrak{a}^{++}$, then:
\begin{itemize}
\item $P^+_X = P^+$ is the minimal parabolic subgroup; $P^-_X = P^-$; $L_X = L$;
\item $\Pi_X = \emptyset$;
\item $W_X = \{\Id\}$.
\end{itemize}
\item If $X=0$, then:
\begin{itemize}
\item $P^+_X = P^-_X = L_X = G$;
\item $\Pi_X = \Pi$;
\item $W_X = W$.
\end{itemize}
\end{enumerate}
\end{example}

\subsection{Extreme vectors}
\label{sec:extreme}

Besides $W_X$, we are also interested in the group
\begin{equation}
W_{\rho, X} := \setsuch{w \in W}{w X \text{ has the same type as } X},
\end{equation}
which is the stabilizer of~$X$ ``up to type''. It obviously contains~$W_X$. The goal of this subsection is to show that in every equivalence class, we can actually choose~$X$ in such a way that both groups coincide.
\begin{example}
In Example~\ref{equivalence_class_examples}.3 ($G = \SO^+(3,2)$ acting on~$V = \mathbb{R}^5$), the group $W_{\rho, X}$ corresponding to any generic~$X$ is a two-element group. If we take~$X$ to be generic not only with respect to~$\rho$ but also with respect to the adjoint representation (in other terms if $X$~is in an open Weyl chamber), then the group~$W_X$ is trivial. If however we take as~$X$ any element of the diagonal wall of the Weyl chamber, we have indeed $W_{X} = W_{\rho, X}$.
\end{example}

\begin{definition}
We call an element $X \in \mathfrak{a}^{+}$ \emph{extreme} if $W_X = W_{\rho, X}$, \ie if it satisfies the following property:
\[\forall w \in W,\quad w X \text{ has the same type as } X \iff w X = X.\]
\end{definition}

\begin{proposition}
\label{change_of_X_0}
For every generic $X \in \mathfrak{a}^{+}$, there exists a generic $X' \in \mathfrak{a}^{+}$ that has the same type as~$X$ and that is extreme.

If moreover $X$ is symmetric, then $X'$~is still symmetric.
\end{proposition}
\begin{remark}
The following statement will never be used in the paper (so we leave it without proof), but might help to understand what is going on: for every generic~$X$, we have
\[\mathfrak{a}_{\rho, X} = W_{\rho, X} \mathfrak{a}^+_{\rho, X} = W_{X'} \mathfrak{a}^+_{\rho, X}.\]
Also, it can be shown that a representative $X'$ of a given equivalence class $\mathfrak{a}^+_{\rho, X}$ is extreme if and only if it lies in every wall of the Weyl chamber that ``touches''~$\mathfrak{a}^+_{\rho, X}$ (or, equivalently, passes through~$\mathfrak{a}_{\rho, X}$), hence the term ``extreme''.
\end{remark}
\begin{proof}
To construct an element that has the same type as~$X$ but has the whole group~$W_{\rho, X}$ as stabilizer, we simply average over the action of this group: we set
\begin{equation}
X' = \sum_{w \in W_{\rho, X}} w X.
\end{equation}
(As multiplication by positive scalars does not change anything, we have written it as a sum rather than an average for ease of manipulation.) Then obviously:
\begin{itemize}
\item By definition every~$w X$ for~$w \in W_{\rho, X}$ has the same type as~$X$; since the equivalence class~$\mathfrak{a}_{\rho, X}$ is a convex cone, their sum~$X'$ also has the same type as~$X$.
\item In particular $X'$~is generic.
\item By construction whenever $w X$~has the same type as~$X$, we have~$w X' = X'$; conversely if $w$~fixes $X'$, then $w X$~has the same type as~$w X' = X'$ which has the same type as~$X$. So $X'$~is extreme.
\end{itemize}
Let us now show that $X' \in \mathfrak{a}^+$, \ie that for every $\alpha \in \Pi$, we have $\alpha(X') \geq 0$:
\begin{itemize}
\item If $s_\alpha X' = X'$, then obviously $\alpha(X') = 0$.
\item Otherwise, since~$X'$ is extreme, it follows that $s_\alpha X'$ does not even have the same type as~$X'$. Since $X'$~is generic, this means that there exists a restricted weight~$\lambda$ of~$\rho$ such that
\begin{equation}
\label{eq:transgressing_weight}
\begin{cases}
\lambda(X') > 0, \\
s_\alpha(\lambda)(X') < 0.
\end{cases}
\end{equation}
By definition, the same inequalities then hold for any~$Y$ with the same type as~$X'$ (or as~$X$):
\[\begin{cases}
\lambda(Y) > 0, \\
s_\alpha(\lambda)(Y) < 0.
\end{cases}\]
In particular the form $\lambda - s_\alpha(\lambda)$, which is a multiple of~$\alpha$, takes a positive value on every such~$Y$; hence $\alpha$ never vanishes on the equivalence class~$\mathfrak{a}_{\rho, X}$. By hypothesis~$X \in \mathfrak{a}^+$, so~$\alpha(X) \geq 0$. Since $\mathfrak{a}_{\rho, X}$~is connected, we conclude that~$\alpha(X') > 0$.
\end{itemize}
Finally, assume that $X$~is symmetric, \ie~$-w_0(X) = X$. Then since $w_0$~belongs to the Weyl group, it induces a permutation on~$\Omega$, hence we have:
\begin{equation}
w_0 \Omega^\subg_X = \Omega^\subg_{w_0 X} = \Omega^\subg_{-X} = \Omega^\subl_X,
\end{equation}
so that $w_0$ swaps the sets $\Omega^\subg_X$ and~$\Omega^\subl_X$. Now by definition we have
\begin{equation}
W_{\rho, X} = \Stab_W(\Omega^\subg_X) \cap \Stab_W(\Omega^\subl_X),
\end{equation}
hence $w_0$~normalizes~$W_{\rho, X}$. Obviously the map~$X \mapsto -X$ commutes with everything, so $-w_0$ also normalizes~$W_{\rho, X}$. We conclude that
\begin{align}
-w_0(X') &= \sum_{w \in W_{\rho, X}} -w_0 (w (X)) \nonumber \\
         &= \sum_{w' \in W_{\rho, X}} w' (-w_0 (X)) \nonumber \\
         &= X',
\end{align}
so that $X'$~is still symmetric.
\end{proof}
\begin{remark}
\label{classification_of_sets_Pi_X}
In practice, it can be shown that if $G$ is simple, the set $\Pi_X$ for extreme, symmetric, generic~$X$ can actually only be one of the following:
\begin{enumerate}[label=(\alph*)]
\item empty;
\item the set of long simple restricted roots;
\item the whole set~$\Pi$.
\end{enumerate}
Case~(a) accounts for the vast majority of representations. Case~(b) obviously only occurs when the restricted root system has a non-simply-laced Dynkin diagram ($G_2$, $F_4$, $B_n$, $C_n$ or~$BC_n$), and then only occurs in finitely many representations of each group. Case~(c) only occurs in trivial situations, namely when either~$\dim \mathfrak{a} = 0$ (\ie the group~$G$ is compact) or the representation is trivial.

The proof of this fact mostly relies on the following two observations:
\begin{itemize}
\item As soon as $\Omega$~is large enough to include some simple restricted root~$\alpha$, no set~$\Pi_{X'}$ may contain~$\alpha$. Indeed in that case, $\alpha(X')$~never vanishes for generic~$X'$.
\item The Weyl group acts transitively on the set of restricted roots of the same length; so as soon as $\Omega$~contains one restricted root of a given length, it contains all of them.
\end{itemize}
\end{remark}

For the remainder of the paper, we fix some symmetric generic vector~$X_0$ in the closed dominant Weyl chamber~$\mathfrak{a}^{+}$ that is extreme.

\section{Dynamics of maps of type~$X_0$}
\label{sec:dynamics_X_0}

Now we take an element $g$ in the affine group~$\rho(G) \ltimes V$ such that the Jordan projection of its linear part has the same type as~$X_0$. The goal of this section is to understand the dynamics of~$g$ acting on the affine space corresponding to~$V$, in particular its ``dynamical spaces'' defined in Subsection~\ref{sec:dynamical_spaces_definition}. There is a lot of parallelism between this section and Section~2 in~\cite{Smi14}.

In Subsection~\ref{sec:reference_dynamical_spaces}, we introduce the dynamical subspaces of~$X_0$. We also show that the stabilizers in~$G$ of those subspaces (except for the neutral one) are precisely the parabolic subgroups introduced in Subsection~\ref{sec:parabolics}.

In Subsection~\ref{sec:affine}, we introduce some formalism that reduces the study of the affine space~$V_{\Aff}$ corresponding to~$V$ to the study of a vector space called~$A$. We also introduce affine equivalents of linear notions defined previously.

In Subsection~\ref{sec:dynamical_spaces_definition}, we define the linear and affine dynamical subspaces associated to an element of the affine group~$\rho(G) \ltimes V$. This is very similar to Section~2.1 in~\cite{Smi14}.

In Subsection~\ref{sec:algebraic}, we give a description of the dynamical subspaces of an element ${g \in \rho(G) \ltimes V}$ whose Jordan projection has the same type as~$X_0$.

In Subsection~\ref{sec:quasi-translations}, we show that the action of any such element on its affine neutral space is a ``quasi-translation'', and explain what that means. This is a generalization of Section~2.4 in~\cite{Smi14}.

In Subsection~\ref{sec:canonical}, we introduce a family of canonical identifications between different affine neutral spaces, and use them to define the ``Margulis invariant'' for any such element~$g$, which is a vector measuring its translation part along a subspace of its affine neutral space. This is a generalization of Section~2.5 in~\cite{Smi14}.

\subsection{Reference dynamical spaces}
\label{sec:reference_dynamical_spaces}

Recall that~$X_0$ is some generic, symmetric, extreme vector in the closed dominant Weyl chamber~$\mathfrak{a}^{+}$, chosen once and for all.

\begin{definition}
\label{reference_dynamical_spaces}
We define the following subspaces of~$V$:
\begin{itemize}
\item $V^\subg_0 := \bigoplus_{\lambda(X_0) > 0} V^\lambda$,\quad the \emph{reference expanding space};
\item $V^\subl_0 := \bigoplus_{\lambda(X_0) < 0} V^\lambda$,\quad the \emph{reference contracting space};
\item $V^\sube_0 := \bigoplus_{\lambda(X_0) = 0} V^\lambda$,\quad the \emph{reference neutral space};
\item $V^\subge_0 := \bigoplus_{\lambda(X_0) \geq 0} V^\lambda$,\quad the \emph{reference noncontracting space};
\item $V^\suble_0 := \bigoplus_{\lambda(X_0) \leq 0} V^\lambda$,\quad the \emph{reference nonexpanding space}.
\end{itemize}
In other terms, $V^\subge_0$~is the direct sum of all restricted weight spaces corresponding to weights in~$\Omega^\subge_{X_0}$, and similarly for the other spaces.
\end{definition}
Clearly these are precisely the dynamical spaces (see Subsection~\ref{sec:dynamical_spaces_definition}) associated to the map~$\exp(X_0)$ (acting on~$V$ by~$\rho$).
\begin{remark}
Note that since~$X_0$ is generic, $V^\sube_0$~is actually just the zero restricted weight space:
\[V^\sube_0 = V^0;\]
moreover by Assumption~\ref{zero_is_a_weight}, zero \emph{is} a restricted weight, so this space is nontrivial.
\end{remark}
\begin{example}~
\begin{enumerate}
\item For $G = \SO^+(p, q)$ acting on~$V = \mathbb{R}^{p+q}$ (where~$p \geq q$), there is only one generic type. The spaces $V^\subg_0$ and~$V^\subl_0$ are some maximal totally isotropic subspaces (transverse to each other), $V^\subge_0$ and~$V^\suble_0$ are their respective orthogonal complements, and~$V^\sube_0$ is the $(p-q)$-dimensional space orthogonal to both~$V^\subg_0$ and~$V^\subl_0$.
\item If $G$~is any semisimple real Lie group acting on~$V = \mathfrak{g}$ (its Lie algebra) by the adjoint representation, then the reference noncontracting space~$\mathfrak{g}^\subge_0$ is obviously equal to~$\mathfrak{p}^+_{X_0}$. There is once again only one generic type, given by any~$X_0 \in \mathfrak{a}^{++}$; we then have~$\Pi_{X_0} = \emptyset$, so that $\mathfrak{p}^+_{X_0} = \mathfrak{p}^+$~is actually the (reference) \emph{minimal} parabolic subalgebra. We have similar identities for the other dynamical spaces, namely:
\begin{align*}
\mathfrak{g}^\subge_0 &= \mathfrak{p}^+; & \mathfrak{g}^\subg_0 &= \mathfrak{n}^+; \\
\mathfrak{g}^\suble_0 &= \mathfrak{p}^-; & \mathfrak{g}^\subl_0 &= \mathfrak{n}^-; \\
\mathfrak{g}^\sube_0 &= \mathfrak{l}. & &
\end{align*}
\end{enumerate}
\end{example}
Let us now understand what happens when we apply an element of~$G$ to one of those subspaces. The motivation for this, as well as the explanation of the term ``reference subspace'', comes from Corollary~\ref{dynamical_spaces_description}.
\begin{proposition}
\label{stabg=stabge}
We have:
\begin{hypothenum}
\item $\Stab_W(V^\subge_0) = \Stab_W(V^\subg_0) = \Stab_W(V^\suble_0) = \Stab_W(V^\subl_0) = W_{X_0}$.
\item $\Stab_G(V^\subge_0) = \Stab_G(V^\subg_0) = P_{X_0}^+$.
\item $\Stab_G(V^\suble_0) = \Stab_G(V^\subl_0) = P_{X_0}^-$.
\end{hypothenum}
\end{proposition}

\begin{remark}
\label{W_action_makes_sense}
Note that every restricted weight space is invariant by $Z_G(A) = L$: indeed, take some $\lambda \in \mathfrak{a}^*$, $v \in V^\lambda$, $l \in L$, $X \in \mathfrak{a}$; then we have:
\begin{equation}
X \cdot l(v) = l(\Ad (l^{-1}) (X) \cdot v) = l(X \cdot v) = \lambda(X)l(v).
\end{equation}
Moreover, the group $N_G(A)$ permutes these spaces. So if we have a direct sum of several restricted weight spaces, it makes sense to talk about its image by an element of~$W$; and we have the obvious identity
\begin{equation}
\forall w \in W, \;\forall \lambda \in \mathfrak{a}^*,\quad w V^\lambda = V^{w \lambda}.
\end{equation}
\end{remark}

\begin{proof}[Proof of Proposition~\ref{stabg=stabge}]~
\begin{hypothenum}
\item First note that since $X_0$~is generic, the only restricted weight that vanishes on~$X_0$ is the zero weight, so we have indeed
\[\Stab_W(\Omega^\subge_0) = \Stab_W(\Omega^\subg_0) = \Stab_W(\Omega^\suble_0) = \Stab_W(\Omega^\subl_0),\]
hence
\[\Stab_W(V^\subge_0) = \Stab_W(V^\subg_0) = \Stab_W(V^\suble_0) = \Stab_W(V^\subl_0).\]
Moreover, this group is obviously included in $W_{\rho, X_0} = \Stab_W(\mathfrak{a}_{\rho, X_0})$, which is also equal to~$W_{X_0}$ since $X_0$ is extreme. Conversely, let~$w \in W_{X_0}$; then $X_0$~is fixed by~$w$, and so is (say) the set~$\Omega^\subge_{X_0}$ of restricted weights nonnegative on~$X_0$. It follows that $\Stab_W(V^\subge_0)$ contains~$W_{X_0}$.

\item We first show that both $\Stab_G V^\subg_0$ and $\Stab_G V^\subge_0$ contain the group $P^+$. Indeed:
\begin{itemize}
\item The group $L$ stabilizes every restricted weight space $V^\lambda$, as noted in Remark~\ref{W_action_makes_sense} above.
\item Let $\alpha$ be a positive restricted root and $\lambda$ a restricted weight such that the value~$\lambda(X_0)$ is positive (resp. nonnegative). Then clearly we have
\[\mathfrak{g}^\alpha \cdot V^\lambda \subset V^{\lambda+\alpha},\]
and $(\lambda+\alpha)(X_0)$ is still positive (resp. nonnegative). Hence $\mathfrak{n}^+$ stabilizes $V^\subg_0$ and $V^\subge_0$.
\item The statement follows as $P^+ = L\exp(\mathfrak{n}^+)$.
\end{itemize}
Now take any element $g \in G$. Let us apply the Bruhat decomposition: we may write
\[g = p_1wp_2,\]
with $p_1, p_2$ some elements of the minimal parabolic subgroup $P^+$ and $w$ some element of the restricted Weyl group $W$ (see \eg \cite{Kna96}, Theorem 7.40). (Technically we need to replace~$w \in W = N_G(A)/Z_G(A)$ by some representative~$\tilde{w} \in N_G(A)$; but by the remark preceding this proof, we may ignore this distinction.) From the statement that we just proved it immediately follows that
\begin{equation}
\Stab_G(V^\subge_0) = \Stab_G(V^\subg_0) = P^+ \Stab_W(V^\subge_0) P^+ = P^+ W_{X_0} P^+.
\end{equation}
On the other hand, we have the Bruhat decomposition for parabolic subgroups:
\begin{equation}
\label{eq:Bruhat_type_X_0}
P_{X_0}^+ := \Stab_G(\mathfrak{p}_{X_0}^+) = P^+ W_{X_0} P^+.
\end{equation}
This can be shown by applying a similar reasoning to the adjoint representation: indeed in that case the space~$V^\subge_0$ corresponding to the same~$X_0$ is just $\mathfrak{p}_{X_0}^+$. (There is just a small difficulty due to the fact that $X_0$~is~not, in general, generic with respect to the adjoint representation.)

The conclusion follows.

\item Replacing $P^+$ and~$P_{X_0}^+$ respectively by $P^-$ and~$P_{X_0}^-$, the same reasoning applies. \qedhere
\end{hypothenum}
\end{proof}

\subsection{Extended affine space}
\label{sec:affine}
Let $V_{\Aff}$ be an affine space whose underlying vector space is~$V$.

\begin{definition}[Extended affine space] We choose once and for all a point~$p_0$ in~$V_{\Aff}$ which we take as an origin; we call $\mathbb{R} p_0$ the one-dimensional vector space formally generated by this point, and we set $A := V \oplus \mathbb{R} p_0$ the \emph{extended affine space} corresponding to $V$. (We hope that $A$,~the extended affine space, and $A$,~the~group corresponding to the Cartan space, occur in sufficiently different contexts that the reader will not confuse them.) Then $V_{\Aff}$ is the affine hyperplane ``at height~1'' of this space, and $V$ is the corresponding vector hyperplane:
\[
V        = V \times \{0\} \subset V \times \mathbb{R} p_0; \qquad
V_{\Aff} = V \times \{1\} \subset V \times \mathbb{R} p_0.
\]
\end{definition}
\begin{definition}[Linear and affine group]
Any affine map $g$ with linear part $\ell(g)$ and translation vector $v$, defined on $V_{\Aff}$ by
\[g: x \mapsto \ell(g)(x)+v,\]
can be extended in a unique way to a linear map defined on $A$, given by the matrix
\[\begin{pmatrix} \ell(g) & v \\ 0 & 1 \end{pmatrix}.\]

From now on, we identify the abstract group~$G$ with the group $\rho(G) \subset \GL(V)$, and the corresponding affine group~$G \ltimes V$ with a subgroup of~$\GL(A)$.
\end{definition}
\begin{definition}[Affine subspaces]
We define an \emph{extended affine subspace} of $A$ to be a vector subspace of $A$ not contained in $V$. For every~$k$, there is a one-to-one correspondence between $k+1$-dimensional extended affine subspaces of~$A$ and $k$-dimensional affine subspaces of~$V_{\Aff}$. For any extended affine subspace of~$A$ denoted by~$A_1$ (or~$A_2$, $A'$ and so on), we denote by~$V_1$ (or~$V_2$, $V'$ and so on) the space~$A \cap V$ (which is the linear part of the corresponding affine space $A \cap V_{\Aff}$).
\end{definition}
\begin{definition}[Translations]
By abuse of terminology, elements of the normal subgroup $V \lhd G \ltimes V$ will still be called \emph{translations}, even though we shall see them mostly as endomorphisms of $A$ (so that they are formally transvections). For any vector $v \in V$, we denote by~$\tau_{v}$ the corresponding translation.
\end{definition}
\begin{definition}[Reference affine dynamical spaces]
We now give a name for (the vector extensions of) the affine subspaces of $V_{\Aff}$ parallel respectively to $V^\subge_0$,~$V^\suble_0$ and $V^\sube_0$ and passing through the origin: we set
\begin{align*}
A^\subge_0 &:= V^\subge_0 \oplus \mathbb{R} p_0,\quad \text{ the \emph{reference affine noncontracting space};} \\
A^\suble_0 &:= V^\suble_0 \oplus \mathbb{R} p_0,\quad \text{ the \emph{reference affine nonexpanding space};} \\
A^\sube_0 &:= V^\sube_0 \oplus \mathbb{R} p_0,\quad \text{ the \emph{reference affine neutral space}.}
\end{align*}
These are obviously the affine dynamical spaces (see next subsection) corresponding to the map~$\exp(X_0)$, seen as an element of~$G \ltimes V$ by identifying~$G$ with the stabilizer of~$p_0$ in~$G \ltimes V$.
\end{definition}
\begin{definition}[Affine Jordan projection]
Finally, we extend the notion of Jordan projection to the whole group~$G \ltimes V$, by setting
\[\forall g \in G \ltimes V,\quad \jordan(g) := \jordan(\ell(g)).\]
\end{definition}
\needspace{\baselineskip}
\begin{remark}~
\begin{enumerate}
\item It is tempting to try to define an ``affine Jordan decomposition'', by observing that any affine map~$g \in G \ltimes V$ may be written as~$g = \tau_v g_h g_e g_u$, with $g_h$ (resp. $g_e$, $g_u$) conjugate in~$G \ltimes V$ to an element of~$A$ (resp. of~$K$, of~$N^+$) and~$v$~some element of~$V$. Unfortunately, we can neither require that $\tau_v$~commute with the other three factors, nor (as erroneously claimed in the author's previous paper~\cite{Smi14}) determine~$v$ in a unique fashion. The trouble comes from unipotent elements; to understand the problem, examine the affine transformation
\[g: \begin{pmatrix}x \\ y\end{pmatrix} \mapsto
\begin{pmatrix}1 & 1 \\ 0 & 1\end{pmatrix} \begin{pmatrix}x \\ y\end{pmatrix}
+ \begin{pmatrix}0 \\ 1\end{pmatrix}.\]
So we must be a little more careful; see the proof of Proposition~\ref{Asubge_is_ampa} for a more detailed study of conjugacy classes in~$G \ltimes V$.
\item We do not extend similarly the Cartan projection to~$G \ltimes V$, for the following reason. While eigenvalues of an element of~$G \ltimes V$ depend only on the eigenvalues of its linear part, the same statement does not hold for its singular values.
\end{enumerate}
\end{remark}

\subsection{Definition of dynamical spaces}
\label{sec:dynamical_spaces_definition}

For every $g \in G \ltimes V$, we define its \emph{linear dynamical spaces} as follows:
\begin{itemize}
\item $V^\subg_g$, the \emph{expanding space} associated to~$g$:

the largest vector subspace of~$V$ stable by~$g$ such that all eigenvalues~$\lambda$ of the restriction of~$g$ to that subspace satisfy $|\lambda| > 1$;
\item $V^\subl_g$, the \emph{contracting space} associated to~$g$:

the same thing with $|\lambda| < 1$;
\item $V^\sube_g$, the \emph{neutral space} associated to~$g$:

the same thing with $|\lambda| = 1$;
\item $V^\subge_g$, the \emph{noncontracting space} associated to~$g$:

the same thing with $|\lambda| \geq 1$;
\item $V^\suble_g$, the \emph{nonexpanding space} associated to~$g$:

the same thing with $|\lambda| \leq 1$.
\end{itemize}
Equivalently, $V^\subg_g$ is the direct sum of all the generalized eigenspaces~$E^\lambda$ of~$g$ associated to eigenvalues~$\lambda$ of modulus larger than~$1$ (defined as $E^\lambda = \ker (g - \lambda \Id)^n$ where~$n = \dim V$), and similarly for the four other spaces. We then obviously have
\begin{equation}
\label{eq:dynamical_spaces_decomposition}
V =
\lefteqn{\overbrace{\phantom{
  V^\subg_g \oplus V^\sube_g
}}^{\displaystyle V^\subge_g}}
  V^\subg_g \oplus \underbrace{
  V^\sube_g \oplus V^\subl_g
}_{\displaystyle V^\suble_g}.
\end{equation}
Also note that the restriction of~$g$ from~$A$ to~$V$ is just its linear part, so that the linear dynamic subspaces of~$g$ only depend on~$\ell(g)$.

For every $g \in G \ltimes V$, we define its \emph{affine dynamical subspaces}:
\begin{itemize}
\item $A^\subge_g$, the \emph{affine noncontracting space} associated to~$g$,
\item $A^\suble_g$, the \emph{affine nonexpanding space} associated to~$g$,
\item and $A^\sube_g$, the \emph{affine neutral space} associated to~$g$,
\end{itemize}
in the same way as the linear dynamical subspaces, but with~$V$ replaced everywhere by~$A$.

\begin{remark}~
\begin{itemize}
\item Note that if we defined in the same way~$A^\subg_g$ (resp.~$A^\subl_g$), it would actually be contained in~$V$ and so just be equal to~$V^\subg_g$ (resp.~$V^\subl_g$). Indeed an element of~$G \ltimes V$ can never act on a vector in $A \setminus V$ (\ie an element of~$A$ with a nonzero component along~$\mathbb{R} p_0$) with an eigenvalue other than~$1$.
\item Thus the affine analog of the decomposition~\eqref{eq:dynamical_spaces_decomposition} is now:
\begin{equation}
\label{eq:affine_dynamical_spaces_decomposition}
A =
\lefteqn{\overbrace{\phantom{
  V^\subg_g \oplus A^\sube_g
}}^{\displaystyle A^\subge_g}}
  V^\subg_g \oplus \underbrace{
  A^\sube_g \oplus V^\subl_g
}_{\displaystyle A^\suble_g}
\end{equation}
(pay attention to the distribution of $A$'s and $V$'s).
\item From this identity, it immediately follows that neither $A^\sube_g$, $A^\subge_g$ nor~$A^\suble_g$ are contained in~$V$.
\item Finally, it is obvious that the intersections of these three spaces with~$V$ are respectively  $V^\sube_g$, $V^\subge_g$ and~$V^\suble_g$. Thus this notation is consistent with the convention outlined above.
\end{itemize}
\end{remark}
In purely affine terms, these spaces may be understood as follows:
\begin{itemize}
\item $A^\sube_g \cap V_{\Aff}$ is the unique $g$-invariant affine space parallel to~$V^\sube_g$ (the ``axis'' of~$g$);
\item $A^\subge_g \cap V_{\Aff}$ is the unique affine space parallel to~$V^\subge_g$ and containing~$A^\sube_g \cap V_{\Aff}$, and similarly for~$A^\suble_g \cap V_{\Aff}$.
\end{itemize}

\subsection{Description of dynamical spaces}
\label{sec:algebraic}

We shall now characterize the dynamical subspaces of those elements of~$G \ltimes V$ that satisfy the following property.

\begin{definition}
\label{type_X0_definition}
We say that an element~$g \in G \ltimes V$ is \emph{of type~$X_0$} if $\jordan(g)$~has the same type as~$X_0$, \ie if
\[\jordan(g) \in \mathfrak{a}_{\rho, X_0}.\]
\end{definition}
\needspace{\baselineskip}
\begin{example}~
\begin{enumerate}
\item For $G = \SO^+(p, q)$ acting on~$V = \mathbb{R}^{p+q}$ (where $p \geq q$), there is only one generic type. For every~$g \in G$, we have
\begin{equation}
\dim V^\subg_g = \dim V^\subl_g \leq q.
\end{equation}
An element $g \in G$ is of generic type if and only if equality is attained. Such elements have been called \emph{pseudohyperbolic} in the previous literature (\cite{AMS02, Smi13}).
\item If $G$ is any semisimple real Lie group acting on~$V = \mathfrak{g}$ (its Lie algebra) by the adjoint representation, there is only one generic type and an element $g \in G$ is of that type if and only if $\jordan(g) \in \mathfrak{a}^{++}$. Such elements are called \emph{$\mathbb{R}$-regular} or (particularly in~\cite{BenQui}) \emph{loxodromic}.
\end{enumerate}
\end{example}

Here is a partial description of the dynamical spaces of an element of type~$X_0$.
\begin{proposition}
\label{Asubge_is_ampa}
Let $g \in G \ltimes V$ be a map of type~$X_0$. In that case:
\begin{hypothenum}
\item There exists a map~$\phi \in G \ltimes V$, called a \emph{canonizing map} for~$g$, such that
\[\begin{cases}
\phi(A^{\subge}_g) = A^{\subge}_0; \\
\phi(A^{\suble}_g) = A^{\suble}_0.
\end{cases}\]
\item The space~$V^\subg_g$ is uniquely determined by~$A^\subge_g$. The space~$V^\subl_g$ is uniquely determined by~$A^\suble_g$.
\end{hypothenum}
\end{proposition}
(Compare this with Claim 2.5 in~\cite{Smi14}.)
\begin{proof} \mbox{ }
\begin{hypothenum}
\item
\begin{itemize}
\item We start with the obvious decomposition
\begin{equation}
g = \tau_v \ell(g),
\end{equation}
where $\ell(g) \in G$~is the linear part of~$g$ (seen as an element of~$G \ltimes V$ by identifying~$G$ with the stabilizer of the origin~$p_0$) and $v \in V$~is its translation part. We then observe that we may rewrite this as
\begin{equation}
g = \tau_{v'} \tau_w^{-1} \ell(g) \tau_w
\end{equation}
for some $w \in V$, where $v'$ is now actually an element of $V^\sube_g$. Indeed, for any translation vector~$v \in V$ and linear map~$f \in G$, we have
\begin{equation}
f \tau_v = \tau_{f(v)} f.
\end{equation}
The statement then follows from the fact that the map induced by ${\ell(g) - \Id}$ on $V^\subg_g \oplus V^\subl_g$ does not have~$0$ as an eigenvalue, hence is surjective. (In fact, this argument shows that we could even require~$v'$ to lie in the actual characteristic space corresponding to the eigenvalue~$1$.)

\item Now let $\ell(g) =: g_h g_e g_u$ be the Jordan decomposition of~$\ell(g)$, so that
\begin{equation}
\tau_w g \tau_w^{-1} = \tau_{v'} g_h g_e g_u;
\end{equation}
let $\phi_\ell \in G$ be any map that conjugates $g_h$ to~$\exp(\jordan(g))$, \ie such that $\phi_\ell g_h \phi_\ell^{-1} = \exp(\jordan(g))$; and let $\phi := \phi_\ell \tau_w$.

Calling $g' := \phi g \phi^{-1}$ and $\tau_{v''}$, $g'_e$, $g'_u$ the respective conjugates of the maps $\tau_{v'}$,~$g_e$,~$g_u$ by~$\phi_\ell$ (so that $v'' = \phi_\ell(v')$), we then have
\begin{equation}
g' = \tau_{v''} \exp(\jordan(g)) g'_e g'_u,
\end{equation}
where $g'_e \in G$ is elliptic, $g'_u \in G$ is unipotent, both of them commute with $\exp(\jordan(g))$, and $v'' \in V^\sube_{g'}$.

As already seen in the proof of Proposition~\ref{eigenvalues_and_singular_values_characterization}, $g'_e$ and $g'_u$ have all eigenvalues of modulus~$1$ and commute with $\exp(\jordan(g))$. Hence the linear dynamical spaces of~$g'$ coincide with those of~$\exp(\jordan(g))$.

Now since $\exp(\jordan(g)) \in G$ fixes~$p_0$, the space~$A^{\sube}_{\exp(\jordan(g))}$ is equal to $V^{\sube}_{\exp(\jordan(g))} \oplus \mathbb{R} p_0$; and since $v'' \in V^\sube_{g'}$, that space is still invariant by~$g'$. It follows that we have
\begin{equation}
A^{\sube}_{g'} = A^{\sube}_{\exp(\jordan(g))} = V^{\sube}_{\exp(\jordan(g))} \oplus \mathbb{R} p_0.
\end{equation}
By taking the direct sum with~$V^\subg$ and with~$V^\subl$, we deduce that all the affine dynamical spaces of~$g'$ coincide with those of~$\exp(\jordan(g))$.

Now since $g$~is of type~$X_0$, by definition, $\jordan(g)$~is a vector in~$\mathfrak{a}$ that has the same type as~$X_0$. It follows that the affine dynamical subspaces of~$\exp(\jordan(g))$ coincide with those of~$\exp(X_0)$, which are the reference subspaces. We conclude that
\[\begin{cases}
A^{\subge}_{g'} = A^{\subge}_0; \\
A^{\suble}_{g'} = A^{\suble}_0.
\end{cases}\]
Since obviously $A^{\subge}_{g'} = \phi(A^{\subge}_g)$ and similarly for~$A^{\suble}$, the conclusion follows.
\end{itemize}

\item Suppose that $g_1$ and~$g_2$ are two maps of type~$X_0$ such that $A^\subge_{g_1} = A^\subge_{g_2}$. Define $g'_1 = \phi_1 g_1 \phi_1^{-1}$ and $g'_2 = \phi_2 g_2 \phi_2^{-1}$ as in the previous point; then we have
\[A^\subge_{g_{1,2^{\vphantom{x}}}} = \phi_1^{-1}(A^\subge_0) = \phi_2^{-1}(A^\subge_0).\]
In other terms, the transition map $\phi_2 \circ \phi_1^{-1}$ stabilizes~$A^\subge_0$.

Clearly the linear part of $\phi_2 \circ \phi_1^{-1}$ then stabilizes $V^\subge_0$. It follows from Proposition~\ref{stabg=stabge} that it also stabilizes $V^\subg_0$. Since the latter space is contained in $V$, the translation part of~$\phi_2 \circ \phi_1^{-1}$ acts trivially on it; so the affine map~$\phi_2 \circ \phi_1^{-1}$ itself also stabilizes~$V^\subg_0$.

Now obviously we also have $V^\subg_{g_1} = \phi_1^{-1}(V^\subg_{g'_1})$, and similarly for $g_2$ and~$\phi_2$. But it also follows from the previous point that
\[V^\subg_{g'_1} = V^\subg_{g'_2} = V^\subg_0.\]
We conclude that $V^\subg_{g_1} = V^\subg_{g_2}$ as required.

The same proof works for $A^\suble$ and $V^\subl$. \qedhere
\end{hypothenum}
\end{proof}
This immediately allows us to describe the remaining dynamical spaces of~$g$:
\begin{corollary}
\label{dynamical_spaces_description}
Let $g \in G \ltimes V$ be a map of type~$X_0$. Then if $\phi \in G \ltimes V$ is \emph{any} canonizing map of~$g$, we have:
\begin{align*}
\phi(A^\subge_g) &= A^\subge_0 & \phi(V^\subge_g) &= V^\subge_0 & \phi(V^\subg_g)  &= V^\subg_0 \\
\phi(A^\suble_g) &= A^\suble_0 & \phi(V^\suble_g) &= V^\suble_0 & \phi(V^\subl_g)  &= V^\subl_0 \\
\phi(A^\sube_g)  &= A^\sube_0 & \phi(V^\sube_g)  &= V^\sube_0. & &
\end{align*}
\end{corollary}
In other terms, if $\phi$~is a canonizing map of~$g$ then all eight dynamical spaces of the conjugate $\phi g \phi^{-1}$ coincide with the reference dynamical spaces. This explains why we called them ``reference'' spaces.
\begin{proof}
The equalities for~$A^\subge$ and~$A^\suble$ hold by definition of a canonizing map. The equality for~$A^\sube$ follows by taking the intersection. The equalities for~$V^\subge$, $V^\suble$ and~$V^\sube$ follow by taking the linear part. The equalities for~$V^\subg$ and~$V^\subl$ follow from Proposition~\ref{Asubge_is_ampa}~(ii).
\end{proof}

\subsection{Quasi-translations}
\label{sec:quasi-translations}

Let us now investigate the action of a map~$g \in G \ltimes V$ of type~$X_0$ on its affine neutral space~$A^\sube_g$. The goal of this subsection is to prove that it is ``almost'' a translation (Proposition~\ref{V=_translation}).

We fix on~$V$ a Euclidean form~$B$ satisfying the conditions of Lemma~\ref{K-invariant} for the representation~$\rho$.

\begin{definition}
\label{quasi-translation_def}
We call \emph{quasi-translation} any affine automorphism of $A^\sube_0$ induced by an element of $L \ltimes V^\sube_0$.
\end{definition}

Let us explain and justify this terminology. First note that the action of~$L$ on~$V^\sube_0$ preserves~$B$: indeed, the action of $M$ does so because $M \subset K$, and the action of $A$ on this space is just trivial. The following statement is then immediate:
\begin{proposition}
Let $V^\transl_0$ be the set of fixed points of~$L$ in~$V^\sube_0$:
\[V^\transl_0 := \setsuch{v \in V^\sube_0}{\forall l \in L,\; l v = v}.\]
(Note that this is also the set of fixed points of~$M$). Let $V^\rotat_0$ be the $B$-orthogonal complement of~$V^\transl_0$ in~$V^\sube_0$, and let $\Orth(V^\rotat_0)$~denote the set of $B$-preserving automorphisms of $V^\rotat_0$. Then any quasi-translation is an element of
\[\Big( \Orth(V^\rotat_0) \ltimes V^\rotat_0 \Big) \times V^\transl_0.\]
\end{proposition}
In other words, quasi-translations are affine isometries of $V^\sube_0$ that preserve the directions of $V^\rotat_0$ and $V^\transl_0$ and act by a pure translation on the $V^\transl_0$ component. You may think of a quasi-translation as a kind of ``screw displacement''; the superscripts $\transl$ and~$\rotat$ respectively stand for ``translation'' and ``rotation''.

We now claim that any map of type~$X_0$ acts on its affine neutral space by quasi-translations:
\begin{proposition}
\label{V=_translation}
Let $g \in G \ltimes V$ be a map of type~$X_0$, and let~$\phi \in G \ltimes V$ be any canonizing map for~$g$. Then the restriction of the conjugate $\phi g \phi^{-1}$ to $A^\sube_0$ is a quasi-translation.
\end{proposition}
Let us actually formulate an even more general result, which will have another application in the next subsection:
\begin{lemma}
\label{quasi-translation}
Any map $f \in G \ltimes V$ stabilizing both $A^\subge_0$ and $A^\suble_0$ acts on $A^\sube_0$ by quasi-translation.
\end{lemma}
\begin{proof}~
\begin{itemize}
\item We begin by showing that any element of $\mathfrak{l}_{X_0} = \mathfrak{p}_{X_0}^+ \cap \mathfrak{p}_{X_0}^-$ acts on $V^\sube_0$ in the same way as some element of $\mathfrak{l}$. Recall that by definition
\[\mathfrak{l}_{X_0} = \mathfrak{l} \oplus \bigoplus_{\alpha(X_0) = 0} \mathfrak{g}^\alpha;\]
hence it is sufficient to show that for every restricted root $\alpha$ such that $\alpha(X_0) = 0$, we have $\mathfrak{g}^\alpha \cdot V^\sube_0 = 0$. Indeed, since $V^\sube_0 = V^0$ (because $X_0$~is generic), we have
\[\mathfrak{g}^\alpha \cdot V^\sube_0 \subset V^\alpha.\]
On the other hand, we know by Proposition~\ref{stabg=stabge} that for such $\alpha$, the action of~$\mathfrak{g}^\alpha$ stabilizes both $V^\subge_0$ and $V^\suble_0$; it follows that the image~$\mathfrak{g}^\alpha \cdot V^\sube_0$ lies in both of these spaces, hence in their intersection $V^\sube_0$, which is also~$V^0$. Since $\alpha$ is nonzero, we have $V^0 \cap V^\alpha = 0$, which yields the desired equality.

\item Let $P_{X_0, e}^+$ and $P_{X_0, e}^-$ denote the identity components of $P_{X_0}^+$ and $P_{X_0}^-$ respectively; by integrating the previous statement, it follows that any element of~$P_{X_0, e}^+ \cap P_{X_0, e}^-$ acts on~$V^\sube_0$ in the same way as some element of~$L$.

\item Now it follows from \cite{Kna96} Proposition~7.82~(d) (using~7.83~(e)) that \begin{equation}
L_{X_0} = P_{X_0}^+ \cap P_{X_0}^- \subset M(P_{X_0, e}^+ \cap P_{X_0, e}^-).
\end{equation}
(Here we are using the assumption that $G$~is connected.) We deduce that any element of $L_{X_0}$ acts on $V^\sube_0$ in the same way as some element of~$L$.

\item Finally, any $f \in G \ltimes V$ stabilizing both $A^\subge_0$ and $A^\suble_0$ has linear part stabilizing both $V^\subge_0$ and $V^\suble_0$ (hence lying in $L_{X_0}$, by Proposition~\ref{stabg=stabge}), and translation part contained both in $V^\subge_0$ and in $V^\suble_0$ (in other words, in $V^\sube_0$). The conclusion follows. \qedhere
\end{itemize}
\end{proof}
\begin{proof}[Proof of Proposition~\ref{V=_translation}]
The proposition follows immediately from this lemma by taking $f = \phi g \phi^{-1}$. Indeed, by definition the ``canonized'' map~$\phi g \phi^{-1}$ has~$A^\subge_0$ and~$A^\suble_0$ as dynamical spaces; in particular it stabilizes them.
\end{proof}

\begin{example}~
\label{transl_space_example}
\begin{enumerate}
\item For $G = \SO^+(p, q)$ acting on $V = \mathbb{R}^{p+q}$ (with~$p \geq q$), we have:
\begin{itemize}
\item $M \simeq \SO_{p-q}(\mathbb{R}) \times (\mathbb{Z}/2\mathbb{Z})^{q-1}$,
\item $V^\sube_0 = V^0 \simeq \mathbb{R}^{p-q}$,
\end{itemize}
and the action of~$M$ on~$V^\sube_0$ is as follows: the connected factor~$\SO_{p-q}(\mathbb{R})$ acts in the obvious way; the discrete factor~$(\mathbb{Z}/2\mathbb{Z})^{q-1}$ acts trivially. We may then distinguish two cases:
\begin{enumerate}[label=\alph*.]
\item If $p - q \geq 2$, then the action of~$M$ is transitive. The space~$V^\transl_0$ is trivial and~$V^\rotat_0 = V^\sube_0$. Any affine isometry of~$V^\sube_0$ may be a quasi-translation.
\item If $p - q = 1$, then the group~$M$ is trivial. We have on the contrary~$V^\transl_0 = V^\sube_0$ and $V^\rotat_0$~is trivial. A quasi-translation is just a translation.
\end{enumerate}
(We exclude the case $p = q$ because in that case $V^0 = 0$, which violates Assumption~\ref{zero_is_a_weight}.)
\item More generally if $G$~is split, then we have $\mathfrak{m} = 0$. The group~$M$ is in general a nontrivial finite group; however, it can be shown (by considering the complexification of~$G$) that we still always have~$V^\transl_0 = V^\sube_0$, and a quasi-translation is still just a translation.
\item If $G$ is any semisimple real Lie group acting on~$V = \mathfrak{g}$ (its Lie algebra) by the adjoint representation, then:
\begin{itemize}
\item $\mathfrak{g}^\sube_0 = \mathfrak{g}^0 = \mathfrak{l}$;
\item $\mathfrak{g}^\transl_0$ is the direct sum of~$\mathfrak{a}$ and of the center of~$\mathfrak{m}$;
\item $\mathfrak{g}^\rotat_0$ is the semisimple part of~$\mathfrak{m}$ (in other terms, its derived subalgebra).
\end{itemize}
The example of $G = \SO^+(4, 1)$ (acting on~$\mathfrak{so}(4, 1)$, not on~$\mathbb{R}^5$) shows that $V^\transl_0$ and~$V^\rotat_0$ can both be nontrivial at the same time.
\end{enumerate}
\end{example}
We would like to treat quasi-translations a bit like translations; for this, we need to have at least a nontrivial space~$V^\transl_0$. So from now on, we exclude cases like~1.a. in the list of examples we just considered (Example~4.22):
\begin{assumption}
\label{transl_space}
The representation~$\rho$ is such that
\[\dim V^\transl_0 > 0.\]
\end{assumption}
This is precisely condition~\ref{itm:main_condition}\ref{itm:fixed_by_l} from the Main Theorem.

\subsection{Canonical identifications and the Margulis invariant}
\label{sec:canonical}
The main goal of this subsection is to associate to every map~$g \in G \ltimes V$ of type~$X_0$ a vector in~$V^\transl_0$, called its ``Margulis invariant'' (see Definition~\ref{margulis_invariant}). The two propositions (\ref{pair_transitivity} and~\ref{canonical_identification}) and and the lemma (\ref{projections_commute}) that lead up to this definition are important as well, and will be often used subsequently.

Corollary~\ref{dynamical_spaces_description} has shown us that the ``geometry'' of any map~$g$ of type~$X_0$ (namely the position of its dynamical spaces) is entirely determined by the pair of spaces
\[(A^\subge_g, A^\suble_g) = \phi(A^\subge_0, A^\suble_0).\]
In fact, such pairs of spaces play a crucial role. Let us begin with a definition; its connection with the observation we just made will become clear after Proposition~\ref{pair_transitivity}.

~
\begin{definition}~
\begin{itemize}
\item We define a \emph{parabolic space} to be any subspace of $V$ that is the image of~$V^\subge_0$ by some element of~$G$.
\item We define an \emph{affine parabolic space} to be any subspace of $A$ that is the image of~$A^\subge_0$ by some element of~$G \ltimes V$.
\item We say that two parabolic spaces (or two affine parabolic spaces) are \emph{transverse} if their intersection has the lowest possible dimension.
\end{itemize}
\end{definition}
\begin{remark}~
\begin{itemize}
\item Since $X_0$ is symmetric, $V^\suble_0$ (resp. $A^\suble_0$) is in particular a parabolic space (resp. an affine parabolic space).
\item A subspace $A^\subge \subset A$ is an affine parabolic space if and only if it is not contained in $V$ and its linear part $V^\subge = A^\subge \cap V$ is a parabolic space.
\item Clearly $V^\subge_0$ and~$V^\suble_0$ are transverse, and so are $A^\subge_0$ and~$A^\suble_0$. So two parabolic spaces (resp. affine parabolic spaces) are transverse if and only if their intersection has the same dimension as~$V^\sube_0$ (resp. $A^\sube_0$).
\end{itemize}
\end{remark}
\begin{example}~
\label{transverse_parabolic_example}
\begin{enumerate}
\item For $G = \SO^+(p, q)$ acting on~$V = \mathbb{R}^{p+q}$ (let us assume~$p \geq q$), a subspace $F \subset \mathbb{R}^{p+q}$ is a parabolic space if and only if $F^\perp$~is a maximal isotropic subspace. Equivalently, $F$~is a parabolic space if and only if $F$~contains~$F^\perp$ and is minimal for that property (namely $p$-dimensional). Two parabolic spaces are transverse if and only if their intersection has dimension~$p-q$. Pairs of transverse parabolic spaces were called \emph{frames} in~\cite{Smi13}.
\item If $G$~is any semisimple real Lie group acting on~$V = \mathfrak{g}$ (its Lie algebra) by the adjoint representation, a parabolic space is just an arbitrary minimal parabolic subalgebra of~$\mathfrak{g}$ (hence the name ``parabolic space'').
\end{enumerate}
\end{example}
\begin{proposition}
\label{pair_transitivity}
A pair of parabolic spaces (resp. of affine parabolic spaces) is transverse if and only if it may be sent to~$(V^\subge_0, V^\suble_0)$ (resp. to~$(A^\subge_0, A^\suble_0)$) by some element of~$G$ (resp. of~$G \ltimes V$).
\end{proposition}
In particular, it follows from Proposition~\ref{Asubge_is_ampa} that for any map~$g \in G \ltimes V$ of type~$X_0$, the pair~$(A^\subge_g, A^\suble_g)$ is a transverse pair of affine parabolic spaces.

This Proposition, as well as its proof, is very similar to Claim~2.8 in~\cite{Smi14}.
\begin{proof}
Let us prove the linear version; the affine version follows immediately. Let $(V_1, V_2)$ be any pair of parabolic spaces. By definition, for $i = 1, 2$, we may write $V_i = \phi_i(V^\subge_0)$ for some $\phi_i \in G$. Let us apply the Bruhat decomposition to the map~$\phi_1^{-1}\phi_2$: we may write
\begin{equation}
\phi_1^{-1}\phi_2 = p_1wp_2,
\end{equation}
where $p_1, p_2$ belong to the minimal parabolic subgroup $P^+$, and $w$ is an element of the restricted Weyl group $W$ (or, technically, some representative thereof). Let $\phi := \phi_1p_1 = \phi_2p_2^{-1}w^{-1}$; since $P^+$ stabilizes $V^\subge_0$, we have
\begin{equation}
V_1 = \phi(V^\subge_0) \;\text{ and }\; V_2 = \phi(w V^\subge_0).
\end{equation}
Thus $V_1$ and~$V_2$ are transverse if and only if $w V^\subge_0$ is transverse to $V^\subge_0$, which means that the dimension of their intersection, which is also equal to the sum of the multiplicities of restricted weights contained in the intersection
\[\Omega^\subge_{X_0} \cap w\Omega^\subge_{X_0},\]
is the smallest possible.

Clearly, this last intersection always contains~$\{0\}$. Since $X_0$~is generic, it can actually be equal to~$\{0\}$ if only we can choose~$w$ so as to have
\begin{equation}
\label{eq:w_acts_like_w0}
w \Omega^\subge_{X_0} = \Omega^\suble_{X_0}.
\end{equation}
Since $X_0$~is symmetric, this identity~\eqref{eq:w_acts_like_w0} \emph{is} realized in particular for $w = w_0$. This means that $V_1$ and~$V_2$ are transverse if and only if $w$~satisfies~\eqref{eq:w_acts_like_w0}, in which case we have indeed $V_1 = \phi(V^\subge_0)$ and $V_2 = \phi(V^\suble_0)$ as required.
\end{proof}

\begin{remark}~
\label{transversality_in_flag_variety}
\begin{itemize}
\item It follows from Proposition~\ref{stabg=stabge} that the set of all parabolic spaces can be identified with the flag variety~$G/P^+_{X_0}$, by identifying every parabolic space~$\phi(V^\subge_0)$ with the coset~$\phi P^+_{X_0}$.
\item In this interpretation, two parabolic spaces $V_1 = \phi_1(V^\subge_0)$ and~$V_2 = \phi_2(V^\subge_0) = \phi_2 \circ w_0(V^\suble_0)$ are then transverse if and only if the corresponding pair of cosets
\[(\phi_1 P^+_{X_0},\; \phi_2 w_0 P^-_{X_0})\]
is in the $G$-orbit of the point~$(P^+_{X_0}, P^-_{X_0})$ in~$G/P^+_{X_0} \times G/P^-_{X_0}$, also known as the \emph{open $G$-orbit} in~$G/P^+_{X_0} \times G/P^-_{X_0}$, since it can be shown that it is indeed the unique open $G$-orbit in that space.
\end{itemize}
\end{remark}

Consider a transverse pair of affine parabolic spaces. Their intersection may be seen as a sort of ``abstract affine neutral space''. We now introduce a family of ``canonical identifications'' between those spaces. Unfortunately, these identifications have an inherent ambiguity: they are only defined up to quasi-translation.
\begin{proposition}
\label{canonical_identification}
Let $(A_1, A_2)$ be a pair of transverse affine parabolic spaces. Then any map $\phi \in G \ltimes V$ such that $\phi(A_1, A_2) = (A^\subge_0, A^\suble_0)$ gives, by restriction, an identification of the intersection $A_1 \cap A_2$ with $A^\sube_0$, which is unique up to quasi-translation.
\end{proposition}
Here by $\phi(A_1, A_2)$ we mean the pair $(\phi(A_1), \phi(A_2))$. Note that if $A_1 \cap A_2$ is obtained in another way as an intersection of two affine parabolic spaces, the identification with~$A^\sube_0$ will, in general, no longer be the same, not even up to quasi-translation: there could also be an element of the Weyl group involved.

Compare this with Corollary~2.14 in~\cite{Smi14}.
\begin{proof}
The existence of such a map~$\phi$ follows from Proposition~\ref{pair_transitivity}. Now let $\phi$ and~$\phi'$ be two such maps, and let $f$~be the map such that
\begin{equation}
\phi' = f \circ \phi
\end{equation}
(\ie $f := \phi' \circ \phi^{-1}$). Then by construction $f$~stabilizes both $A^\subge_0$ and~$A^\suble_0$. It follows from Lemma~\ref{quasi-translation} that the restriction of~$f$ to~$A^\sube_0$ is a quasi-translation.
\end{proof}
Let us now explain why we call these identifications ``canonical''. The following lemma, while seemingly technical, is actually crucial: it tells us that the identifications defined in Proposition~\ref{canonical_identification} commute (up to quasi-translation) with the projections that naturally arise if we change one of the parabolic subspaces in the pair while fixing the other.
\begin{lemma}
\label{projections_commute}
Take any affine parabolic space~$A_1$.

Let $A_2$ and~$A'_2$ be any two affine parabolic spaces both transverse to~$A_1$.

Let~$\phi$ (respectively~$\phi'$) be an element of~$G \ltimes V$ that sends the pair of spaces $(A_1, A_2)$ (respectively~$(A_1, A'_2)$) to~$(A^\subge_0, A^\suble_0)$; these two maps exist by Proposition~\ref{pair_transitivity}.

Let~$W_1$ be the inverse image of~$V^\subg_0$ by any map~$\phi$ such that~$A_1 = \phi^{-1}(A^\subge_0)$ (this image is unique by Proposition~\ref{stabg=stabge}).

Let
\[\psi: A_1 \longlongrightarrow A_1 \cap A'_2\]
be the projection parallel to~$W_1$.

Then the map~$\overline{\psi}$ defined by the commutative diagram
\[
\begin{tikzcd}
  A^\sube_0
    \arrow{rr}{\overline{\psi}}
&
& A^\sube_0\\
\\
  A_1 \cap A_2
    \arrow{rr}{\psi}
    \arrow{uu}{\phi}
&
& A_1 \cap A'_2
    \arrow{uu}{\phi'}
\end{tikzcd}
\]
is a quasi-translation.
\end{lemma}
The space~$W_1$ is, in some sense, the ``abstract linear expanding space'' corresponding to the ``abstract affine noncontracting space''~$A_1$: more precisely, for any map~$g \in G \ltimes V$ of type~$X_0$ such that $A^\subge_g = A_1$, we have $V^\subg_g = W_1$ (by Proposition~\ref{Asubge_is_ampa}~(ii)).

The projection $\psi$ is well-defined because $A^\subge_0 = V^\subg_0 \oplus A^\sube_0 = V^\subg_0 \oplus (A^\subge_0 \cap A^\suble_0)$, and so $A_1 = \phi'^{-1}(A^\subge_0) = W_1 \oplus (A_1 \cap A'_2)$.

This statement generalizes Lemma~2.18 in~\cite{Smi14}. The proof is similar, but care must be taken to replace minimal parabolics by parabolics of type~$X_0$.
\begin{proof}
Without loss of generality, we may assume that $\phi = \Id$ (otherwise we simply replace the three affine parabolic spaces by their images under $\phi^{-1}$.) Then we have $A_1 = A^\subge_0$, $A_2 = A^\suble_0$ and $A'_2 = \phi'^{-1}(A^\suble_0)$, where $\phi'$ can be any map stabilizing the space~$A^\subge_0$. We want to show that the map $\overline{\psi} = \phi' \circ \psi$ (considered as a map from~$A^\sube_0$ to itself) is a quasi-translation.

We know that $\phi'$ lies in the stabilizer $\Stab_{G \ltimes V}(A^\subge_0)$; by Proposition~\ref{stabg=stabge}, the latter is equal to $P_{X_0}^+ \ltimes V^\subge_0$. We now introduce the algebra
\begin{equation}
\mathfrak{n}_{X_0}^+ := \bigoplus_{\alpha(X_0) > 0} \mathfrak{g}^\alpha
\end{equation}
and the group $N_{X_0}^+ := \exp \mathfrak{n}_{X_0}^+$. We then have the Langlands decomposition
\begin{equation}
P_{X_0}^+ = L_{X_0}N_{X_0}^+
\end{equation}
(see \eg \cite{Kna96}, Proposition 7.83). Since $L_{X_0}$ stabilizes $V^\subg_0$, this generalizes to the ``affine Langlands decomposition''
\begin{equation}
P_{X_0}^+ \ltimes V^\subge_0 = (L_{X_0} \ltimes V^\sube_0)(N_{X_0}^+ \ltimes V^\subg_0).
\end{equation}
Thus we may write $\phi' = l \circ n$ with $l \in L_{X_0} \ltimes V^\sube_0$ and $n \in N_{X_0}^+ \ltimes V^\subg_0$.

We shall use the following fact: every element $n$ of the group $N_{X_0}^+ \ltimes V^\subg_0$ stabilizes the space $V^\subg_0$ and induces the identity map on the quotient space $A^\subge_0 / V^\subg_0$. Indeed, when the element~$n$ lies in~$N_{X_0}^+$, since $N_{X_0}^+$~is connected, this follows from the fact that $\mathfrak{n}_{X_0}^+ \cdot V^\subge_0 \subset V^\subg_0$ (which, in turn, follows from the obvious fact that if $\lambda(X_0) \geq 0$ and~$\alpha(X_0) > 0$, then~$(\lambda + \alpha)(X_0) > 0$). When $n$ is a pure translation by a vector of $V^\subg_0$, this is obvious.

By definition, $\psi$ also stabilizes $V^\subg_0$ and induces the identity on $A^\subge_0 / V^\subg_0$; hence so does the map $n \circ \psi$. But we also know that $n \circ \psi$ is defined on $A_1 \cap A_2 = A^\sube_0$, and sends it onto
\[n \circ \psi(A_1 \cap A_2) = n(A_1 \cap A'_2) = l^{-1}(A^\sube_0) = A^\sube_0.\]
Hence the map $n \circ \psi$ is the identity on $A^\sube_0$. It follows that $\overline{\psi} = \phi' \circ \psi = l \circ n \circ \psi = l$ (in restriction to $A^\sube_0$); by Lemma~\ref{quasi-translation}, $\overline{\psi}$ is a quasi-translation as required.
\end{proof}

Now let $g$~be a map of type~$X_0$. We already know that it acts on its neutral affine space by quasi-translation; now the canonical identifications we have just introduced allow us to compare the actions of different elements on their respective neutral affine spaces, as if they were both acting on the same space~$A^\sube_0$. However there is a catch: since the identifications are only canonical up to quasi-translation, we lose information about the rotation part; only the translation part along~$V^\transl_0$ remains.

Formally, we make the following definition. Let $\pi_\transl$~denote the projection from~$V^\sube_0$ onto~$V^\transl_0$ parallel to~$V^\rotat_0$.
\begin{definition}
\label{margulis_invariant}
Let $g \in G \ltimes V$ be a map of type~$X_0$. Take any point $x$ in the affine space $A^{\sube}_{g} \cap V_{\Aff}$ and any map $\phi \in G$ such that $\phi(V^{\subge}_{g}, V^{\suble}_{g}) = (V^{\subge}_0, V^{\suble}_0)$. Then the vector
\[M(g) := \pi_\transl(\phi(g(x)-x)) \in V^\transl_0.\]
is called the \emph{Margulis invariant} of $g$.
\end{definition}
This vector does not depend on the choice of~$x$ or~$\phi$: indeed, composing~$\phi$ with a quasi-translation does not change the $V^\transl_0$-component of the image. See Proposition~2.16 in~\cite{Smi14} for a detailed proof of this claim (for $V = \mathfrak{g}$).

\section{Quantitative properties}
\label{sec:quantitative}
In this section, we define and study two important quantitative properties of maps of type~$X_0$:
\begin{itemize}
\item $C$-non-degeneracy, which means that the geometry of the map is not too close to a degenerate case;
\item and contraction strength, which measures the extent to which the map~$g$ is ``much more contracting'' on its contracting space than on its affine nonexpanding space.
\end{itemize}

In Subsection~\ref{sec:metric}, we define these and several other quantitative properties. Several definitions coincide with those from Section~2.6 in~\cite{Smi14} or generalize them.

In the very short Subsection~\ref{sec:affine_to_linear} (which is a straightforward generalization of Section~2.7 from~\cite{Smi14}), we compare these properties for an affine map and its linear part.

In Subsection~\ref{sec:proximal_product}, we define analogous quantitative properties for proximal maps, and relate properties of a product of a (sufficiently contracting and nondegenerate) pair of proximal maps to the properties of the factors. This is almost the same thing as Section~3.1 in~\cite{Smi14}, but with one additional result.

\subsection{Definitions}
\label{sec:metric}

We endow the extended affine space~$A$ with a Euclidean norm (written simply~$\| \cdot \|$) whose restriction to~$V$ coincides with the norm~$B$ defined in Lemma~\ref{K-invariant} and that makes $p_0$ orthogonal to~$V$. Then the subspaces $V^\subg_0$, $V^\subl_0$, $V^\rotat_0$, $V^\transl_0$ and $\mathbb{R} p_0$ are pairwise orthogonal, and the restriction of this norm to $V^\rotat_0$ is invariant by quasi-translations. For any linear map $g$ acting on~$A$, we write $\|g\| := \sup_{x \neq 0} \frac{\|g(x)\|}{\|x\|}$ its operator norm.

Consider a Euclidean space $E$ (for the moment, the reader may suppose that $E = A$; later we will also need the case $E = \ext^p A$ for some integer $p$). We introduce on the projective space $\mathbb{P}(E)$ a metric by setting, for every $\overline{x}, \overline{y} \in \mathbb{P}(E)$,
\begin{equation}
\alpha (\overline{x}, \overline{y}) := \arccos \frac{| \langle x, y \rangle |}{\|x\| \|y\|} \in \textstyle [0, \frac{\pi}{2}],
\end{equation}
where $x$ and $y$ are any vectors representing respectively $\overline{x}$ and $\overline{y}$ (obviously, the value does not depend on the choice of $x$ and $y$). This measures the angle between the lines $\overline{x}$ and $\overline{y}$. For shortness' sake, we will usually simply write $\alpha(x, y)$ with $x$ and $y$ some actual vectors in $E \setminus \{0\}$.

For any vector subspace $F \subset E$ and any radius $\eps > 0$, we shall denote the $\eps$-neighborhood of $F$ in $\mathbb{P}(E)$ by:
\begin{equation}
B_{\mathbb{P}}(F, \eps) := \setsuch{x \in \mathbb{P}(E)}{\alpha(x,\mathbb{P}(F)) < \eps}.
\end{equation}
(You may think of it as a kind of ``conical neighborhood''.)

Consider a metric space $(\mathcal{M}, \delta)$; let $X$ and $Y$ be two subsets of $\mathcal{M}$. We shall denote the ordinary, minimum distance between $X$ and $Y$ by
\begin{equation}
\delta(X, Y) := \inf_{x \in X} \inf_{y \in Y} \delta(x, y),
\end{equation}
as opposed to the Hausdorff distance, which we shall denote by
\begin{equation}
\delta^\mathrm{Haus}(X, Y) := \max\left( \sup_{x \in X} \delta\big(\{x\}, Y\big),\; \sup_{y \in Y} \delta\big(\{y\}, X\big) \right).
\end{equation}

Finally, we introduce the following notation. Let $X$ and $Y$ be two positive quantities, and $p_1, \ldots, p_k$ some parameters. Whenever we write
\[X \lesssim_{p_1, \ldots, p_k} Y,\]
we mean that there is a constant $K$, depending on nothing but $p_1, \ldots, p_k$, such that $X \leq KY$. (If we do not write any subscripts, this means of course that $K$ is an ``absolute'' constant --- or at least, that it does not depend on any ``local'' parameters; we consider the ``global'' parameters such as the choice of $G$ and of the Euclidean norms to be fixed once and for all.) Whenever we write
\[X \asymp_{p_1, \ldots, p_k} Y,\]
we mean that $X \lesssim_{p_1, \ldots, p_k} Y$ and $Y \lesssim_{p_1, \ldots, p_k} X$ at the same time.

\begin{definition}
\label{regular_definition}
Take a pair of affine parabolic spaces $(A_1, A_2)$. An \emph{optimal canonizing map} for this pair is a map $\phi \in G \ltimes V$ satisfying
\[\phi(A_1, A_2) = (A^\subge_0, A^\suble_0)\]
and minimizing the quantity $\max \left( \|\phi\|, \|\phi^{-1}\| \right)$. By Proposition~\ref{pair_transitivity} and a compactness argument, such a map exists if and only if $A_1$ and $A_2$ are transverse.

We define an \emph{optimal canonizing map} for a map $g \in G \ltimes V$ of type~$X_0$ to be an optimal canonizing map for the pair $(A^\subge_g, A^\suble_g)$.

Let $C \geq 1$. We say that a pair of affine parabolic spaces $(A_1, A_2)$ (resp. a map $g$ of type~$X_0$) is \emph{$C$-non-degenerate} if it has an optimal canonizing map $\phi$ such that
\[\left \|\phi \right\| \leq C \quad\textnormal{and}\quad \left \|\phi^{-1} \right\| \leq C.\]

Now take $g_1$, $g_2$ two maps of type~$X_0$ in $G \ltimes V$. We say that the pair $(g_1, g_2)$ is \emph{$C$-non-degenerate} if every one of the four possible pairs $(A^{\subge}_{g_i}, A^{\suble}_{g_j})$ is $C$-non-degenerate.
\end{definition}

The point of this definition is that there are a lot of calculations in which, when we treat a $C$-non-degenerate pair of spaces as if they were perpendicular, we err by no more than a (multiplicative) constant depending on $C$. The following result will often be useful:
\begin{lemma}
\label{bounded_norm_is_bilipschitz}
Let $C \geq 1$. Then any map $\phi \in \GL(E)$ such that $\|\phi^{\pm 1}\| \leq C$ induces a $C^2$-Lipschitz continuous map on $\mathbb{P}(E)$.
\end{lemma}
This is exactly Lemma~2.20 from~\cite{Smi14}.

\begin{remark}
\label{unified_treatment}
The set of transverse pairs of extended affine spaces is characterized by two open conditions: there is of course transversality of the spaces, but also the requirement that each space not be contained in $V$. What we mean here by ``degeneracy'' is failure of one of these two conditions. Thus the property of a pair $(A_1, A_2)$ being $C$-non-degenerate actually encompasses two properties.

First, it implies that the spaces $A_1$ and $A_2$ are transversal in a quantitative way. More precisely, this means that some continuous function that would vanish if the spaces were not transversal is bounded below. An example of such a function is the smallest non identically vanishing of the ``principal angles'' defined in the proof of Lemma~\ref{regular_to_proximal}~(iv).

Second, it implies that both $A_1$ and $A_2$ are ``not too close'' to the space $V$ (in the same sense). In purely affine terms, this means that the affine spaces $A_1 \cap V_{\Aff}$ and $A_2 \cap V_{\Aff}$ contain points that are not too far from the origin.

Both conditions are necessary, and appeared in the previous literature (such as \cite{Mar87} and~\cite{AMS02}). However, they were initially treated separately. The idea of encompassing both in the same concept of ``$C$-non-degeneracy'' seems to have been first introduced in the author's previous paper~\cite{Smi14}.
\end{remark}

\begin{definition}
\label{gaps}
Let $g \in \GL(E)$, let $n = \dim E$, and let $p$ be an integer such that $1 \leq p < n$. Let $\lambda_1, \ldots, \lambda_n$ be the eigenvalues of~$g$ ordered by nondecreasing modulus. Then we define the \emph{$p$-th spectral gap} of~$g$ to be the quotient
\begin{equation}
\kappa_p(g) := \frac{|\lambda_{p+1}|}{|\lambda_{p}|}.
\end{equation}
Note that we chose the convention where the gap is a number \emph{smaller} than or equal to~$1$. 

When $E = A$, we will most often use the $p$-th spectral gap for~$p = \dim A^\subge_0$. In this case we will omit the index:
\begin{equation}
\kappa(g) := \kappa_{\dim A^\subge_0}(g).
\end{equation}

Also, we denote the \emph{spectral radius} of~$g$, \ie the largest modulus of any eigenvalue, by:
\begin{equation}
r(g) := |\lambda_1|.
\end{equation}
(The usual notation, $\rho(g)$, is already taken to mean ``$g$~in~the representation~$\rho$''.)
\end{definition}

\begin{definition}
\label{s_definition}
Let $s > 0$. For a map $g \in G \ltimes V$ of type~$X_0$, we say that $g$ is \emph{$s$-contracting} if we have:
\begin{equation}
\forall (x, y) \in V^{\subl}_{g} \times A^{\subge}_{g}, \quad
  \frac{\|g(x)\|}{\|x\|} \leq s\frac{\|g(y)\|}{\|y\|}.
\end{equation}
(Note that by Corollary~\ref{dynamical_spaces_description} the spaces $V^{\subl}_{g}$ and $A^{\subge}_{g}$ always have the same dimensions as $V^\subl_0$ and $A^\subge_0$ respectively, hence they are nonzero.)

We define the \emph{strength of contraction} of $g$ to be the smallest number $s(g)$ such that $g$ is $s(g)$-contracting. In other words, we have
\begin{equation}
s(g) =
\left\| \restr{g}     {V^{\subl}_{g}}  \right\|
\left\| \restr{g^{-1}}{A^{\subge}_{g}} \right\|.
\end{equation}
\end{definition}

\begin{remark}
\label{blabla_contraction_strength}
This strength of contraction~$s(g)$ is defined as a kind of ``mixed gap'': it measures the gap between \emph{singular} values of the restrictions of~$g$ to some sums of its \emph{eigen}spaces. It turns out that this definition is the most convenient for our purposes.

However, if the map~$g$ from the above definition is $C$-non-degenerate, then we may pretend that $s(g)$~is a ``purely singular'' gap, as long as we do not care about multiplicative constants. Indeed, let $g' = \phi g \phi^{-1}$, where $\phi$~is an optimal canonizing map for~$g$; then it is easy to see that we have
\begin{equation}
s(g) \asymp_C s(g').
\end{equation}
On the other hand, since $V^\subl_{g'} = V^\subl_0$ and $A^\subge_{g'} = A^\subge_0$ are orthogonal (by convention), every singular value of~$g'$ is either a singular value of~$\restr{g'}{V^\subl_0}$ or of~$\restr{g'}{A^\subge_0}$. It follows that~$s(g')$ is the quotient between two actual singular values of~$g'$, and two \emph{consecutive} singular values if~$s(g)$ is small enough. See the proof of Lemma~\ref{regular_definition}~(iii) for a more detailed discussion.
\end{remark}

\begin{remark}
\label{contraction_strength_grows}
The spectral gap and contraction strength are somewhat related. Take some affine map~$g \in G \ltimes V$ of type~$X_0$; then since the norm of any linear map is at least equal to its spectral radius, we obviously have
\begin{equation}
s(g) \geq \kappa(g).
\end{equation}
On the other hand, for any map $g \in G \ltimes V$, we have
\begin{equation}
\log s(g^N) = N \log \kappa(g) + \underset{N \to \infty}{\bigo}(\log N).
\end{equation}
If $g$~is of type~$X_0$, then $\kappa(g) < 1$, so that
\begin{equation}
s(g^N) \underset{N \to \infty}{\to} 0.
\end{equation}
\end{remark}

\subsection{Affine and linear case}
\label{sec:affine_to_linear}

For any map $f \in G \ltimes V$, we denote by $\ell(f)$ the linear part of $f$, seen as an element of~$G \ltimes V$ by identifying $G$ with the stabilizer of the ``origin'' $p_0$. In other words, for every vector~$(x, t) \in V \oplus \mathbb{R} p_0 = A$, we set
\begin{equation}
\ell(f)(x, t) = f(x, 0) + (0, t).
\end{equation}
(Seeing $G$ as a subgroup of $G \ltimes V$ allows us to avoid introducing new definitions of $C$-non-degeneracy and contraction strength for elements of $G$.)
\begin{lemma}
\label{affine_to_vector}
Let $C \geq 1$, and take any $C$-non-degenerate map $g$ (or $C$-non-degenerate pair of maps $(g, h)$) of type~$X_0$ in $G \ltimes V$. Then:
\begin{hypothenum}
\item The map $\ell(g)$ (resp. the pair $(\ell(g), \ell(h))$) is still $C$-non-degenerate;
\item We have $s(\ell(g)) \leq s(g)$;
\item Suppose that $s(g^{-1}) \leq 1$. Then we actually have $s(g) \asymp_C s(\ell(g)) \left\| \restr{g}{A^\sube_g} \right\|$.
\end{hypothenum}
\end{lemma}
\begin{proof} The proof is exactly the same as the proof of Lemma~2.25 in~\cite{Smi14}, \emph{mutatis mutandis}.
\end{proof}

\subsection{Proximal maps}
\label{sec:proximal_product}
Let $E$ be a Euclidean space. The goal of this section is to show Proposition~\ref{proximal_product}. We begin with a few definitions.

\begin{definition}
\label{proximal_definition}
Let $\gamma \in \GL(E)$; let $\lambda_1, \ldots, \lambda_n$ be its eigenvalues repeated according to multiplicity and ordered by nonincreasing modulus. We define the \emph{proximal spectral gap} of~$\gamma$ as its first spectral gap:
\[\tilde{\kappa}(\gamma) := \kappa_1(\gamma) = \frac{|\lambda_2|}{|\lambda_1|}.\]
We say that $\gamma$ is \emph{proximal} if $\tilde{\kappa}(\gamma) < 1$. We may then decompose $E$ into a direct sum of a line $E^s_\gamma$, called its \emph{attracting space}, and a hyperplane $E^u_\gamma$, called its \emph{repelling space}, both stable by $\gamma$ and such that:
\[\begin{cases}
\restr{\gamma}{E^s_\gamma} = \lambda_1 \Id; \\
\text{for every eigenvalue } \lambda \text{ of } \restr{\gamma}{E^u_\gamma},\; |\lambda| < |\lambda_1|.
\end{cases}\]
\end{definition}

\begin{definition}
Consider a line $E^s$ and a hyperplane $E^u$ of $E$, transverse to each other. An \emph{optimal canonizing map} for the pair $(E^s, E^u)$ is a map $\phi \in GL(E)$ satisfying
\[\phi(E^s) \perp \phi(E^u)\]
and minimizing the quantity $\max \left( \|\phi\|, \|\phi^{-1}\| \right)$.

We define an \emph{optimal canonizing map} for a proximal map $\gamma \in \GL(E)$ to be an optimal canonizing map for the pair $(E^s_\gamma, E^u_\gamma)$.

Let $C \geq 1$. We say that the pair formed by a line and a hyperplane $(E^s, E^u)$ (resp. that a proximal map $\gamma$) is \emph{$C$-non-degenerate} if it has an optimal canonizing map $\phi$ such that $\left \|\phi^{\pm 1} \right\| \leq C$.

Now take $\gamma_1, \gamma_2$ two proximal maps in $\GL(E)$. We say that the pair $(\gamma_1, \gamma_2)$ is \emph{$C$-non-degenerate} if every one of the four possible pairs $(E^s_{\gamma_i}, E^u_{\gamma_j})$ is $C$-non-degenerate.
\end{definition}

\begin{definition}
\label{s_tilde_definition}
Let $\gamma \in \GL(E)$ be a proximal map. We define the \emph{proximal strength of contraction} of $\gamma$ by
\[\tilde{s}(\gamma)
:= \frac{\left\| \restr{\gamma}{E^u_\gamma} \right\|}{\left\| \restr{\gamma}{E^s_\gamma} \right\|}
= \frac{\left\| \restr{\gamma}{E^u_\gamma} \right\|}{r(\gamma)}\]
(where $r(\gamma)$ is the spectral radius of~$\gamma$, equal to~$|\lambda_1|$ in the notations of the previous definition). We say that $\gamma$ is \emph{$\tilde{s}$-contracting} if $\tilde{s}(\gamma) \leq \tilde{s}$.
\end{definition}

Be careful that the meaning of some of these terms changes depending on the context: they mean different things for maps of type~$X_0$ (see Definitions~\ref{gaps} and~\ref{s_definition} above) and for proximal maps. We tried to at least keep the notations unambiguous: compare the definitions of $\tilde{s}$ and~$\tilde{\kappa}$ with those of $s$ and~$\kappa$.

In the following proposition, the notation~$\tilde{s}_{\ref{proximal_product}}(C)$ might puzzle the reader. This constant is in fact indexed by the number of the proposition where it appears, a convention that we will follow throughout the paper.

\begin{proposition}
\label{proximal_product}
For every $C \geq 1$, there is a positive constant $\tilde{s}_{\ref{proximal_product}}(C)$ with the following property. Take a $C$-non-degenerate pair of proximal maps $\gamma_1, \gamma_2$ in~$\GL(E)$, and suppose that both $\gamma_1$ and $\gamma_2$ are $\tilde{s}_{\ref{proximal_product}}(C)$-contracting. Then $\gamma_1 \gamma_2$~is proximal, and we have:
\begin{hypothenum}
\item $\alpha \left(E^s_{\gamma_1 \gamma_2},\; E^s_{\gamma_1} \right) \lesssim_C \tilde{s}(\gamma_1)$;
\item $\tilde{s}(\gamma_1 \gamma_2) \lesssim_C \tilde{s}(\gamma_1)\tilde{s}(\gamma_2)$.
\item $r(\gamma_1 \gamma_2) \asymp_C \|\gamma_1\| \|\gamma_2\|$.
\end{hypothenum}
\end{proposition}
Note that since we have $r(\gamma_1 \gamma_2) \leq \|\gamma_1 \gamma_2\| \leq \|\gamma_1\| \|\gamma_2\|$, what the last point really says is that \emph{all three} values have the same order of magnitude.

Similar results have appeared in the literature for a long time, \eg Lemma~5.7 in~\cite{AMS02}, Proposition~6.4 in~\cite{Ben96} or Lemma~2.2.2 in~\cite{Ben97}.

\begin{proof}
The first two points have already been proved in the author's previous paper: see Proposition~3.4 in~\cite{Smi14}. To prove~(iii), we start with the following observation. Let~$\eta = \frac{\pi}{2 C^2}$; then by Lemma~\ref{bounded_norm_is_bilipschitz} we have:
\[\alpha(E^s_{\gamma_1}, E^u_{\gamma_2}) \geq \eta.\]
On the other hand, we have already seen in the proof of Proposition~3.4 in~\cite{Smi14} that we have
\[E^s_{\gamma_1 \gamma_2} \in B(E^s_{\gamma_1}, \textstyle \frac{\eta}{3}).\]
The triangular inequality immediately gives us
\begin{equation}
\label{eq:gamma12gamma2separation}
\alpha \left( E^s_{\gamma_1 \gamma_2}, E^u_{\gamma_2} \right) \geq \frac{2 \eta}{3}.
\end{equation}
Take any nonzero $x \in E^s_{\gamma_1 \gamma_2}$. We are going to show the estimates
\begin{subequations}
\label{eq:norm_estimate}
  \begin{equation}
  \label{eq:norm_estimate_1}
    \frac{\|\gamma_2(x)\|}{\|x\|} \asymp_C \|\gamma_2\|;
  \end{equation}
  \begin{equation}
  \label{eq:norm_estimate_2}
    \frac{\|\gamma_1(\gamma_2(x))\|}{\|\gamma_2(x)\|} \asymp_C \|\gamma_1\|.
  \end{equation}
\end{subequations}
Since by definition, we have $\gamma_1(\gamma_2(x)) = \lambda x$ for some $\lambda \in \mathbb{R}$ having absolute value~$r(\gamma_1 \gamma_2)$, the estimate~(iii) follows by multiplying \eqref{eq:norm_estimate_1} and \eqref{eq:norm_estimate_2} together.

Let us first show~\eqref{eq:norm_estimate_1}. Let~$\phi$ be an optimal canonizing map for~$\gamma_2$; since~$\gamma_2$ is $C$-non-degenerate, we lose no generality by replacing $\gamma_2$ and~$x$ respectively by~$\gamma'_2 := \phi \gamma_2 \phi^{-1}$ and~$x' := \phi(x)$. Obviously we have:
\begin{equation}
\label{eq:gamma2x_majoration}
\|\gamma'_2(x')\| \leq \|\gamma'_2\|\|x'\|.
\end{equation}
To show the other inequality, let us decompose
\begin{equation}
x' =: \underbrace{x'_s}_{\in E^s_{\gamma'_2}}
   +  \underbrace{x'_u}_{\in E^u_{\gamma'_2}}.
\end{equation}
Then we have
\begin{equation}
\label{eq:gamma2x}
\|\gamma'_2(x')\| \geq \|\gamma'_2(x'_s)\| - \|\gamma'_2(x'_u)\|.
\end{equation}
For the first term, we have:
\begin{align}
\|\gamma'_2(x'_s)\| &= r(\gamma'_2) \cdot \|x'_s\| \nonumber \\
                    &= \|\gamma'_2\| \cdot
                       \sin \alpha \left( \phi(E^s_{\gamma_1 \gamma_2}), E^u_{\gamma'_2} \right) \cdot
                       \|x'\| \nonumber \\
                    &\geq \|\gamma'_2\| \cdot
                          \sin \frac{\alpha \left( E^s_{\gamma_1 \gamma_2}, E^u_{\gamma_2} \right)}{C^2} \cdot
                          \|x'\| & \text{ by Lemma~\ref{bounded_norm_is_bilipschitz}} \nonumber \\
                    &\geq \|\gamma'_2\| \cdot
                          \sin \frac{1}{C^2}\frac{2 \eta}{3} \cdot
                          \|x'\| & \text{ by~\eqref{eq:gamma12gamma2separation}.}
\end{align}
For the second term, we have:
\begin{align}
\|\gamma'_2(x'_u)\| &\leq \left\| \restr{\gamma'_2}{E^u_{\gamma'_2}} \right\| \|x'_u\| \nonumber \\
                    &\leq \left\| \restr{\gamma'_2}{E^u_{\gamma'_2}} \right\| \|x'\| \nonumber \\
                    &= \|\gamma'_2\| \; \tilde{s}(\gamma'_2) \; \|x'\| \nonumber \\
                    &\leq \|\gamma'_2\| \; C^2\tilde{s}(\gamma_2) \; \|x'\|.
\end{align}
Plugging those two estimates into~\eqref{eq:gamma2x}, we obtain
\begin{equation}
\|\gamma'_2(x')\| \geq \|\gamma'_2\| \left( \sin \frac{2 \eta}{3 C^2} - C^2\tilde{s}(\gamma_2) \right) \|x'\|.
\end{equation}
We may assume that $\tilde{s}(\gamma_2) \leq \frac{1}{2}\frac{1}{C^2} \sin \frac{2 \eta}{3 C^2}$. Since by construction $\eta$~depends only on~$C$, we conclude that
\begin{equation}
\label{eq:gamma2x_minoration}
\|\gamma'_2(x')\| \gtrsim_C \|\gamma'_2\|\|x'\|.
\end{equation}
Putting together \eqref{eq:gamma2x_majoration} and \eqref{eq:gamma2x_minoration}, we get \eqref{eq:norm_estimate_1} as required.

Now to show~\eqref{eq:norm_estimate_2}, simply notice that
\begin{equation}
\gamma_2(E^s_{\gamma_1 \gamma_2}) = E^s_{\gamma_2 \gamma_1}
\end{equation}
(since~$\gamma_2 \gamma_1$ is the conjugate of~$\gamma_1 \gamma_2$ by~$\gamma_2$), so that $\gamma_2(x) \in E^s_{\gamma_2 \gamma_1}$. Hence we may follow the same reasoning as for~\eqref{eq:norm_estimate_1}, simply exchanging the roles of $\gamma_1$ and~$\gamma_2$.
\end{proof}

\section{Additivity of Jordan projections}
\label{sec:jordan_additivity}

The goal of this section is to prove Proposition~\ref{regular_product_qualitative}, which says that the product of two sufficiently contracting maps of type~$X_0$ and in general position is still of type~$X_0$. As it is a purely linear property, we forget about translation parts and work exclusively in the linear group~$G$ for the duration of this section. We proceed in four stages.

We start with Proposition~\ref{type_X0_to_proximality}, which shows that if an element of~$G$ is of type~$X_0$ and strongly contracting in the default representation~$\rho$, it is proximal and strongly contracting in some of the fundamental representations $\rho_i$ defined in Proposition~\ref{fundamental_real_representation}.

We continue with Proposition~\ref{C-non-deg-in-Vi}, which relates $C$-non-degeneracy in~$V$ and $C'$-non-degeneracy in the spaces~$V_i$.

We then prove Proposition~\ref{jordan_additivity} (and a reformulated version, Corollary~\ref{jordan_additivity_reformulation}), which constrains the Jordan projection of~$gh$ in terms of the Cartan projections of~$g$ and~$h$.

Finally, we use Corollary~\ref{jordan_additivity_reformulation} to prove Proposition~\ref{regular_product_qualitative}.

\begin{proposition}
\label{type_X0_to_proximality}
For every~$C \geq 1$, there is a positive constant $s_{\ref{type_X0_to_proximality}}(C)$ with the following property.
Let $g \in G$ be a $C$-non-degenerate map of type~$X_0$ such that $s(g) \leq s_{\ref{type_X0_to_proximality}}(C)$. Then for every~$i \in \Pi \setminus \Pi_{X_0}$, the map~$\rho_i(g)$ is proximal and we have
\[\tilde{s}(\rho_i(g)) \lesssim_C s(g).\]
\end{proposition}
\begin{remark}~
\begin{itemize}
\item Note that since all Euclidean norms on a finite-dimensional vector space are equivalent, this estimate makes sense even though we did not specify any norm on~$V_i$. In the course of the proof, we shall choose one that is convenient for us.
\item Recall that ``$i \in \Pi \setminus \Pi_{X_0}$'' is a notation shortcut for ``$i$ such that $\alpha_i \in \Pi \setminus \Pi_{X_0}$''.
\end{itemize}
\end{remark}
\begin{remark}
Note that we have excluded the indices~$i$ that lie in~$\Pi_{X_0}$. The latter should be thought of as a kind of ``exceptional set''; indeed, recall (Remark~\ref{classification_of_sets_Pi_X}) that it is often empty.
\end{remark}

To pave the way for proving the proposition, let us prove a few lemmas that lead to a relation between the contraction strength of an element of~$G$ and its Cartan projection.

\begin{lemma}
\label{de_part_et_d_autre}
For every $i \in \Pi \setminus \Pi_{X_0}$, we may find two restricted weights~$\lambda^\subge_i \in \Omega^\subge_{X_0}$ and~ $\lambda^\subl_i \in \Omega^\subl_{X_0}$ such that
\[\lambda^\subge_i - \lambda^\subl_i = \alpha_i.\]
\end{lemma}
(Recall that $\Omega^\subge_{X_0}$~is the set of restricted weights that take nonnegative values on~$X_0$, and~$\Omega^\subl_{X_0}$ is its complement in~$\Omega$.)
\begin{proof}
Fix some $i \in \Pi \setminus \Pi_{X_0}$. Since $X_0$~is extreme, $s_{\alpha_i}(X_0)$~then does not have the same type as~$X_0$. Since $X_0$~is generic, we may then find a restricted weight~$\lambda$ of~$\rho$ such that
\begin{equation}
\lambda(X_0) > 0 \quad\text{ and }\quad s_{\alpha_i}(\lambda)(X_0) < 0
\end{equation}
(we already made this observation in~\eqref{eq:transgressing_weight}). Since $\lambda$ is a restricted weight, by Proposition~\ref{restr_weight_lattice}, the number
\begin{equation}
n_\lambda := \frac{\langle \lambda, \alpha_i \rangle}{2\langle \alpha_i, \alpha_i \rangle}
\end{equation}
is an integer. We have, on the one hand:
\[n_\lambda \alpha_i (X_0) = \left( \lambda - s_{\alpha_i}(\lambda) \right) (X_0) > 0;\]
on the other hand, $\alpha_i(X_0) \geq 0$ (because $X_0 \in \mathfrak{a}^{+}$); hence $n_\lambda$~is positive.

By Proposition~\ref{convex_hull}, every element of the sequence
\[\lambda,\; \lambda - \alpha_i,\; \ldots,\; \lambda - n_\lambda \alpha_i\]
is a restricted weight of~$\rho$. We may then simply take $\lambda^\subge_i$ to be the last term of this sequence that still belongs to~$\Omega^\subge_{X_0}$, and take $\lambda^\subl_i := \lambda^\subge_i - \alpha_i$ to be the immediately following term of the sequence.
\end{proof}

\begin{lemma}[Cartan decomposition in~$L_{X_0}$]
\label{consistent_Cartan_decomposition}
Let $g \in L_{X_0}$. Then there exist two elements $k_1$ and~$k_2$ in $K \cap L_{X_0}$ and a unique element $\cartan_{X_0}(g) \in \mathfrak{a}^+_{\Pi_{X_0}}$ such that
\[g = k_1 \exp(\cartan_{X_0}(g)) k_2.\]
\end{lemma}
(Recall \eqref{eq:levi_weyl_chamber_definition} that $\mathfrak{a}^+_{\Pi_{X_0}} = \setsuch{X \in \mathfrak{a}}{\forall \alpha \in \Pi_{X_0},\; \alpha(X) \geq 0}$.)

\begin{proof}
By Proposition~7.82~(a) in~\cite{Kna96}, $L_{X_0}$ is the centralizer of the intersection of the kernels of simple roots in~$\Pi_{X_0}$:
\begin{equation}
L_{X_0} = Z_G \Big( \setsuch{X \in \mathfrak{a}}{\forall \alpha \in \Pi_{X_0},\; \alpha(X) = 0} \Big).
\end{equation}
By Proposition~7.25 in~\cite{Kna96}, it follows:
\begin{itemize}
\item that $L_{X_0}$~is reductive;
\item that $K \cap L_{X_0}$ is a maximal compact subgroup in~$L_{X_0}$.
\end{itemize}
Obviously $\mathfrak{a} \subset \mathfrak{l}_{X_0}$ is a Cartan subspace of~$\mathfrak{l}_{X_0}$, and $\mathfrak{a}^+_{\Pi_{X_0}}$ is a Weyl chamber for~$L_{X_0}$. So this result is just the Cartan decomposition in the reductive group~$L_{X_0}$ (see Theorem~7.39 in~\cite{Kna96}).
\end{proof}

\begin{lemma}
\label{contraction_strength_and_cartan_projection}
For every $C \geq 1$, there is a constant~$k_{\ref{contraction_strength_and_cartan_projection}}(C)$ with the following property. Let $g \in G$ be a $C$-non-degenerate map of type~$X_0$ such that $\log s(g) \leq -k_{\ref{contraction_strength_and_cartan_projection}}(C)$. Then we have
\[\min_{\lambda \in \Omega^\subge_{X_0}} \lambda(\cartan(g)) \;-\;
\max_{\lambda \in \Omega^\subl_{X_0}} \lambda(\cartan(g))
  \;\;\geq\;\;
- \log s(g) - k_{\ref{contraction_strength_and_cartan_projection}}(C).\]
\end{lemma}
Note that the first term on the left-hand side is certainly nonpositive, since $0 \in \Omega^\subge_{X_0}$.
\begin{proof}
Let us first focus on the particular case where $g$~satisfies
\[\begin{cases}
V^{\subge}_{g} = V^{\subge}_0; \\
V^{\suble}_{g} = V^{\suble}_0.
\end{cases}\]
In this case, we shall prove that the statement holds with $k_{\ref{contraction_strength_and_cartan_projection}}(C) = 0$. In fact, we shall even prove that in this case, if $\log s(g) \leq 0$, we actually have the equality
\begin{equation}
\label{eq:sg'_majoration}
\min_{\lambda \in \Omega^\subge_{X_0}} \lambda(\cartan(g)) -
\max_{\lambda \in \Omega^\subl_{X_0}} \lambda(\cartan(g))
  \;=\;
- \log s(g).
\end{equation}

By construction, obviously $g$ stabilizes $V^\subge_0$ and~$V^\suble_0$; hence (using Proposition~\ref{stabg=stabge}) it also stabilizes $V^\subl_0$, and we have
\begin{equation}
g \in P_{X_0}^+ \cap P_{X_0}^- = L_{X_0}.
\end{equation}
By Lemma~\ref{consistent_Cartan_decomposition}, we then have
\begin{equation}
\label{eq:LX0_Cartan_decomposition_of_g}
g = k_1 \exp(\cartan_{X_0}(g)) k_2
\end{equation}
with $k_1, k_2 \in K \cap L_{X_0}$. In particular both $k_1$ and~$k_2$ stabilize both $V^\subge_0$ and~$V^\subl_0$. Hence so does the $L_{X_0}$-Cartan projection~$\cartan_{X_0}(g)$, and we have
\begin{equation}
\begin{cases}
\left\| \restr{g}{V^\subge_0} \right\| = \left\| \restr{\exp(\cartan_{X_0}(g))}{V^\subge_0} \right\|; \vspace{2mm} \\
\left\| \restr{(g)^{-1}}{V^\subl_0} \right\| = \left\| \restr{\exp(\cartan_{X_0}(g))^{-1}}{V^\subl_0} \right\|.
\end{cases}
\end{equation}
Now we know that $\exp(\cartan_{X_0}(g))$ (seen in the default representation~$\rho$) is self-adjoint (by choice of the Euclidean structure~$B$), hence its singular values coincide with its eigenvalues. (Moreover $V^\subge_0$ and $V^\subl_0$ are orthogonal.) As $\exp(\cartan_{X_0}(g)) \in A$, obviously it acts on every restricted weight space~$V^\lambda$ with the eigenvalue
\[\exp(\lambda(\cartan_{X_0}(g))).\]
This almost gives us the identity we want, but with $\cartan_{X_0}(g)$ instead of $\cartan(g)$:
\begin{equation}
\label{eq:sg'_general_majoration}
\min_{\lambda \in \Omega^\subge_{X_0}} \lambda(\cartan_{X_0}(g)) -
\max_{\lambda \in \Omega^\subl_{X_0}} \lambda(\cartan_{X_0}(g))
  \;=\;
- \log s(g).
\end{equation}
Note that, in contrast to the identity~\eqref{eq:sg'_majoration} that we are trying to prove, this identity holds for all values of~$s(g)$. To conclude, it remains to show that if we assume $\log s(g) \leq 0$, then we actually have $\cartan_{X_0}(g) = \cartan(g)$.

Indeed if $\log s(g) \leq 0$, then the left-hand side of~\eqref{eq:sg'_general_majoration} must be nonnegative. By Lemma~\ref{de_part_et_d_autre}, it then follows that in particular, we have
\begin{equation}
\forall i \in \Pi \setminus \Pi_{X_0},\quad \lambda^\subge_i(\cartan_{X_0}(g)) \geq \lambda^\subl_i(\cartan_{X_0}(g)),
\end{equation}
hence
\begin{equation}
\forall i \in \Pi \setminus \Pi_{X_0},\quad \alpha_i(\cartan_{X_0}(g)) \geq 0.
\end{equation}
On the other hand we also have
\begin{equation}
\forall i \in \Pi_{X_0},\quad \alpha_i(\cartan_{X_0}(g)) \geq 0,
\end{equation}
since $\cartan_{X_0}(g) \in \mathfrak{a}^+_{\Pi_{X_0}}$ by construction. Joining both systems of inequalities, we obtain that
\[\cartan_{X_0}(g) \in \mathfrak{a}^+.\]
This shows that \eqref{eq:LX0_Cartan_decomposition_of_g} actually gives a Cartan decomposition of~$g$ in the whole group~$G$. By uniqueness of Cartan projection, we conclude that $\cartan_{X_0}(g) = \cartan(g)$ as desired.

Now let us deal with arbitrary~$g$. Let $\phi$ be an optimal canonizing map for~$g$, and let $g' = \phi g \phi^{-1}$. Then it is easy to see that we have
\[s(g') \asymp_C s(g)\]
(we already mentioned this in Remark~\ref{blabla_contraction_strength}), and the difference $\cartan(g') - \cartan(g)$ is bounded by a constant that depends only on~$C$. Taking a suitable value of~$k_{\ref{contraction_strength_and_cartan_projection}}(C)$, the general result for~$g$ then follows from the particular result applied to~$g'$.
\end{proof}

\begin{proof}[Proof of Proposition~\ref{type_X0_to_proximality}]~
Let~$s_{\ref{type_X0_to_proximality}}(C)$ be a positive constant small enough to satisfy all the constraints that will appear in the course of the proof. Let us fix $i \in \Pi \setminus \Pi_{X_0}$, and let $g \in G$ be a map satisfying the hypotheses. Let us prove the two estimates
\begin{equation}
\label{eq:kappa_estimate}
\tilde{\kappa}(\rho_i(g)) = \exp(\alpha_i(\jordan(g)))^{-1} \leq \kappa(g),
\end{equation}
which will show that~$\rho_i(g)$ is proximal; and then the two estimates
\begin{equation}
\label{eq:s_estimate}
\tilde{s}(\rho_i(g)) \asymp_C \exp(\alpha_i(\cartan(g)))^{-1} \lesssim_C s(g),
\end{equation}
whose combination completes the proof.

\begin{itemize}
\item Let us start with the right part of~\eqref{eq:kappa_estimate}. Lemma~\ref{de_part_et_d_autre} gives us two restricted weights $\lambda^\subge_i$ and~$\lambda^\subl_i$ of~$\rho$ such that:
\[\begin{cases}
\lambda^\subge_i (X_0) \geq 0; \\
\lambda^\subl_i (X_0) < 0,
\end{cases}\]
and whose difference is~$\alpha_i$. Now since~$g$ is of type~$X_0$, by definition, any restricted weight of~$\rho$ has the same sign when evaluated at~$\jordan(g)$ or at~$X_0$. Thus we also have
\[\begin{cases}
\lambda^\subge_i (\jordan(g)) \geq 0; \\
\lambda^\subl_i (\jordan(g)) < 0.
\end{cases}\]

From Proposition~\ref{eigenvalues_and_singular_values_characterization}, it then follows that
\[\kappa(g) \geq \exp(\alpha_i(\jordan(g)))^{-1}\]
as desired.

\item Similarly we may establish the right part of~\eqref{eq:s_estimate}. By using once again the restricted weights $\lambda^\subge_i$ and~$\lambda^\subl_i$ given by Lemma~\ref{de_part_et_d_autre}, it follows from Lemma~\ref{contraction_strength_and_cartan_projection} that
\[\alpha_i(\cartan(g)) \geq - \log s(g) - k_{\ref{contraction_strength_and_cartan_projection}}(C),\]
provided we take $s_{\ref{type_X0_to_proximality}}(C) \leq \exp(-k_{\ref{contraction_strength_and_cartan_projection}}(C))$. By negating both sides and exponentiating, the desired estimate follows immediately.

\item Let us now prove the left part of~\eqref{eq:kappa_estimate}. By Proposition~\ref{eigenvalues_and_singular_values_characterization}~(i), the list of the moduli of the eigenvalues of~$\rho_i(g)$ is precisely
\[\left( e^{\lambda_i^j(\jordan(g))} \right)_{1 \leq j \leq d_i},\]
where $d_i$ is the dimension of~$V_i$ and~$(\lambda_i^j)_{1 \leq j \leq d_i}$ is the list of restricted weights of~$\rho_i$ repeated according to their multiplicity.

Up to reordering that list, we may suppose that
\[\lambda_i^1 = n_i \varpi_i\]
is the highest restricted weight of~$\rho_i$. We may also suppose that
\[\lambda_i^2 = n_i \varpi_i - \alpha_i;\]
indeed it is also a restricted weight of~$\rho_i$ by Lemma~\ref{fund_repr_other_weights}~(i). Now take any~$j > 2$. Since by hypothesis, the restricted weight~$n_i \varpi_i$ has multiplicity~$1$, we have $\lambda_i^j \neq \lambda_i^1$. By Lemma~\ref{fund_repr_other_weights}~(ii), it follows that this restricted weight has the form
\[\lambda_i^j = n_i \varpi_i - \alpha_i - \sum_{i' = 1}^r c_{i'} \alpha_{i'},\]
with $c_{i'} \geq 0$ for every index~$i'$.

Finally, since by definition $\jordan(g)$ lies in~$\mathfrak{a}^{+}$, for every index~$i'$ we have ${\alpha_{i'}(\jordan(g)) \geq 0}$. It follows that for every~$j > 2$, we have
\begin{equation}
\label{eq:highest_eigenvalue}
\lambda_i^1(\jordan(g)) \geq \lambda_i^2(\jordan(g)) \geq \lambda_i^j(\jordan(g)).
\end{equation}
In other words, among the moduli of the eigenvalues of~$\rho_i(g)$, the largest is
\[\exp(\lambda_i^1(\jordan(g))) = \exp(n_i \varpi_i(\jordan(g))),\]
and the second largest is
\[\exp(\lambda_i^2(\jordan(g))) = \exp(n_i \varpi_i(\jordan(g)) - \alpha_i(\jordan(g))).\]

It follows that
\[\tilde{\kappa}(\rho_i(g)) = \exp(\alpha_i(\jordan(g)))^{-1}\]
as desired.

\item Let us finish with the left part of~\eqref{eq:s_estimate}. We start with the following observation: for every~$C \geq 1$, the set
\begin{equation}
\setsuch{\phi \in G}{\|\phi\| \leq C,\; \|\phi^{-1}\| \leq C}
\end{equation}
is compact. It follows that the continuous map
\begin{equation}
\phi \mapsto \max \Big( \left\|\rho_i(\phi\vphantom{^{-1}})\right\|,\; \left\|\rho_i(\phi^{-1})\right\| \Big)
\end{equation}
is bounded on that set, by some constant~$C'_i$ that depends only on~$C$ (and on the choice of a norm on~$V_i$, to be made soon). Let~$\phi$ be the optimal canonizing map of~$g$, and let~$g' = \phi g \phi^{-1}$; then we get
\begin{equation}
\label{eq:srhoig_srhoig'}
\tilde{s}(\rho_i(g)) \asymp_C \tilde{s}(\rho_i(g')).
\end{equation}

Now let us choose, on the space~$V_i$ where the representation~$\rho_i$ acts, a $K$-invariant Euclidean form~$B_i$ such that all the restricted weight spaces for~$\rho_i$ are pairwise $B_i$\nobreakdash-\hspace{0pt}orthogonal (this is possible by Lemma~\ref{K-invariant} applied to~$\rho_i$). Then $\tilde{s}(\rho_i(g'))$ is simply the quotient of the two largest singular values of~$\rho_i(g')$. By Proposition~\ref{eigenvalues_and_singular_values_characterization}~(ii) (giving the singular values of an element of~$G$ in a given representation) and by a calculation analogous to the previous point, we have
\begin{equation}
\label{eq:srhoig'_calculation}
\tilde{s}(\rho_i(g')) = \exp(\alpha_i(\cartan(g)))^{-1}.
\end{equation}
The desired estimate follows by combining \eqref{eq:srhoig_srhoig'} with~\eqref{eq:srhoig'_calculation}. \qedhere
\end{itemize}
\end{proof}

\begin{proposition}
\label{C-non-deg-in-Vi}
Let $(g_1, g_2)$ be a $C$-non-degenerate pair of elements of~$G$ of type~$X_0$. Then for every~$i \in \Pi \setminus \Pi_{X_0}$, the pair~$(\rho_i(g_1), \rho_i(g_2))$ is a $C'_i$-non-degenerate pair of proximal maps in~$\GL(V_i)$, where $C'_i$~is some constant that depends only on~$C$ and~$i$.
\end{proposition}

Before proving this proposition, we need a couple of lemmas. 

\begin{lemma} Let $i \in \Pi \setminus \Pi_{X_0}$.
\label{stabilizer_Es_Eu}
\begin{hypothenum}
\item The restricted weight space~$V^{n_i \varpi_i}_i$ is stable by~$\rho_i(P_{X_0}^+)$.
\item The direct sum of all restricted weight spaces~$V^\lambda_i$ with~$\lambda \neq n_i \varpi_i$ is stable by~$\rho_i(P_{X_0}^-)$.
\end{hypothenum}
\end{lemma}
\begin{proof}~
\begin{hypothenum}
\item Let us first prove that this space is stable by~$\mathfrak{p}_{X_0}^+$. By definition, we have:
\[\mathfrak{p}_{X_0}^+ = \mathfrak{l} \oplus \bigoplus_{\beta(X_0) \geq 0} \mathfrak{g}^\beta;\]
Since $\mathfrak{l}$ centralizes~$\mathfrak{a}$, it preserves the restricted weight space decomposition; so clearly $\mathfrak{l}$ stabilizes~$V^{n_i \varpi_i}_i$.

Now let $\beta$~be a root such that~$\beta(X_0) \geq 0$; let us write
\[\beta = \sum_{\alpha \in \Pi} c_\alpha \alpha.\]
By definition of the set~$\Pi_{X_0}$, we then have
\begin{equation}
c_\alpha \geq 0 \quad\text{for } \alpha \in \Pi \setminus \Pi_{X_0}.
\end{equation}
Now we know that
\[\mathfrak{g}^\beta \cdot V^{n_i \varpi_i}_i \subset V^{n_i \varpi_i + \beta}_i.\]
The latter space is actually zero. Indeed, otherwise, $n_i \varpi_i + \beta$ would have to be a restricted root. But from Lemma~\ref{fund_repr_other_weights}, we know that this would imply
\[c_{\alpha_i} \leq -1,\]
which contradicts the inequality above, since~$i$ (or, technically, the root~$\alpha_i$) is in~${\Pi \setminus \Pi_{X_0}}$. It follows that for every~$\beta$ such that~$\beta(X_0) \geq 0$, the space~$V^{n_i \varpi_i}_i$ is stable by~$\mathfrak{g}^\beta$; we conclude that it is stable by~$\mathfrak{p}_{X_0}^+$.

By integration, we deduce that this space is also stable by~$P_{X_0, e}^+$. Now we know (it follows from \cite{Kna96}, Proposition~7.82~(d)) that $P_{X_0}^+ = MP_{X_0, e}^+$. Since~$M$ centralizes~$\mathfrak{a}$, it preserves the restricted weight space decomposition, so it stabilizes~$V^{n_i \varpi_i}_i$. We conclude that~$P_{X_0}^+$ stabilizes~$V^{n_i \varpi_i}_i$.

\item The proof is completely analogous. \qedhere
\end{hypothenum}
\end{proof}

In the following lemma, we denote by~$\mathcal{PS}$ the set of all parabolic spaces of~$V$; we also identify the projective space~$\mathbb{P}(V_i)$ with the set of vector lines in~$V_i$ and the projective space~$\mathbb{P}(V^*_i)$ with the set of vector hyperplanes of~$V_i$.
\begin{remark}
Recall (Remark~\ref{transversality_in_flag_variety}) that by Proposition~\ref{stabg=stabge}, the manifold~$\mathcal{PS}$ is diffeomorphic to~$G/P^+_{X_0}$ (in a $G$-equivariant way).
\end{remark}
\begin{lemma}~
\label{PS_to_PVi_and_PVistar}
\begin{hypothenum}
\item For every $i \in \Pi \setminus \Pi_{X_0}$, there exists a unique pair of continuous maps
\[\Phi_i^s: \mathcal{PS} \to \mathbb{P}(V_i)
\quad\text{ and }\quad
\Phi_i^u: \mathcal{PS} \to \mathbb{P}(V^*_i)\]
such that for every map~$g \in G$ of type~$X_0$, we have
\[\begin{cases}
E^s_{\rho_i(g)} = \Phi_i^s(V^\subge_g); \\
E^u_{\rho_i(g)} = \Phi_i^u(V^\suble_g).
\end{cases}\]
\item Moreover, these maps have the following property: whenever $V_1, V_2 \in \mathcal{PS}$ are transverse, we have $\Phi_i^s(V_1) \not\in \Phi_i^u(V_2).$
\end{hypothenum}
\end{lemma}
\begin{proof}
Take any $g$ of type~$X_0$; then from the inequality~\eqref{eq:highest_eigenvalue} ranking the values of different restricted weights of~$\rho_i$ evaluated at~$\jordan(g)$, we deduce that we have
\begin{equation}
\begin{cases}
E^s_{\rho_i(\exp(\jordan(g)))} = V^{n_i \varpi_i}_i; \\
E^u_{\rho_i(\exp(\jordan(g)))} = \bigoplus_{\lambda \neq n_i \varpi_i} V^\lambda_i.
\end{cases}
\end{equation}
Now take any $\phi \in G$; applying the defining identities of the maps $\Phi_i^{s, u}$ to the conjugate $\phi \exp(\jordan(g)) \phi^{-1}$, we deduce that these two maps, if they exist, must necessarily satisfy
\begin{equation}
\begin{cases}
\Phi_i^s(\phi(V^\subge_0)) = \rho_i(\phi) \left( V^{n_i \varpi_i}_i \right); \\
\Phi_i^u(\phi(V^\suble_0)) = \rho_i(\phi) \left( \bigoplus_{\lambda \neq n_i \varpi_i} V^\lambda_i \right).
\end{cases}
\end{equation}
We may take this as a definition of $\Phi_i^s$ and~$\Phi_i^u$; it remains to check that it is not ambiguous. Clearly it is enough to check that whenever some $\phi \in G$ stabilizes the space~$V^\subge_0$ (resp.~$V^\suble_0$), it also stabilizes the line~$V^{n_i \varpi_i}_i$ (resp. hyperplane~$\bigoplus_{\lambda \neq n_i \varpi_i} V^\lambda_i$). Since~$i \in \Pi \setminus \Pi_{X_0}$, this follows from Lemma~\ref{stabilizer_Es_Eu} and Proposition~\ref{stabg=stabge}. That the maps $\Phi_i^s$ and $\Phi_i^u$ thus defined are continuous is then obvious.

As for property~(ii), it now follows from Proposition~\ref{pair_transitivity}, which says that $G$~acts transitively on the set of transverse pairs of parabolic spaces.
\end{proof}

\begin{proof}[Proof of Proposition~\ref{C-non-deg-in-Vi}]
Let us fix some $i \in \Pi \setminus \Pi_{X_0}$ and some $C \geq 1$. Then the set of $C$-non-degenerate pairs of parabolic spaces is compact. On the other hand, the function
\[(V_1, V_2) \mapsto \alpha(\Phi_i^s(V_1), \Phi_i^u(V_2))\]
is continuous, and (by Lemma~\ref{PS_to_PVi_and_PVistar}~(ii)) takes positive values on that set. Hence it is bounded below. So there is a constant~$C'_i \geq 1$, depending only on~$C$, such that whenever a pair $(V_1, V_2)$ of parabolic spaces is $C$-non-degenerate, the pair $(\Phi_i^s(V_1), \Phi_i^u(V_2))$ is $C'_i$-non-degenerate.

The conclusion then follows by Lemma~\ref{PS_to_PVi_and_PVistar}~(i).
\end{proof}

\begin{proposition}
\label{jordan_additivity}
For every $C \geq 1$, there are positive constants $s_{\ref{jordan_additivity}}(C)$ and $k_{\ref{jordan_additivity}}(C)$ with the following property. Take any $C$-non-degenerate pair $(g, h)$ of elements of~$G$ of type~$X_0$ such that $s(g) \leq s_{\ref{jordan_additivity}}(C)$ and $s(h) \leq s_{\ref{jordan_additivity}}(C)$. Then we have:
\begin{hypothenum}
\item $\forall i \in \mathrlap{\Pi,}\qquad\qquad\quad
\varpi_i \left( \jordan(gh) - \cartan(g) - \cartan(h) \right) \leq 0$;
\item $\forall i \in \mathrlap{\Pi \setminus \Pi_{X_0},}\qquad\qquad\quad
\varpi_i \left( \jordan(gh) - \cartan(g) - \cartan(h) \right) \geq - k_{\ref{jordan_additivity}}(C)$.
\end{hypothenum}
\end{proposition}
See Figure~\ref{fig:trapezoid_picture} for a picture explaining both this proposition and the corollary below.
\begin{remark}
Though we shall not use it, a very important particular case is $g = h$. We then obviously have $\jordan(gh) = 2\jordan(g)$ and $\cartan(g) + \cartan(h) = 2\cartan(g)$, so that the inequalities~(i) and~(ii) give a relationship between the Cartan and Jordan projections of a $C$-non-degenerate, sufficiently contracting map of type~$X_0$.
\end{remark}
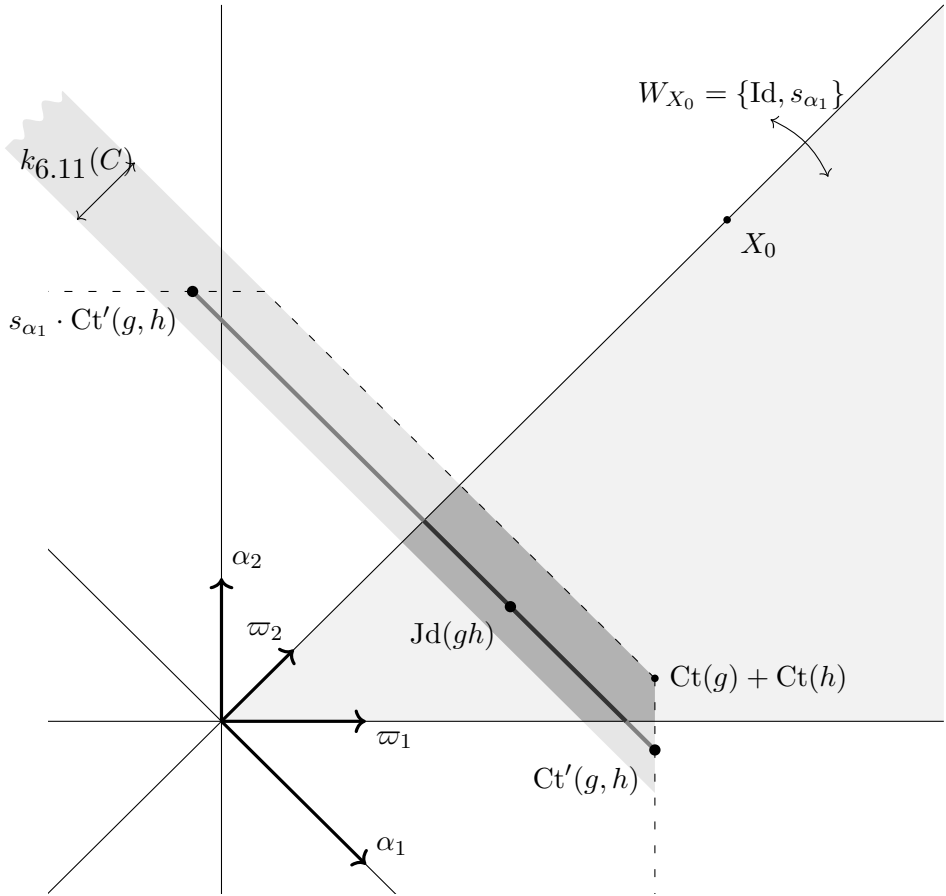
\begin{figure}
\[\begin{tikzpicture}[scale=1.9]
\fill[black!5!white] (0,0) -- (5,0) -- (5,5) -- (0,0);
\fill[black!10!white] (3,0.3) -- (-1.1,4.4) decorate[decoration=snake] {-- (-1.5,4)} -- (3,-0.5) -- (3,0.3);
\fill[black!30!white] (3,0.3) -- (1.65,1.65) -- (1.25,1.25) -- (2.5,0) -- (3,0) -- (3,0.3);
\draw[<->] (-0.6,3.9) -- (-1,3.5);
\node at (-1,3.9) {$k_{\ref{jordan_additivity}}(C)$};
\begin{scope}
  \clip (-1.2, -1.2) rectangle (5, 5);
  \draw (-5, 0) -- (5, 0);
  \draw (-5, -5) -- (5, 5);
  \draw (0, -5) -- (0, 5);
  \draw (5, -5) -- (-5, 5);
  \draw[loosely dashed] (3, 0.3) -- (3, -5);
  \draw[loosely dashed] (0.3, 3) -- (-5, 3);
\end{scope}
\draw[loosely dashed] (3, 0.3) -- (0.3, 3);
\draw[ultra thick, black!50!white] (3, -0.2) -- (-0.2, 3);
\draw[ultra thick, black!80!white] (2.8, 0) -- (1.4, 1.4);
\node[circle, fill, inner sep=1pt, label=right:$\cartan(g) + \cartan(h)$] at (3,0.3) {};
\node[circle, fill, inner sep=1.5pt, label=below left:${\cartan'(g,h)}$] at (3,-0.2) {};
\node[circle, fill, inner sep=1.5pt, label=below left:$\jordan(gh)$] at (2,0.8) {};
\node[circle, fill, inner sep=1.5pt, label=below left:${s_{\alpha_1} \cdot \cartan'(g,h)}$] at (-0.2,3) {};
\node[circle, fill, inner sep=1pt, label=below right:${X_0}$] at (3.5,3.5) {};
\draw[<->] (3.8,4.2) to[bend left = 20] (4.2,3.8);
\node[above] at (3.6,4.2) {$W_{X_0} = \{\Id, s_{\alpha_1}\}$};
\draw[very thick, ->] (0,0) -- (0,1);
\node[above right] at (0,1) {$\alpha_2$};
\draw[very thick, ->] (0,0) -- (0.5,0.5);
\node[above left] at (0.5,0.5) {$\varpi_2$};
\draw[very thick, ->] (0,0) -- (1,0);
\node[below right] at (1,0) {$\varpi_1$};
\draw[very thick, ->] (0,0) -- (1,-1);
\node[above right] at (1,-1) {$\alpha_1$};
\end{tikzpicture}\]
\caption{
This picture represents the situation of Example~\ref{equivalence_class_examples}.3, namely $G = \SO^+(3,2)$ acting on~$\mathbb{R}^5$. We have chosen a generic, symmetric, extreme vector~$X_0$. The set~$\Pi_{X_0}$ is then~$\{\alpha_1\}$ (or~$\{1\}$ with the usual abuse of notations), and the group~$W_{X_0}$ is generated by the single reflection~$s_{\alpha_1}$. Proposition~\ref{jordan_additivity} states that $\jordan(gh)$ lies in the shaded ``infinite trapezoid''. Corollary~\ref{jordan_additivity_reformulation} states that it lies on the thick line segment. In any case it lies by definition in the dominant open Weyl chamber (the shaded sector).
}
\label{fig:trapezoid_picture}
\end{figure}

Before proving the proposition, let us give a more palatable (though slightly weaker) reformulation.
\begin{corollary}
\label{jordan_additivity_reformulation}
For every $C \geq 1$, there exists a positive constant $k_{\ref{jordan_additivity_reformulation}}(C)$ with the following property. For any pair $(g, h)$ satisfying the hypotheses of Proposition~\ref{jordan_additivity}, we have
\begin{equation}
\label{eq:cartan_in_jordan_convex_hull}
\jordan(gh) \in \Conv \Big( W_{X_0} \cdot \cartan'(g, h) \Big),
\end{equation}
where $\Conv$ denotes the convex hull and $\cartan'(g, h)$~is some vector in~$\mathfrak{a}$ satisfying
\begin{equation}
\label{eq:ct'gh_ctg_cth}
\|\cartan'(g, h) - \cartan(g) - \cartan(h)\| \leq k_{\ref{jordan_additivity_reformulation}}(C).
\end{equation}
\end{corollary}
In fact, we can already give an explicit expression for this vector~$\cartan'(g,h)$:
\begin{definition}
\label{cartan_prime_definition}
We define~$\cartan'(g, h)$ to be the unique solution of the linear system
\begin{equation}
\label{eq:defining_system}
\begin{cases}
\forall i \in \mathrlap{\Pi \setminus \Pi_{X_0},}\qquad\qquad\quad
\varpi_i(\cartan'(g, h)) = \varpi_i(\jordan(gh)); \\
\forall i \in \mathrlap{\Pi_{X_0},}\qquad\qquad\quad
\varpi_i(\cartan'(g, h)) = \varpi_i(\cartan(g) + \cartan(h)).
\end{cases}
\end{equation}
(This works since $(\varpi_i)_{i \in \Pi}$ is a basis of~$\mathfrak{a}^*$.)
\end{definition}
\begin{remark}
Note that the vector~$\cartan'(g, h)$ might \emph{not} lie in the closed dominant Weyl chamber~$\mathfrak{a}^{+}$ (even though it is very close to the vector~$\cartan(g) + \cartan(h)$ which does).
\end{remark}
It remains to check that the vector~$\cartan'(g, h)$ thus defined satisfies indeed the required conditions.
\begin{proof}[Proof of Corollary~\ref{jordan_additivity_reformulation}.]
The estimate~\eqref{eq:ct'gh_ctg_cth} immediately follows from the inequalities of Proposition~\ref{jordan_additivity}. On the other hand, we may now rewrite Proposition~\ref{jordan_additivity} without the epsilons: combining Proposition~\ref{jordan_additivity}~(i) with the definition of~$\cartan'(g, h)$, we get
\begin{equation}
\label{eq:almost_convex_hull}
\begin{cases}
\forall i \in \mathrlap{\Pi,}\qquad\qquad\quad
  \varpi_i \left( \jordan(gh) - \cartan'(g, h) \right) \leq 0; \\
\forall i \in \mathrlap{\Pi \setminus \Pi_{X_0},}\qquad\qquad\quad
  \varpi_i \left( \jordan(gh) - \cartan'(g, h) \right) = 0.
\end{cases}
\end{equation}
Let us now show the inequalities
\begin{equation}
\label{eq:Weyl_subgroup_domain}
\forall i \in \Pi_{X_0},\quad \alpha_i(\cartan'(g, h)) \geq 0.
\end{equation}

Let $(H_i)_{i \in \Pi}$ be the basis of~$\mathfrak{a}$ dual to~$(\varpi_i)_{i \in \Pi}$, \ie the unique basis such that the identity
\begin{equation}
\label{eq:dual_basis}
\varpi_i \left( \sum_{j \in \Pi} c_j H_j \right) = c_i
\end{equation}
holds for any~$i \in \Pi$ and any tuple $(c_j) \in \mathbb{R}^\Pi$. By definition of the fundamental restricted weights $\varpi_i$, it then follows that we also have the identity
\begin{equation}
\forall i \in \Pi,\; \forall \lambda \in \mathfrak{a}^*,\quad \lambda(H_i) = \frac{2 \langle \lambda, \alpha'_i \rangle}{\|\alpha'_i\|^2}.
\end{equation}
By decomposing the vector~$\cartan(g) + \cartan(h) - \cartan'(g, h)$ in the basis~$(H_i)_{i \in \Pi}$ and by plugging the formula~\eqref{eq:dual_basis} into the second line of the defining system~\eqref{eq:defining_system}, we find that we may write
\begin{equation}
\cartan'(g, h) = \cartan(g) + \cartan(h) - \sum_{j \in \Pi \setminus \Pi_{X_0}} c_j H_j;
\end{equation}
by combining the first line of the defining system~\eqref{eq:defining_system} with Proposition~\ref{jordan_additivity}~(i), we also obtain that~$c_j \geq 0$ for every~$j \in \Pi \setminus \Pi_{X_0}$.

Finally, take any index $i \in \Pi_{X_0}$. Then we have
\begin{align}
\alpha_i \left( \sum_{j \in \Pi \setminus \Pi_{X_0}} c_j H_j \right)
  &= \sum_{j \in \Pi \setminus \Pi_{X_0}} c_j \alpha_i(H_j) \nonumber \\
  &= \sum_{j \in \Pi \setminus \Pi_{X_0}} c_j \frac{2 \langle \alpha_i, \alpha'_j \rangle}{\|\alpha'_j\|^2} \nonumber \\
  &\leq 0:
\end{align}
indeed since $j$~varies in~$\Pi \setminus \Pi_{X_0}$ and $i \in \Pi_{X_0}$, we have $i \neq j$ hence $\langle \alpha_i, \alpha_j \rangle \leq 0$; and $\alpha'_j$~is by construction a positive multiple of~$\alpha_j$. We conclude that
\[\alpha_i(\cartan'(g, h)) \geq \alpha_i(\cartan(g)) + \alpha_i(\cartan(h)) \geq 0\]
(since $\cartan(g), \cartan(h) \in \mathfrak{a}^{+}$), which gives us~\eqref{eq:Weyl_subgroup_domain}.

Now the system of inequalities~\eqref{eq:Weyl_subgroup_domain} is equivalent to saying that
\begin{equation}
\cartan'(g, h) \in \mathfrak{a}^{+}_{X_0},
\end{equation}
where $\mathfrak{a}^{+}_{X_0}$ is a fundamental domain for the action of the Weyl subgroup~$W_{X_0}$ on~$\mathfrak{a}$, more specifically the one that contains the dominant Weyl chamber~$\mathfrak{a}^{+}$. The statement~\eqref{eq:cartan_in_jordan_convex_hull} then follows from this and from~\eqref{eq:almost_convex_hull}, by applying Proposition~\ref{convex_hull_and_inequalities} which characterizes convex hulls of orbits of~$W_{X_0}$.
\end{proof}

\begin{proof}[Proof of Proposition~\ref{jordan_additivity}]
Let $i \in \Pi$. We know (see \eqref{eq:highest_eigenvalue} above) that for any vector~$X \in \mathfrak{a}^{+}$, the number~$n_i \varpi_i(X)$ is the largest eigenvalue of $\rho_i(X)$. From Proposition~\ref{eigenvalues_and_singular_values_characterization}, it then follows that:
\begin{equation}
\begin{cases}
n_i \varpi_i(\cartan(g)) = \log \|\rho_i(g)\|; \\
n_i \varpi_i(\jordan(g)) = \log r(\rho_i(g))
\end{cases}
\end{equation}
(recall that $r$~denotes the spectral radius).

\begin{hypothenum}
\item is straightforward from here: indeed,
\begin{align}
n_i \varpi_i \Big( \jordan(gh) \Big) &= \log r(\rho_i(gh)) \nonumber \\
                                  &\leq \log \|\rho_i(gh)\| \nonumber \\
                                  &\leq \log \|\rho_i(g)\| \|\rho_i(h)\| \nonumber \\
                                  &= n_i \varpi_i \Big( \cartan(g) + \cartan(h) \Big).
\end{align}

\item Assume that $i \in \Pi \setminus \Pi_{X_0}$. By Proposition~\ref{type_X0_to_proximality}, we know that the maps $\rho_i(g)$ and~$\rho_i(h)$ are proximal. By Proposition~\ref{C-non-deg-in-Vi}, they form a $C'_i$-non-degenerate pair, for some~$C'_i$ that depends only on~$C$. By Proposition~\ref{type_X0_to_proximality}, if we take~$s_{\ref{jordan_additivity}}(C)$ small enough, we may then assume that both $\rho_i(g)$ and~$\rho_i(h)$ are $\tilde{s}_{\ref{proximal_product}}(C'_i)$-contracting. We may then apply Proposition~\ref{proximal_product}~(iii) to these two maps: we get
\[r(\rho_i(g)\rho_i(h)) \asymp_C \|\rho_i(g)\| \|\rho_i(h)\|.\]
Now from Proposition~\ref{eigenvalues_and_singular_values_characterization}, it follows that we have:
\[\begin{cases}
r(\rho_i(gh)) = \exp(n_i \varpi_i(\jordan(gh))); \\
\|\rho_i(g)\| = \exp(n_i \varpi_i(\cartan(g))); \\
\|\rho_i(h)\| = \exp(n_i \varpi_i(\cartan(h))).
\end{cases}\]
Taking the logarithm, we deduce that there exists $\eps_i(C)$ such that for sufficiently contracting $g$ and~$h$, we have
\begin{equation}
n_i \varpi_i \Big( \jordan(gh) - \cartan(g) - \cartan(h) \Big) \in [-\eps_i(C),\; \eps_i(C)].
\end{equation}
Taking
\begin{equation}
k_{\ref{jordan_additivity}}(C) := \max_{i \in \Pi \setminus \Pi_{X_0}} \frac{1}{n_i} \eps_i(C),
\end{equation}
the conclusion follows. \qedhere
\end{hypothenum}
\end{proof}

\begin{remark}
\label{comparison_with_Benoist}
Corollary~\ref{jordan_additivity_reformulation} generalizes a result given by Benoist in~\cite{Ben97}. More specifically, by taking together Lemma~4.1 and Lemma~4.5.2 from that paper, we obtain that under suitable conditions, the vector
\[\jordan(gh) - \cartan(g) - \cartan(h)\]
(which is $\lambda(gh) - \mu(g) - \mu(h)$ in Benoist's notations) is bounded. This seems to be stronger than our result; but in fact, it also relies on stronger assumptions. More precisely, there are two possible ways to interpret Benoist's result in the context of our paper:
\begin{itemize}
\item Either we may take his set~$\theta$ to be our~$\Pi \setminus \Pi_{X_0}$. In that case, \cite{Ben97} uses the additional assumption that $g$ and~$h$ are ``of type~$\theta$'', which is very restrictive: it means that their Jordan projections must lie in the intersection of the kernels of all roots in~$\Pi_{X_0}$
(which is also the space of fixed points of~$W_{X_0}$). To control the Cartan projections of $g$ and~$h$, \cite{Ben97} uses the assumption that $g$ and~$h$ actually belong to a whole Zariski-dense subgroup of~$G$, all of whose elements are of type~$\theta$. As Benoist remarks in the second paragraph of Remark~3.2.1 in~\cite{Ben97}, the latter assumption only makes sense for $p$-adic groups; in the case of real groups which is of interest to us, as shown in the appendix of~\cite{BenLab}, this is actually impossible unless $\theta = \Pi$.

\item Or we may take~$\theta$ to be the whole set~$\Pi$. But in that case, \cite{Ben97} needs the assumption that $g$ and~$h$ are proximal (and in general position) in \emph{all} representations~$\rho_i$, which is stronger than the hypotheses we have made.
\end{itemize}
\end{remark}

\begin{proposition}
\label{regular_product_qualitative}
For every $C \geq 1$, there is a positive constant $s_{\ref{regular_product_qualitative}}(C) \leq 1$ with the following property. Take any $C$-non-degenerate pair $(g, h)$ of maps of type~$X_0$ in~$G$ such that $s(g^{\pm 1}) \leq s_{\ref{regular_product_qualitative}}(C)$ and $s(h^{\pm 1}) \leq s_{\ref{regular_product_qualitative}}(C)$. Then $gh$ is still of type~$X_0$.
\end{proposition}
\begin{proof}
Let $C \geq 1$, and let $(g, h)$ be a $C$-non-degenerate pair of maps in $G \ltimes V$ of type~$X_0$, such that
\[s(g^{\pm 1}) \leq s_{\ref{regular_product_qualitative}}(C) \quad\text{ and }\quad s(h^{\pm 1}) \leq s_{\ref{regular_product_qualitative}}(C)\]
for some positive constant $s_{\ref{regular_product_qualitative}}(C)$ to be specified later.

Lemma~\ref{contraction_strength_and_cartan_projection} then gives us
\begin{align}
\max_{\lambda \in \Omega^\subl_{X_0}} \lambda(\cartan(g))
  &\;\leq\; \log s(g) + k_{\ref{contraction_strength_and_cartan_projection}}(C) +
\min_{\lambda \in \Omega^\subge_{X_0}} \lambda(\cartan(g)) \nonumber \\
  &\;\leq\; \log s(g) + k_{\ref{contraction_strength_and_cartan_projection}}(C).
\end{align}
(Indeed by Assumption~\ref{zero_is_a_weight}, $\lambda = 0$ is a restricted weight that is certainly contained in~$\Omega^\subge_{X_0}$, so the minimum above is nonpositive.)

Taking $s_{\ref{regular_product_qualitative}}(C)$ small enough, we may assume that
\begin{equation}
\label{eq:ljg_estimate}
\forall \lambda \in \Omega^\subl_{X_0},\quad
\lambda(\cartan(g)) < -\frac{1}{2} \left( \max_{\lambda \in \Omega} \|\lambda\| \right) k_{\ref{jordan_additivity_reformulation}}(C).
\end{equation}
Of course a similar estimate holds for~$h$:
\begin{equation}
\label{eq:ljh_estimate}
\forall \lambda \in \Omega^\subl_{X_0},\quad
\lambda(\cartan(h)) < -\frac{1}{2} \left( \max_{\lambda \in \Omega} \|\lambda\| \right) k_{\ref{jordan_additivity_reformulation}}(C).
\end{equation}

Now let $\lambda$ be any restricted weight that does not vanish on~$X_0$. We distinguish two cases:
\begin{itemize}
\item Suppose that $\lambda(X_0) < 0$. Recall Corollary~\ref{jordan_additivity_reformulation}; for the key vector~$\cartan'(g, h)$ that it involves, we will use the value given by Definition~\ref{cartan_prime_definition}. Then on the one hand, we deduce from~\eqref{eq:ct'gh_ctg_cth} that:
\begin{align}
\label{eq:ljgh-ljg-ljh_estimate}
|\lambda(\cartan'(g, h)) - \lambda(\cartan(g)) - \lambda(\cartan(h))|
  &\;\leq\; \|\lambda\|\|\cartan'(g, h) - \cartan(g) - \cartan(h)\| \nonumber \\
  &\;\leq\; \left( \max_{\lambda \in \Omega} \|\lambda\| \right) k_{\ref{jordan_additivity_reformulation}}(C).
\end{align}
Adding together the three estimates \eqref{eq:ljg_estimate}, \eqref{eq:ljh_estimate} and \eqref{eq:ljgh-ljg-ljh_estimate}, we get
\begin{equation}
\label{eq:lambda_cartan_gh_less0}
\lambda(\cartan'(g, h)) < 0;
\end{equation}
and this is true for \emph{any}~$\lambda \in \Omega^\subl_{X_0}$.

On the other hand, we have~\eqref{eq:cartan_in_jordan_convex_hull} which says that
\[\jordan(gh) \in \Conv(W_{X_0} \cdot \cartan'(g, h)).\]
Now since $\Omega^\subl_{X_0}$ is stable by~$W_{X_0}$, it follows from \eqref{eq:lambda_cartan_gh_less0} that we still have
\[\lambda(w(\cartan'(g, h))) = w^{-1}(\lambda)(\cartan'(g, h)) < 0\]
for any $w \in W_{X_0}$. Thus $\lambda$ takes negative values on every point of the orbit ${W_{X_0} \cdot \cartan'(g, h)}$; hence it also takes negative values on every point of its convex hull. In particular, we have
\begin{equation}
\lambda(\jordan(gh)) < 0.
\end{equation}

\item Suppose that $\lambda(X_0) > 0$. Since the set of restricted weights~$\Omega$ is invariant by~$W$, the form $w_0(\lambda)$ is still a restricted weight; since by hypothesis $X_0$~is symmetric (\ie $-w_0(X_0) = X_0$), we then have
\[w_0(\lambda)(X_0) < 0.\]
We may thus apply the previous point to the weight~$w_0(\lambda)$ and to the map~$(gh)^{-1} = h^{-1}g^{-1}$ (since $g^{-1}$ and~$h^{-1}$ verify the same hypotheses as $g$ and~$h$); this gives us
\[w_0(\lambda)(\jordan((gh)^{-1})) < 0.\]
Since $\jordan((gh)^{-1}) = -w_0(\jordan(gh))$, we conclude that
\begin{equation}
\lambda(\jordan(gh)) > 0.
\end{equation}
\end{itemize}
We conclude that $gh$ is indeed of type~$X_0$.
\end{proof}

\begin{remark}
If we assume that both $g$ and~$g^{-1}$ are sufficiently contracting, then clearly Lemma~\ref{contraction_strength_and_cartan_projection} implies that~$\cartan'(g, g)$ and then~$\cartan(g)$ also has the same type as~$X_0$. Conversely, we may show (by a version of Lemma~\ref{contraction_strength_and_cartan_projection} with the inequality going both ways) that if $\cartan(g)$ has the same type as~$X_0$ and is ``far enough'' from the borders of~$\mathfrak{a}_{\rho, X_0}$, then $g$ and~$g^{-1}$ are strongly contracting.
\end{remark}

\section{Products of maps of type~$X_0$}
\label{sec:quantitative_properties_of_products}

%
%
The goal of this section is to prove Proposition~\ref{regular_product}, which not only says that a product of a $C$-non-degenerate, sufficiently contracting pair of maps of type~$X_0$ is itself of type~$X_0$, but allows us to control the geometry and contraction strength of the product. To do this, we proceed almost exactly as in Section~3.2 in~\cite{Smi14}: we reduce the problem to Proposition~\ref{proximal_product}, by considering the action of~$G \ltimes V$ on a suitable exterior power $\ext^{p} A$ (rather than on the spaces~$V_i$ as in the previous section).

There is however one crucial difference from~\cite{Smi14}: while it is still true that when $g$~is of type~$X_0$, its exterior power~$\ext^p g$ is proximal, the converse no longer holds. Filling that gap is what the whole previous section was about.

\begin{remark}
The reader might wonder why we did not (developing upon the final remark from the previous section) prove an additivity theorem for Cartan projections similar to Proposition~\ref{jordan_additivity}, and use it to estimate $s(gh)$ in terms of $s(g^{\pm 1})$ and~$s(h^{\pm 1})$. Since we need to study the action on the spaces~$V_i$ anyway, this would seemingly allow us to forgo the additional introduction of~$\ext^p A$.

The reason is that this approach only works for \emph{linear} maps $g$ and~$h$: for $g \in G \ltimes V$, the Cartan projection is only defined for~$\ell(g)$ and only gives information about the singular values of~$\ell(g)$, not those of~$g$. So while possible, this approach would force us, on the other hand, to abandon the unified treatment of quantitative properties of affine maps (as outlined in Remark~\ref{unified_treatment}).
\end{remark}

We introduce the integers:
\begin{align}
p &:= \dim A^\subge_0 = \dim V^\subge_0 + 1; \nonumber \\
q &:= \dim V^\subl_0; \\
d &:= \dim A = \dim V + 1 = q+p. \nonumber
\end{align}
For every $g \in G \ltimes V$, we may define its exterior power $\ext^p g: \ext^p A \to \ext^p A$. The Euclidean structure of $A$ induces in a canonical way a Euclidean structure on $\ext^p A$.

\begin{lemma} \mbox{ }
\label{regular_to_proximal}
\begin{hypothenum}
\item Let $g \in G \ltimes V$ be a map of type~$X_0$. Then $\ext^{p} g$ is proximal, and the attracting (resp. repelling) space of $\ext^{p} g$ depends on nothing but $A^{\subge}_{g}$ (resp. $V^{\subl}_{g}$):
\[\begin{cases}
E^s_{\ext^{p} g} = \ext^{p} A^{\subge}_{g} \\
E^u_{\ext^{p} g} = \setsuch{x \in \ext^{p} A}
                                        {x \wedge \ext^{q} V^{\subl}_{g} = 0}.
\end{cases}\]
\item For every $C \geq 1$, whenever $(g_1, g_2)$ is a $C$-non-degenerate pair of maps of type~$X_0$, $(\ext^p g_1, \ext^p g_2)$ is a $C^p$-non-degenerate pair of proximal maps.
\item For every $C \geq 1$, for every $C$-non-degenerate map $g \in G \ltimes V$ of type~$X_0$, we have
\begin{equation}
s(g) \lesssim_C \tilde{s}(\ext^{p} g).
\end{equation}
If in addition $s(g) \leq 1$, we have
\begin{equation}
s(g) \asymp_C \tilde{s}(\ext^{p} g).
\end{equation}
(Recall the Definitions~\ref{s_definition} and~\ref{s_tilde_definition} of the two different notions of ``contraction strength'' $s(g)$ and~$\tilde{s}(\gamma)$, respectively.)
\item For any two $p$-dimensional subspaces $A_1$ and $A_2$ of $A$, we have
\[\alpha^\mathrm{Haus}(A_1, A_2)
\;\asymp\; \alpha \left( \ext^{p} A_1,\; \ext^{p} A_2 \right).\]
\end{hypothenum}
\end{lemma}

This is similar to Lemma~3.8 in~\cite{Smi14}, except for point~(i) which here is weaker than there.

\begin{proof} For (i), let $g \in G \ltimes V$ be a map of type~$X_0$. Let $\lambda_{1}, \ldots, \lambda_{d}$ be the eigenvalues of $g$ (acting on $A$) counted with multiplicity and ordered by nondecreasing modulus; then $|\lambda_{q+1}| = 1$ and $|\lambda_{q}| < 1$. On the other hand, we know that the eigenvalues of~$\ext^{p} g$ counted with multiplicity are exactly the products of the form $\lambda_{i_1}\cdots\lambda_{i_{p}}$, where $1 \leq i_1 < \cdots < i_{p} \leq d$. As the two largest of them (by modulus) are $\lambda_{q+1} \cdots \lambda_{d}$ and $\lambda_{q}\lambda_{q+2} \cdots \lambda_{d}$, it follows that $\ext^{p} g$ is proximal.

%
%
As for the expression of $E^s$ and $E^u$, it follows immediately by considering a basis that trigonalizes~$g$.

For (ii), (iii) and~(iv), the proof is exactly the same as for the corresponding points in Lemma~3.8 in~\cite{Smi14}, \emph{mutatis mutandis}.
\end{proof}

We also need the following technical lemma, which generalizes Lemma~3.9 in~\cite{Smi14}:

\begin{lemma}
\label{continuity_of_non_degeneracy}
There is a constant $\eps > 0$ with the following property. Let $A_1, A_2$ be any two affine parabolic spaces such that
\[\begin{cases}
\alpha^\mathrm{Haus}(A_1, A^\subge_0) \leq \eps \\
\alpha^\mathrm{Haus}(A_2, A^\suble_0) \leq \eps.
\end{cases}\]
Then they form a $2$-non-degenerate pair.
\end{lemma}
(Of course the constant~2 is arbitrary; we could replace it by any number larger than~1.)
\begin{proof}
The proof is exactly the same as the proof of Lemma~3.9 in~\cite{Smi14}, \emph{mutatis mutandis}.
\end{proof}

\begin{proposition}
\label{regular_product}
For every $C \geq 1$, there is a positive constant $s_{\ref{regular_product}}(C) \leq 1$ with the following property. Take any $C$-non-degenerate pair $(g, h)$ of maps of type~$X_0$ in~$G \ltimes V$; suppose that we have $s(g^{\pm 1}) \leq s_{\ref{regular_product}}(C)$ and $s(h^{\pm 1}) \leq s_{\ref{regular_product}}(C)$. Then $gh$ is of type~$X_0$, $2C$-non-degenerate, and we have:
\begin{hypothenum}
\item $\begin{cases}
\alpha^\mathrm{Haus} \left(A^{\subge}_{gh},\; A^{\subge}_{g} \right) \lesssim_C s(g) \vspace{1mm} \\
\alpha^\mathrm{Haus} \left(A^{\suble}_{gh},\; A^{\suble}_{h} \right) \lesssim_C s(h^{-1})
\end{cases}$;
\item $s(gh) \lesssim_C s(g)s(h)$.
\end{hypothenum}
\end{proposition}
(This generalizes Proposition~3.6 in~\cite{Smi14}.)

Before giving the proof, let us first formulate a particular case:

\begin{corollary}
\label{vector_spaces_estimate}
Under the same hypotheses, we have
\[\begin{cases}
\alpha^\mathrm{Haus} \left(V^{\subge}_{gh},\; V^{\subge}_{g} \right) \lesssim_C s(\ell(g)) \vspace{1mm} \\
\alpha^\mathrm{Haus} \left(V^{\suble}_{gh},\; V^{\suble}_{h} \right) \lesssim_C s(\ell(h)^{-1}).
\end{cases}\]
\end{corollary}
\begin{proof}
This follows from Lemma~\ref{affine_to_vector}. The proof is the same as for Corollary~3.7 in~\cite{Smi14}.
\end{proof}

\begin{proof}[Proof of Proposition~\ref{regular_product}]
Let us fix some positive constant $s_{\ref{regular_product}}(C)$, small enough to satisfy all the constraints that will appear in the course of the proof. Let $(g, h)$ be a pair of maps satisfying the hypotheses.

First note that by Lemma~\ref{affine_to_vector}, we have
\begin{equation}
s(\ell(g)^{\pm 1}) \leq s(g^{\pm 1}) \leq s_{\ref{regular_product}}(C)
\end{equation}
and similarly for~$h$. If we take $s_{\ref{regular_product}}(C) \leq s_{\ref{regular_product_qualitative}}(C)$, then Proposition~\ref{regular_product_qualitative} tells us that~$\ell(gh)$, hence~$gh$ (indeed the Jordan projection depends only on the linear part), is of type~$X_0$.

The remaining part of the proof works exactly like the proof of Proposition~3.6 in~\cite{Smi14}, namely by applying Proposition~\ref{proximal_product} to the maps $\gamma_1 = \ext^{p} g$ and~${\gamma_2 = \ext^{p} h}$. Taking into account the central position occupied in the paper by the proposition we are currently proving, let us reproduce these details nevertheless. Let us check that $\gamma_1$ and~$\gamma_2$ satisfy the required hypotheses:
\begin{itemize}
\item By Lemma~\ref{regular_to_proximal}~(i), $\gamma_1$ and $\gamma_2$ are proximal.
\item By Lemma~\ref{regular_to_proximal}~(ii), the pair $(\gamma_1, \gamma_2)$ is $C^p$-non-degenerate.
\item Since we have supposed $s_{\ref{regular_product}}(C) \leq 1$, it follows by Lemma~\ref{regular_to_proximal}~(iii) that $\tilde{s}(\gamma_1) \lesssim_C s(g)$ and $\tilde{s}(\gamma_2) \lesssim_C s(h)$. If we choose $s_{\ref{regular_product}}(C)$ sufficiently small, then $\gamma_1$ and $\gamma_2$ are $\tilde{s}_{\ref{proximal_product}}(C^p)$-contracting, \ie sufficiently contracting to apply Proposition~\ref{proximal_product}.
\end{itemize}
Thus we may apply Proposition~\ref{proximal_product}. It remains to deduce from its conclusions the conclusions of Proposition~\ref{regular_product}.
\begin{itemize}
\item We already know that $gh$ is of type~$X_0$.
\item From Proposition~\ref{proximal_product}~(i), using Lemma~\ref{regular_to_proximal} (i), (iii) and~(iv), we get
\[\alpha^\mathrm{Haus} \left(A^{\subge}_{gh},\; A^{\subge}_{g} \right)
  \lesssim_C s(g),\]
which shows the first line of Proposition~\ref{regular_product} (i).
\item By applying Proposition~\ref{proximal_product} to $\gamma_2^{-1} \gamma_1^{-1}$ instead of $\gamma_1 \gamma_2$, we get in the same way the second line of Proposition~\ref{regular_product} (i).
\item Let $\phi$ be an optimal canonizing map for the pair $(A^{\subge}_{g}, A^{\suble}_{h})$. By hypothesis, we have $\left \|\phi^{\pm 1} \right\| \leq C$. But if we take $s_{\ref{regular_product}}(C)$ sufficiently small, the two inequalities that we have just shown, together with Lemma~\ref{continuity_of_non_degeneracy}, allow us to find a map $\phi'$ with $\|\phi'\| \leq 2$, $\|{\phi'}^{-1}\| \leq 2$ and
\[\phi' \circ \phi (A^{\subge}_{gh}, A^{\suble}_{gh}) = (A^\subge_0, A^\suble_0).\]
It follows that the composition map $gh$ is $2C$-non-degenerate.
\item The last inequality, namely Proposition~\ref{regular_product}~(ii), now is deduced from Proposition~\ref{proximal_product}~(ii) by using Lemma~\ref{regular_to_proximal}~(iii). \qedhere
\end{itemize}
\end{proof}

\section{Additivity of Margulis invariants}
\label{sec:additivity}

Proposition~\ref{invariant_additivity} below is the key ingredient of the proof of the Main Theorem. It explains how the Margulis invariant behaves under group operations (inverse and composition).

The first point is easy to prove, but still important. It is a generalization of Proposition~4.1~(i) in~\cite{Smi14}; as the general case is slightly harder, we have now given more details.

The proof of the second point occupies the remainder of this section. We prove it by reducing it successively to Lemma~\ref{gghg=} (which is proved using the technical lemma~\ref{C-bounded}), then to Lemma~\ref{close_to_identity}. The proof follows very closely that of Proposition~4.2~(ii) in~\cite{Smi14}, and we have actually omitted the proofs of Lemmas~\ref{C-bounded} and~\ref{close_to_identity}. We did repeat the proof of the proposition itself (to help the reader figure out precisely what is to be changed), as well as the proof of Lemma~\ref{gghg=} (to clear up a small confusion in the original proof: see Remark~\ref{transvection_screwup}).

\begin{proposition} \mbox{ }
\label{invariant_additivity}
\begin{hypothenum}
\item For every map $g \in G \ltimes V$ of type~$X_0$, we have
\[M(g^{-1}) = -w_0(M(g)).\]
\item For every $C \geq 1$, there are positive constants $s_{\ref{invariant_additivity}}(C) \leq 1$ and $k_{\ref{invariant_additivity}}(C)$ with the following property. Let $g, h \in G \ltimes V$ be a $C$-non-degenerate pair of maps of type~$X_0$, with $g^{\pm 1}$ and $h^{\pm 1}$ all $s_{\ref{invariant_additivity}}(C)$-contracting. Then $gh$ is of type~$X_0$, and we have:
\[\|M(gh) - M(g) - M(h)\| \leq k_{\ref{invariant_additivity}}(C).\]
\end{hypothenum}
\end{proposition}
\begin{remark}
\label{w0_action_works_again}
To justify the slight abuse of notations $w_0(M(g))$, recall Remark~\ref{w0_action_makes_sense}, that we may now reformulate as follows: $w_0$ induces a linear involution on~$V^\transl_0$ (which is the space of fixed points by~$L$), and this involution does not depend on the choice of a representative of~$w_0$ in~$G$.
\end{remark}

Let $C \geq 1$. We choose some positive constant $s_{\ref{invariant_additivity}}(C) \leq 1$, small enough to satisfy all the constraints that will appear in the course of the proof. For the remainder of this section, we fix $g, h \in G \ltimes V$ a $C$-non-degenerate pair of maps of type~$X_0$ such that $g^{\pm 1}$ and $h^{\pm 1}$ are $s_{\ref{invariant_additivity}}(C)$-contracting.

The following remark will be used throughout this section.
\begin{remark}
\label{ggh_gh_2C}
We may suppose that the pairs $(A^{\subge}_{gh}, A^{\suble}_{gh})$, $(A^{\subge}_{hg}, A^{\suble}_{hg})$, $(A^{\subge}_{g}, A^{\suble}_{gh})$ and $(A^{\subge}_{hg}, A^{\suble}_{g})$ are all $2C$-non-degenerate.
Indeed, recall that (by Proposition~\ref{regular_product}), we have
\[\begin{cases}
  \alpha^\mathrm{Haus} \left(A^{\subge}_{gh},\; A^{\subge}_{g} \right) \lesssim_C s(g) \vspace{1mm} \\
  \alpha^\mathrm{Haus} \left(A^{\suble}_{gh},\; A^{\suble}_{h} \right) \lesssim_C s(h^{-1})
\end{cases}\]
and similar inequalities with $g$ and $h$ interchanged. On the other hand, by hypothesis, $(A^{\subge}_{g}, A^{\suble}_{h})$ is $C$-non-degenerate. If we choose $s_{\ref{invariant_additivity}}(C)$ sufficiently small, these four statements then follow from Lemma~\ref{continuity_of_non_degeneracy}.
\end{remark}

\begin{proof}[Proof of Proposition~\ref{invariant_additivity}] \mbox{ }
\begin{hypothenum}
\item Let~$\phi$ be a canonizing map for~$g$. Since $V^{\subge}_{g^{-1}} = V^{\suble}_{g}$ and vice-versa (obviously) and since $V^\suble_0 = w_0 V^\subge_0$ and vice-versa (because $X_0$~is symmetric), it follows that $w_0 \phi$ is a canonizing map for~$g^{-1}$.

It remains to show that $w_0$ commutes with $\pi_{\transl}$. Indeed, it is well-known that the group~$W$, that we defined as the quotient~$N_G(A)/Z_G(A)$, is also equal to the quotient~$N_K(A)/Z_K(A)$ (see \cite{Kna96}, formulas (7.84a) and~(7.84b)); hence
\begin{equation}
N_G(A) = W Z_G(A) = W Z_K(A) A = N_K(A) A \subset K A.
\end{equation}
Let $\tilde{w}_0$ be any representative of~$w_0$ in~$N_G(A)$. We already know that both $V^\sube_0 = V^0$ (Remark~\ref{W_action_makes_sense}) and~$V^\transl_0$ (Remark~\ref{w0_action_works_again}) are invariant by~$\tilde{w}_0$. Now by definition the group~$A$ acts trivially on~$V^0$, and by construction $K$~acts on~$V^0$ by orthogonal transformations (indeed the Euclidean structure was chosen in accordance with Lemma~\ref{K-invariant}); hence~$V^\rotat_0$, which is the orthogonal complement of~$V^\transl_0$ in~$V^0$, is also invariant by~$\tilde{w}_0$.

The desired formula now immediately follows from the definition of the Margulis invariant.

\item The proof of this point is a straightforward generalization of the proof of Proposition~4.1~(ii) in~\cite{Smi14}.

If we take $s_{\ref{invariant_additivity}}(C) \leq s_{\ref{regular_product}}(C)$, then Proposition~\ref{regular_product} ensures that $gh$ is of type~$X_0$.

To estimate $M(gh)$, we decompose $gh: A^{\sube}_{gh} \to A^{\sube}_{gh}$ into a product of several maps.

\begin{itemize}
\item We begin by decomposing the product $gh$ into its factors. We have the commutative diagram
\begin{equation}
\begin{tikzcd}
  A^{\sube}_{gh}
& 
& 
& A^{\sube}_{hg}
    \arrow{lll}{g}
& 
& 
& A^{\sube}_{gh}
    \arrow{lll}{h}
    \arrow[bend right]{llllll}[swap]{gh}
\end{tikzcd}
\end{equation}
Indeed, since $hg$ is the conjugate of $gh$ by $h$ and vice-versa, we have $h(A^{\sube}_{gh}) = A^{\sube}_{hg}$ and $g(A^{\sube}_{hg}) = A^{\sube}_{gh}$.
\item Next we factor the map $g: A^{\sube}_{hg} \to A^{\sube}_{gh}$ through the map $g: A^{\sube}_{g} \to A^{\sube}_{g}$, which is better known to us. We have the commutative diagram
\begin{equation}
\begin{tikzcd}
  A^{\sube}_{gh}
    \arrow{rdd}[swap]{\pi_g}
&
&
& A^{\sube}_{hg}
    \arrow{lll}{g}
    \arrow{ldd}{\pi_g}
\\
\\
& A^{\sube}_{g}
& A^{\sube}_{g}
    \arrow{l}{g}
&
\end{tikzcd}
\end{equation}
where $\pi_g$ is the projection onto $A^{\sube}_{g}$ parallel to $V^{\subg}_{g} \oplus V^{\subl}_{g}$. (It commutes with $g$ because $A^{\sube}_{g}$, $V^{\subg}_{g}$ and $V^{\subl}_{g}$ are all invariant by $g$.)
\item Finally, we decompose again both arrows labelled $\pi_g$ on the last diagram into two factors. For any two maps $u$ and $v$ of type~$X_0$, we introduce the notation
\[A^{\sube}_{u, v} := A^{\subge}_{u} \cap A^{\suble}_{v}.\]
We call $P_1$ (resp. $P_2$) the projection onto $A^{\sube}_{g, gh}$ (resp. $A^{\sube}_{hg, g}$), still parallel to~$V^{\subg}_{g} \oplus V^{\subl}_{g}$. To justify this definition, we must check that $A^{\sube}_{g, gh}$ (and similarly~$A^{\sube}_{hg, g}$) is supplementary to $V^{\subg}_{g} \oplus V^{\subl}_{g}$. Indeed, by Remark~\ref{ggh_gh_2C}, $A^{\suble}_{gh}$ is transverse to $A^{\subge}_{g}$, hence (by Proposition~\ref{Asubge_is_ampa}~(ii)) supplementary to $V^{\subg}_{g}$; thus $A^{\subge}_{g} = V^{\subg}_{g} \oplus A^{\sube}_{g, gh}$ and $A = V^{\subl}_{g} \oplus A^{\subge}_{g} = V^{\subl}_{g} \oplus V^{\subg}_{g} \oplus A^{\sube}_{g, gh}$. Then we have the commutative diagrams
\begin{subequations}
\begin{equation}
\begin{tikzcd}
  A^{\sube}_{gh}
    \arrow[bend right]{rr}[swap]{\pi_g}
    \arrow{r}{P_1}
& A^{\sube}_{g, gh}
    \arrow{r}{\pi_g}
& A^{\sube}_{g}
\end{tikzcd}
\end{equation}
and
\begin{equation}
\begin{tikzcd}
  A^{\sube}_{hg}
    \arrow[bend right]{rr}[swap]{\pi_g}
    \arrow{r}{P_2}
& A^{\sube}_{hg, g}
    \arrow{r}{\pi_g}
& A^{\sube}_{g}
\end{tikzcd}
\end{equation}
\end{subequations}
\end{itemize}

The second and third step can be repeated with $h$ instead of $g$. The way to adapt the second step is straightforward; for the third step, we factor $\pi_h: A^{\sube}_{hg} \to A^{\sube}_{h}$ through $A^{\sube}_{h, hg}$ and $\pi_h: A^{\sube}_{gh} \to A^{\sube}_{h}$ through $A^{\sube}_{gh, h}$.

Combining these three decompositions, we get the lower half of Diagram~\ref{fig:largediagcomm1}. (We left out the expansion of $h$; we leave drawing the full diagram for especially patient readers.) Let us now interpret all these maps as endomorphisms of $A^\sube_0$. To do this, we choose some optimal canonizing maps
\[\phi_{g},\; \phi_{gh},\; \phi_{hg},\; \phi_{g, gh},\; \phi_{hg, g}\]
respectively of $g$, of $gh$, of $hg$, of the pair $(A^{\subge}_{g}, A^{\suble}_{gh})$ and of the pair $(A^{\subge}_{hg}, A^{\suble}_{g})$. This allows us to define $\overline{g_{gh}}$, $\overline{h_{gh}}$, $\overline{g_{g, gh}}$, $\overline{g_{\sube}}$, $\overline{P_1}$, $\overline{P_2}$, $\overline{\psi_1}$, $\overline{\psi_2}$ to be the maps that make the whole Diagram~\ref{fig:largediagcomm1} commutative.

\begin{figure}
\renewcommand{\figurename}{Diagram}
\[
\begin{tikzcd}
  A^\sube_0
    \arrow{rdd}[swap]{\overline{P_1}}
&
&
&
&
&
& A^\sube_0
    \arrow{ldd}{\overline{P_2}}
    \arrow{llllll}{\overline{g_{gh}}}
& A^\sube_0
    \arrow{l}{\overline{h_{gh}}}
\\
\\
&
  A^\sube_0
    \arrow{rdd}[swap]{\overline{\psi_1}}
&
&
&
& A^\sube_0
    \arrow{ldd}{\overline{\psi_2}}
    \arrow{llll}{\overline{g_{g, gh}}}
&
&
\\
\\
&
& A^\sube_0
&
& A^\sube_0
	\arrow{ll}{\overline{g_{\sube}}}
&
&
&
\\
\\
  A^{\sube}_{gh}
    \arrow{rdd}{P_1}
    \arrow[dashed]{uuuuuu}[pos=0.13]{\phi_{gh}}
&
&
&
&
&
& A^{\sube}_{hg}
    \arrow{ldd}[swap]{P_2}
    \arrow{llllll}{g}
    \arrow[dashed]{uuuuuu}[pos=0.13]{\phi_{hg}}
& A^{\sube}_{gh}
    \arrow{l}{h}
    \arrow[dashed]{uuuuuu}[pos=0.13]{\phi_{gh}}
\\
\\
&
  A^{\sube}_{g, gh}
    \arrow{rdd}{\pi_{g}}
    \arrow[dashed]{uuuuuu}{\phi_{g, gh}}
&
&
&
& A^{\sube}_{hg, g}
    \arrow{ldd}[swap]{\pi_{g}}
    \arrow[dashed]{uuuuuu}{\phi_{hg, g}}
&
&
\\
\\
&
& A^{\sube}_{g}
    \arrow[dashed]{uuuuuu}[pos=0.86]{\phi_{g}}
&
& A^{\sube}_{g}
	\arrow{ll}{g}
    \arrow[dashed]{uuuuuu}[pos=0.86]{\phi_{g}}
&
&
&
\end{tikzcd}
\]
\caption{}
\label{fig:largediagcomm1}
\end{figure}

Now let us define
\begin{equation}
\begin{cases}
M_{gh}(g) := \pi_{\transl}(\overline{g_{gh}}(x) - x) \\
M_{gh}(h) := \pi_{\transl}(\overline{h_{gh}}(x) - x)
\end{cases}
\end{equation}
for any $x \in V^\sube_{\Aff, 0}$, where $V^\sube_{\Aff, 0} := A^\sube_0 \cap V_{\Aff}$ is the affine space parallel to~$V^\sube_0$ and passing through the origin. Since $gh$ is the conjugate of $hg$ by~$g$ and vice-versa, the elements of $G \ltimes V$ (defined in an obvious way) whose restrictions to $A^\sube_0$ are $\overline{g_{gh}}$ and~$\overline{h_{gh}}$ stabilize the spaces $A^\subge_0$ and $A^\suble_0$. By Lemma~\ref{quasi-translation}, $\overline{g_{gh}}$ and~$\overline{h_{gh}}$ are thus quasi-translations. It follows that these values $M_{gh}(g)$ and~$M_{gh}(h)$ do not depend on the choice of $x$. Compare this to the definition of a Margulis invariant (Definition~\ref{margulis_invariant}): we have $M(gh) = \pi_{\transl}(\overline{g_{gh}} \circ \overline{h_{gh}}(x) - x)$ for any $x \in V^\sube_{\Aff, 0}$. It immediately follows that
\begin{equation}
M(gh) = M_{gh}(g) + M_{gh}(h).
\end{equation}
We may now estimate each of the two terms separately: if we show that $\|M_{gh}(g) - M(g)\| \lesssim_C 1$ and~${\|M_{gh}(h) - M(h)\| \lesssim_C 1}$, we are done. These two estimates follow immediately from Lemma~\ref{gghg=} below. (Note that while the vectors $M_{gh}(g)$ and $M_{gh}(h)$ are elements of $V^\transl_0$, the maps $\overline{g_{gh}}$ and $\overline{h_{gh}}$ are extended affine isometries acting on the whole subspace $A^\sube_0$.) \qedhere
\end{hypothenum}
\end{proof}
\begin{remark}
In contrast to actual Margulis invariants, the values $M_{gh}(g)$ and $M_{gh}(h)$ \emph{do} depend on our choice of canonizing maps. Choosing other canonizing maps would force us to subtract some constant from the former and add it to the latter.
\end{remark}

\begin{definition}
We shall say that a linear bijection $f$ between two subspaces of the extended affine space~$A$
is \emph{$K(C)$-bounded} if it is bounded by a constant depending only on $C$, that is, $\|f\| \lesssim_C 1$ and $\|f^{-1}\| \lesssim_C 1$. We say that two automorphisms $f_1, f_2$ of $A^\sube_0$ (depending somehow on $g$ and $h$) are \emph{$K(C)$-almost equivalent}, and we write $f_1 \approx_C f_2$, if they satisfy the condition
\[\|f_1 - \xi \circ f_2 \circ \xi'\| \lesssim_C 1\]
for some $K(C)$-bounded quasi-translations $\xi, \xi'$. This is indeed an equivalence relation.
\end{definition}

\begin{lemma}
\label{gghg=}
The maps $\overline{g_{gh}}$ and $\overline{h_{gh}}$ are $K(C)$-almost equivalent to $\overline{g_{\sube}}$ and $\overline{h_{\sube}}$, respectively.
\end{lemma}

To show this, we use the following property:
\begin{lemma}
\label{C-bounded}
All the non-horizontal (\ie vertical or diagonal) arrows in Diagram~\ref{fig:largediagcomm1} represent $K(C)$-bounded, bijective maps.
\end{lemma}

Note that Lemma~\ref{C-bounded} alone does not imply Lemma~\ref{gghg=}: indeed, while the maps $\overline{\psi}_1$ and $\overline{\psi}_2$ are quasi-translations by Lemma~\ref{projections_commute}, the maps $\overline{P}_1$ and $\overline{P}_2$ need not be. This issue will be addressed in Lemma~\ref{close_to_identity}. 

\begin{proof}[Proof of Lemma~\ref{C-bounded}]
The proof is exactly the same as the proof of Lemma~4.6 in~\cite{Smi14}, \emph{mutatis mutandis}.
\end{proof}

\begin{proof}[Proof of Lemma~\ref{gghg=}]
We shall concentrate on the estimate $\overline{g_{gh}} \approx_C \overline{g_{\sube}}$; the proof of the estimate $\overline{h_{gh}} \approx_C \overline{h_{\sube}}$ is analogous.

We now use Lemma~\ref{projections_commute} which shows that canonical identifications commute up to quasi-translation with suitable projections; it implies that the maps $\overline{\psi_1}$ and~$\overline{\psi_2}$ are quasi-translations. Hence $\overline{g_{g, gh}}$ is also a quasi-translation.

We would like to pretend that $\overline{g_{gh}}$ and $\overline{g_{g, gh}}$ are actually translations. To do that, we modify slightly the upper right-hand corner of Diagram~\ref{fig:largediagcomm1}. We set
\begin{equation}
\begin{cases}
\phi'_{hg} := \ell(\overline{g_{gh}}) \circ \phi_{hg} \\
\phi'_{hg, g} := \ell(\overline{g_{g, gh}}) \circ \phi_{hg, g},
\end{cases}
\end{equation}
where $\ell$ stands for the linear part as defined in Section~\ref{sec:affine_to_linear}, and we define $\overline{P'_2}$, $\overline{\psi'_2}$, $\overline{g'_{gh}}$,~$\overline{g'_{g, gh}}$ so as to make the new diagram commutative (see Diagram~\ref{fig:largediagcomm2}). The factors $\ell(\overline{g_{gh}})$ and~$\ell(\overline{g_{g, gh}})$ we introduced (the short horizontal arrows in Diagram~\ref{fig:largediagcomm2}) have norm~$1$: indeed, being quasi-translations of $A^\sube_0$ fixing $p_0$, they are orthogonal linear transformations (by Lemma~\ref{quasi-translation}). Thus Lemma~\ref{C-bounded} still holds for Diagram~\ref{fig:largediagcomm2}; but now, the modified maps $\overline{g'_{gh}}$ and $\overline{g'_{g, gh}}$ are translations by construction.

\begin{figure}[t]
\renewcommand{\figurename}{Diagram}
\[
\definecolor{mygray}{gray}{0.7}
\begin{tikzcd}
  A^\sube_0
    \arrow{rdd}[swap]{\overline{P_1}}
&
&
&
&
& A^\sube_0
    \arrow{ldd}[swap]{\overline{P'_2}}
    \arrow{lllll}{\overline{g'_{gh}}}
& {\color{mygray}A^\sube_0}
    \arrow[color=mygray]{ldd}[color=mygray]{\overline{P_2}}
    \arrow[dashed]{l}{\ell(\overline{g_{gh}})}
&
\\
\\
&
  A^\sube_0
    \arrow{rdd}[swap]{\overline{\psi_1}}
&
&
& A^\sube_0
    \arrow{dd}[swap]{\overline{\psi'_2}}
    \arrow{lll}{\overline{g'_{g, gh}}}
& {\color{mygray}A^\sube_0}
    \arrow[color=mygray]{ldd}[color=mygray]{\overline{\psi_2}}
    \arrow[dashed]{l}{\ell(\overline{g_{g, gh}})}
&
&
\\
\\
&
& A^\sube_0
&
& A^\sube_0
	\arrow{ll}{\overline{g_{\sube}}}
&
&
&
\\
\\
  A^{\sube}_{gh}
    \arrow{rdd}{P_1}
    \arrow[dashed]{uuuuuu}[pos=0.13]{\phi_{gh}}
&
&
&
&
&
& A^{\sube}_{hg}
    \arrow{ldd}[swap]{P_2}
    \arrow{llllll}{g}
    \arrow[dashed]{luuuuuu}[pos=0.17]{\phi'_{hg}}
    \arrow[dashed, color=mygray]{uuuuuu}[swap, pos=0.13, color=mygray]{\phi_{hg}}
&
\\
\\
&
  A^{\sube}_{g, gh}
    \arrow{rdd}{\pi_{g}}
    \arrow[dashed]{uuuuuu}{\phi_{g, gh}}
&
&
&
& A^{\sube}_{hg, g}
    \arrow{ldd}[swap]{\pi_{g}}
    \arrow[dashed]{luuuuuu}[swap, pos=0.45]{\phi'_{hg, g}}
    \arrow[dashed, color=mygray]{uuuuuu}[swap, color=mygray]{\phi_{hg, g}}
&
&
\\
\\
&
& A^{\sube}_{g}
    \arrow[dashed]{uuuuuu}[pos=0.86]{\phi_{g}}
&
& A^{\sube}_{g}
	\arrow{ll}{g}
    \arrow[dashed]{uuuuuu}[pos=0.86]{\phi_{g}}
&
&
&
\end{tikzcd}
\]
\caption{}
\label{fig:largediagcomm2}
\end{figure}

We may write:
\begin{equation}
\overline{g'_{gh}} = (\overline{P_1}^{-1} \circ \overline{g'_{g, gh}} \circ \overline{P_1}) \circ (\overline{P_1}^{-1} \circ \overline{P'_2}).
\end{equation}
Then, since $\overline{g'_{gh}}$ and $\overline{g'_{g, gh}}$ are translations, $\overline{P_1}^{-1} \circ \overline{P'_2}$ is also a translation. By Lemma~\ref{C-bounded} (applied to Diagram~\ref{fig:largediagcomm2}), it is the composition of two $K(C)$-bounded maps, hence $K(C)$-bounded. Thus we have
\begin{equation}
\overline{g'_{gh}} \approx_C \overline{P_1}^{-1} \circ \overline{g'_{g, gh}} \circ \overline{P_1}.
\end{equation}
Since $\ell(\overline{g_{gh}})$, $\ell(\overline{g_{g, gh}})$, $\overline{\psi_1}$ and $\overline{\psi_2}$ are $K(C)$-bounded quasi-translations, $\overline{g_{gh}}$ is $K(C)$-almost equivalent to $\overline{g'_{gh}}$ and $\overline{g_{\sube}}$ is $K(C)$-almost equivalent to $\overline{g'_{g, gh}}$. It remains to check that the map $\overline{g'_{g, gh}}$ is $K(C)$-almost equivalent to its conjugate $\overline{P_1}^{-1} \circ \overline{g'_{g, gh}} \circ \overline{P_1}$.

This follows from Lemma~\ref{close_to_identity} below. Indeed, let $\overline{P''_1}$ be the quasi-translation constructed in Lemma~\ref{close_to_identity}. Let $v \in V^\sube_0$ be the translation vector of $\overline{g'_{g, gh}}$, so that
\begin{equation}
\overline{g'_{g, gh}} =: \tau_v.
\end{equation}
Keep in mind that while we call the map $\tau_v$ a ``translation'', it is formally a transvection: its matrix in a suitable basis is $\bigl( \begin{smallmatrix} \Id & v \\ 0 & 1 \end{smallmatrix} \bigr)$. Then we have
\begin{align}
\left\| \overline{P_1}^{-1} \circ \overline{g'_{g, gh}} \circ \overline{P_1}
 - \overline{P''_1}^{-1} \circ \overline{g'_{g, gh}} \circ \overline{P''_1} \right\|
&= \left\| \tau_{\overline{P_1}^{-1}(v)} - \tau_{\overline{P''_1}^{-1}(v)} \right\| \nonumber \\
&= \left\| \overline{P_1}^{-1}(v) - \overline{P''_1}^{-1}(v) \right\| \nonumber \\
&\leq \left\| \restr{(\overline{P_1}^{-1} - \overline{P''_1}^{-1})}{V^\sube_0} \right\| \| v \|
\end{align}
(as $v \in V^\sube_0$).

\begin{remark}
\label{transvection_screwup}
While the corresponding calculation in~\cite{Smi14} does not technically contain any explicit falsehoods (the inequality just happens to be slightly weaker than what it should be), it implicitly relies on the false ``identity'' $\tau_u - \tau_v = \tau_{u-v}$. Here we have corrected this confusion.
\end{remark}

Now by Lemma~\ref{quasi-translation}, we know that the quasi-translation $\overline{P''_1}$ restricted to $V^\sube_0$ is a linear map preserving the Euclidean norm. We also know that the map $\rho \mapsto \rho^{-1}$ (defined on~$\GL(V^\sube_0)$) is Lipschitz-continuous on a neighborhood of the orthogonal group (which is compact). Finally, by Lemma~\ref{affine_to_vector}, $s(\ell(g))$ does not exceed $s(g)$ which is by hypothesis smaller than or equal to $s_{\ref{invariant_additivity}}(C)$. Taking $s_{\ref{invariant_additivity}}(C)$ small enough, we may deduce from Lemma~\ref{close_to_identity} that
\begin{equation}
\left \| \restr{(\overline{P_1}^{-1} - \overline{P''_1}^{-1})}{V^\sube_0} \right \| \lesssim_C s(\ell(g)).
\end{equation}
On the other hand, we have $\|v\| \leq \|\tau_v\| = \left\| \overline{g'_{g, gh}} \right\| \lesssim_C \left\| \restr{g}{A^\sube_g} \right\|$, since $\overline{g'_{g, gh}}$ is the composition of $\restr{g}{A^\sube_g}$ with several $K(C)$-bounded maps. It follows that
\begin{equation}
\left\| \overline{P_1}^{-1} \circ \overline{g'_{g, gh}} \circ \overline{P_1}
 - \overline{P''_1}^{-1} \circ \overline{g'_{g, gh}} \circ \overline{P''_1} \right\|
\lesssim_C s(\ell(g)) \left\|\restr{g}{A^\sube_g}\right\|.
\end{equation}
By Lemma~\ref{affine_to_vector} (iii), we have $s(\ell(g)) \left\|\restr{g}{A^\sube_g}\right\| \lesssim_C s(g)$; and we know that $s(g) \leq 1$. Finally we get
\begin{equation}
\left\| \overline{P_1}^{-1} \circ \overline{g'_{g, gh}} \circ \overline{P_1}
 - \overline{P''_1}^{-1} \circ \overline{g'_{g, gh}} \circ \overline{P''_1} \right\|
\lesssim_C 1.
\end{equation}
To complete the proof of Lemma~\ref{gghg=}, and hence also the proof of Proposition~\ref{invariant_additivity}, it remains only to prove Lemma~\ref{close_to_identity}.
\end{proof}

\begin{lemma}
\label{close_to_identity}
The linear part of the map $\overline{P_1}$ is ``almost'' a quasi-translation. More precisely, there is a quasi-translation $\overline{P''_1}$ such that
\[\left \| \restr{(\overline{P_1} - \overline{P''_1})}{V^\sube_0} \right \| \lesssim_C s(\ell(g)).\]
\end{lemma}
Recall that $\ell(g)$ is the map with the same linear part as $g$, but with no translation part: see subsection~\ref{sec:affine_to_linear}. We use the double prime because the relationship between $\overline{P''_1}$ and $\overline{P_1}$ is not the same as the relationship between $\overline{P'_2}$ and $\overline{P_2}$.

\begin{proof}
The proof is exactly the same as the proof of Lemma~4.7 in~\cite{Smi14}, \emph{mutatis mutandis}.
\end{proof}

\section{Margulis invariants of words}
\label{sec:induction}

We have already studied how contraction strengths (Proposition~\ref{regular_product}) and Margulis invariants (Proposition~\ref{invariant_additivity}) behave when we take the product of two $C$-non-degenerate, sufficiently contracting maps of type~$X_0$. The goal of this section is to generalize these results to words of arbitrary length on a given set of generators. It is a straightforward generalization of Section~5 in~\cite{Smi14} (we slightly changed the notations).

\begin{definition}
\label{cyclically_reduced_definition}
Take $k$ generators $g_1, \ldots, g_k$. Consider a word $g = g_{i_1}^{\sigma_1} \cdots g_{i_l}^{\sigma_l}$ with length~$l \geq 1$ on these generators and their inverses (for every $m$ we have $1 \leq i_m \leq k$ and~$\sigma_m = \pm 1$). We say that $g$ is \emph{reduced} if for every $m$ such that $1 \leq m \leq l-1$, we have $(i_{m+1}, \sigma_{m+1}) \neq (i_m, -\sigma_m)$. We say that $g$ is \emph{cyclically reduced} if it is reduced and also satisfies $(i_1, \sigma_1) \neq (i_l, -\sigma_l)$.
\end{definition}

\begin{proposition}
\label{Schottky_group}
For every $C \geq 1$, there is a positive constant $s_{\ref{Schottky_group}}(C) \leq 1$ with the following property. Take any family of maps $g_1, \ldots, g_k \in G \ltimes V$ satisfying the following hypotheses:
\begin{enumerate}[label=\textnormal{(H\arabic*)}]
\item \label{itm:all_are_X0} Every $g_i$ is of type~$X_0$.
\item \label{itm:pairwise_C_non_deg} Any pair taken among the maps $\{g_1, \ldots, g_k, g_1^{-1}, \ldots, g_k^{-1}\}$ is $C$-non-degener\-ate, except of course if it has the form $(g_i, g_i^{-1})$ for some $i$.
\item \label{itm:all_are_contracting} For every $i$, we have $s(g_i) \leq s_{\ref{Schottky_group}}(C)$ and $s(g_i^{-1}) \leq s_{\ref{Schottky_group}}(C)$.
\end{enumerate}
Take any nonempty cyclically reduced word $g = g_{i_1}^{\sigma_1} \cdots g_{i_l}^{\sigma_l}$ (with $1 \leq i_m \leq k$, $\sigma_m = \pm 1$ for every $m$). Then $g$ is of type~$X_0$, $2C$-non-degenerate, and we have
\[\left\| M(g) - \sum_{m=1}^{l} M(g_{i_m}^{\sigma_m}) \right\| \leq l k_{\ref{invariant_additivity}}(2C)\]
(where $k_{\ref{invariant_additivity}}(2C)$ is the constant introduced in Proposition~\ref{invariant_additivity}).
\end{proposition}

The proof proceeds by induction, with Proposition~\ref{regular_product} and Proposition~\ref{invariant_additivity} providing the induction step. However, there is a subtlety (already dealt with in~\cite{Smi14}). When we suppose that the pair $(g, h)$ is $C$-non-degenerate, we can only conclude that $gh$ is $2C$-non-degenerate; this would break the induction if we used a direct approach. To guarantee $2C$-non-degeneracy for all words, we must use the fact that the contraction strength of $g$ grows (technically the number $s(g)$ diminishes) exponentially with its length, so that the (Hausdorff) distance between $A^{\subge}_{g}$ and $A^{\subge}_{g_{i_1}^{\sigma_1}}$ is in fact a sum of exponentially diminishing increments and remains bounded. To take this into account, we prove by induction a series of slightly more complicated statements.

\begin{proof}
The proof is exactly the same as proof of Proposition~5.2 in~\cite{Smi14}, \emph{mutatis mutandis}.

Given the importance of this point, let us briefly recap the strategy of this proof. Let us fix $C \geq 1$, a positive constant $s_{\ref{Schottky_group}}(C) \leq 1$ to be determined in the course of the proof, and a family $g_1, \ldots, g_k$ satisfying the hypotheses \ref{itm:all_are_X0}, \ref{itm:pairwise_C_non_deg} and \ref{itm:all_are_contracting}. We show by induction on $l$ that whenever we take a nonempty cyclically reduced word $g = g_{i_1}^{\sigma_1} \cdots g_{i_l}^{\sigma_l}$, we have the following properties:
\begin{hypothenum}
\item The map $g$ is of type~$X_0$.
\item $\begin{cases}
       \alpha^\mathrm{Haus} \left(A^{\subge}_{g},\;
                                  A^{\subge}_{g_{i_1}^{\sigma_1}} \right)
          \lesssim_C 2 \left( 1 - 2^{-(l-1)} \right) s_{\ref{Schottky_group}}(C) \vspace{1mm} \\
       \alpha^\mathrm{Haus} \left(A^{\suble}_{g},\;
                                  A^{\suble}_{g_{i_l}^{\sigma_l}} \right)
          \lesssim_C 2 \left( 1 - 2^{-(l-1)} \right) s_{\ref{Schottky_group}}(C).
       \end{cases}$
\item $s(g) \leq 2^{-(l-1)} s_{\ref{Schottky_group}}(C)$.
\item $\displaystyle \left\| M(g) - \sum_{m=1}^{l} M(g_{i_m}^{\sigma_m}) \right\| \leq (l-1)k_{\ref{invariant_additivity}}(2C)$.
\item If $h = g_{i'_1}^{\sigma'_1} \cdots g_{i'_{l'}}^{\sigma'_{l'}}$ is another nonempty cyclically reduced word of length $l' \leq l$ such that $gh$ (or equivalently $hg$) is still cyclically reduced, the pair $(g, h)$ is $2C$-non-degenerate.
\end{hypothenum}
The proposition then follows from the properties (i), (iv) and (v). For the actual proof of these five statements, we refer the reader to the proof of Proposition~5.2 in~\cite{Smi14}.
\end{proof}

\section{Construction of the group}
\label{sec:construction}

Here we prove the Main Theorem. We closely follow Section~6 from~\cite{Smi14}, with only two substantial differences:
\begin{itemize}
\item While in the case of the adjoint representation, existence of a $-w_0$-invariant vector in~$V^\transl_0$ was automatic, here we must postulate it explicitly (Assumption~\ref{inverse_ok}).
\item Where we originally relied on Lemma~7.2 in~\cite{Ben96}, we now need the more general Lemma~4.3.a in~\cite{Ben97}.
\end{itemize}
In the next-to-last paragraph of the proof, we have also made more explicit the relationship between $s_{\textnormal{Main}}(C)$ and~$s_{\ref{Schottky_group}}(C)$.

Let us recall the outline of the proof. We begin by showing (Lemma~\ref{properly_discontinuous}) that if we take a group generated by a family of $C$-non-degenerate, sufficiently contracting maps of type~$X_0$ with suitable Margulis invariants, it satisfies all of the conclusions of the Main Theorem, except Zariski-density. We then exhibit such a group that is also Zariski-dense (and thus prove the Main Theorem).

The idea is to ensure that the Margulis invariants of all elements of the group remain close to some half-line. Obviously if $-w_0$ maps every element of~$V^\transl_0$ to its opposite, Proposition~\ref{invariant_additivity}~(i) makes this impossible. So we now exclude this case:
\begin{assumption}
\label{inverse_ok}
The representation~$\rho$ is such that the action of~$w_0$ on~$V^\transl_0$ is not trivial.
\end{assumption}
This is precisely condition~\ref{itm:main_condition} from the Main Theorem. More precisely, $V^\transl_0$ is the set of all vectors that satisfy~\ref{itm:main_condition}\ref{itm:fixed_by_l}, and what we say here is that some of them also satisfy~\ref{itm:main_condition}\ref{itm:not_fixed_by_w0}.

~
\begin{example}~
\label{inverse_ok_example}
\begin{enumerate}
\item Consider $G = \SO^+(p, q)$ acting on~$\mathbb{R}^{p+q}$ (with~$p \geq q$); we have already seen that the only case when $V^\transl_0 \neq 0$ is when $p - q = 1$ (see Example~\ref{transl_space_example}.1). So let $p = n+1$, $q=n$; then we may show that
\begin{equation}
\restr{w_0}{V^\transl_0} = (-1)^n \Id
\end{equation}
(this is essentially the content of Lemma~3.1 in~\cite{AMS02} or of Proposition~2.7 in~\cite{Smi13}). So $G = \SO^+(n+1, n)$ satisfies this assumption if and only if $n$~is odd.
\item If $G$ is any semisimple real Lie group acting on~$V = \mathfrak{g}$ (its Lie algebra) by the adjoint representation, then $\mathfrak{g}^\transl_0$~contains the Cartan subspace~$\mathfrak{a}$ (see Example~\ref{transl_space_example}.2), on which $w_0$~obviously acts nontrivially unless $\mathfrak{a}$ is itself trivial. So this assumption is satisfied whenever $G$~is noncompact.
\end{enumerate}
\end{example}

Thanks to Assumption~\ref{inverse_ok}, we choose once and for all some nonzero vector $M_C \in V^\transl_0$ that is a fixed point of $-w_0$ (which is possible since $w_0$~is an involution). This requirement still leaves us free to prescribe the norm of this vector; let us additionally assume that $\|M_C\| = 2k_{\ref{invariant_additivity}}(2C)$.
\begin{lemma}
\label{properly_discontinuous}
Take any family $g_1, \ldots, g_k \in G \ltimes V$ satisfying the hypotheses \ref{itm:all_are_X0},~\ref{itm:pairwise_C_non_deg} and~\ref{itm:all_are_contracting} from Proposition~\ref{Schottky_group}, and also the additional condition
\begin{enumerate}[label=\textnormal{(H\arabic*)},start=4]
\item \label{itm:prescribed_marg_inv} For every $i$, $M(g_i) = M_C$.
\end{enumerate}
Then the group generated by $g_1, \ldots, g_k$ is free (with $g_1, \ldots, g_k$ being a basis) and
acts properly discontinuously on the affine space~$V_{\Aff}$.
\end{lemma}
\begin{proof}
The proof is exactly the same as the proof of Lemma~6.1 in~\cite{Smi14}, \emph{mutatis mutandis}.

The (orthogonal) projection
\[\hat{\pi}_{\mathfrak{z}}: \hat{\mathfrak{g}} \to \mathfrak{z} \oplus \mathbb{R}^0
\quad\text{ parallel to } \mathfrak{d} \oplus \mathfrak{n}^+ \oplus \mathfrak{n}^-\]
now becomes the (orthogonal) projection
\[\hat{\pi}_{\transl}: A \to V^\transl_0 \oplus \mathbb{R}^0
\quad\text{ parallel to } V^\rotat_0 \oplus V^\subg_0 \oplus V^\subl_0.\]

(Let us just explicitly restate the proof that the group is free, as it is very short. The group is free simply because any nonempty reduced word on the $g_i^{\pm 1}$ is conjugate to some cyclically reduced word, which, by Proposition~\ref{Schottky_group}, is of type~$X_0$ and in particular different from the identity.)
\end{proof}

\begin{proof}[Proof of Main Theorem]
First note that all the assumptions we have made on~$\rho$ in the course of the paper were legitimate, in the sense that they follow from the hypotheses of the Main Theorem. Indeed:
\begin{itemize}
\item Assumption~\ref{inverse_ok} is just the condition~\ref{itm:main_condition};
\item Assumption~\ref{transl_space} is the weaker condition~\ref{itm:main_condition}\ref{itm:fixed_by_l};
\item Assumption~\ref{zero_is_a_weight} is an even weaker condition that follows from~\ref{itm:main_condition}\ref{itm:fixed_by_l} (see Remark~\ref{zero_weight_follows_from_hypotheses}); and
\item Assumption~\ref{no_swinging} is just the condition~\ref{itm:technical_condition}.
\end{itemize}

Once again, we use the same strategy as in the proof of the Main Theorem of~\cite{Smi14}. We find a positive constant $C \geq 1$ and a family of maps $g_1, \ldots, g_k$ in~$G \ltimes V$ (with $k \geq 2$) that satisfy the conditions \ref{itm:all_are_X0} through \ref{itm:prescribed_marg_inv} and whose linear parts generate a Zariski-dense subgroup of $G$, then we apply Lemma~\ref{properly_discontinuous}. We proceed in several stages.
\begin{itemize}
\item We begin by using a result of Benoist: we apply Lemma~4.3.a in~\cite{Ben97} to
\begin{itemize}
\item $\Gamma = G$;
\item $t = k+1$;
\item $\Omega_1 = \cdots = \Omega_k = \mathfrak{a}_{\rho, X_0} \cap \mathfrak{a}^{++}$.
\end{itemize}
This gives us, for any $k \geq 2$, a family of maps $\gamma_1, \ldots, \gamma_k \in G$ (which we shall see as elements of~$G \ltimes V$, by identifying~$G$ with the stabiliser of~$p_0$), such that:
\begin{hypothenum}
\item Every $\gamma_i$ is of type~$X_0$ (this is \ref{itm:all_are_X0}).
\item For any two indices $i$, $i'$ and signs $\sigma$, $\sigma'$ such that $(i', \sigma') \neq (i, -\sigma)$, the spaces $V^{\subge}_{\gamma_i^\sigma}$ and $V^{\suble}_{\gamma_{i'}^{\sigma'}}$ are transverse.
\item Any single $\gamma_i$ generates a Zariski-connected group.
\item All of the $\gamma_i$ generate together a Zariski-dense subgroup of $G$.
\end{hypothenum}

A comment about item~(i): we actually get not only that every $\gamma_i$~is of type~$X_0$, but also that every~$\gamma_i$ is $\mathbb{R}$-regular.

A comment about item~(ii): since we have taken Benoist's $\Gamma$ to be the whole group~$G$, we have $\theta = \Pi$, so that $Y_\Gamma$ is the complete flag variety~$G/P^+$. Benoist's conclusion can then be restated by saying that the pair of cosets
\[(\phi_g P^+,\; \phi_h P^-)\]
(where $\phi_g$ and~$\phi_h$ are respective canonizing maps of $g$ and~$h$ as defined by us in Proposition~\ref{Asubge_is_ampa}) is in the open $G$-orbit of~$G/P^+ \times G/P^-$. Once again it is actually stronger than our conclusion, which is equivalent to saying that the pair of cosets
\[(\phi_g P^+_{X_0},\; \phi_h P^-_{X_0})\]
is in the open $G$-orbit of~$G/P^+_{X_0} \times G/P^-_{X_0}$.

\item Clearly every pair of transverse spaces is $C$-non-degenerate for some finite~$C$; and here we have a finite number of such pairs. Hence if we choose some suitable value of $C$ (which we fix for the rest of this proof), the hypothesis~\ref{itm:pairwise_C_non_deg} becomes a direct consequence of the condition (ii) above.
\item From condition (iii) (Zariski-connectedness), it follows that any algebraic group containing some power $\gamma_i^N$ of some generator must actually contain the generator~$\gamma_i$ itself. This allows us to replace every $\gamma_i$ by some power $\gamma_i^N$ without sacrificing condition (iv) (Zariski-density). Clearly, conditions (i), (ii) and (iii) are then preserved as well. If we choose $N$ large enough, we may suppose (thanks to Remark~\ref{contraction_strength_grows}) that the numbers $s(\gamma_i^{\pm 1})$ are as small as we wish: this gives us \ref{itm:all_are_contracting}. In fact, we shall suppose that for every $i$, we have $s(\gamma_i^{\pm 1}) \leq s_{\textnormal{Main}}(C)$ for an even smaller constant $s_{\textnormal{Main}}(C)$, to be specified soon.
\item To satisfy \ref{itm:prescribed_marg_inv}, we replace the maps $\gamma_i$ by the maps
\begin{equation}
g_i := \tau_{\phi_i^{-1}(M_C)} \circ \gamma_i
\end{equation}
(for $1 \leq i \leq k$), where $\phi_i$ is a canonizing map for $\gamma_i$.

We need to check that this does not break the first three conditions. Indeed, for every $i$, we have $\gamma_i = \ell(g_i)$; even better, since the translation vector $\phi_i^{-1}(M_C)$ lies in the subspace $V^{\sube}_{\gamma_i}$ stable by $\gamma_i$, obviously the translation commutes with $\gamma_i$, hence $g_i$ has the same geometry as $\gamma_i$ (meaning that $A^{\subge}_{g_i} = A^{\subge}_{\gamma_i} = V^{\subge}_{\gamma_i} \oplus \mathbb{R} p_0$ and $A^{\suble}_{g_i} = A^{\suble}_{\gamma_i} = V^{\suble}_{\gamma_i} \oplus \mathbb{R} p_0$). Hence the $g_i$ still satisfy the hypotheses \ref{itm:all_are_X0} and \ref{itm:pairwise_C_non_deg}, but now we have $M(g_i) = M_C$ (this is \ref{itm:prescribed_marg_inv}). As for contraction strength, we have, by Lemma~\ref{affine_to_vector}:
\begin{align}
s(g_i) &\lesssim_C s(\gamma_i)\|\tau_{M_C}\| \nonumber \\
       &\leq\;\; s_{\textnormal{Main}}(C)\|\tau_{M_C}\|,
\end{align}
and similarly for $g_i^{-1}$. Recall that $\|M_C\| = 2k_{\ref{invariant_additivity}}(2C)$, hence $\|\tau_{M_C}\|$ depends only on $C$: in fact it is equal to the norm of the 2-by-2 matrix $\bigl( \begin{smallmatrix} 1 & \|M_C\| \\ 0 & 1 \end{smallmatrix} \bigr)$. It follows that if we choose
\begin{equation}
s_{\textnormal{Main}}(C) \leq s_{\ref{Schottky_group}}(C) \left\| \begin{matrix} 1 & 2k_{\ref{invariant_additivity}}(2C) \\ 0 & 1 \end{matrix} \right\|^{-1},
\end{equation}
then the hypothesis \ref{itm:all_are_contracting} is satisfied.

We conclude that the group generated by the elements $g_1, \ldots, g_k$ acts properly discontinuously (by Lemma~\ref{properly_discontinuous}), is free (by the same result), nonabelian (since $k \geq 2$), and has linear part Zariski-dense in $G$. \qedhere
\end{itemize}
\end{proof}

\bibliographystyle{alpha}
\bibliography{/home/ilia/Documents/Travaux_mathematiques/mybibliography.bib}

\noindent I. Smilga, Mathematics Department, Yale University, P.O. Box 208283, New Haven, CT 06520-8283, U.S.A. 

\noindent E-mail: \url{ilia.smilga@normalesup.org}
\end{document}